\documentclass[10pt]{article}
\usepackage{geometry}
\usepackage[T1]{fontenc}
\usepackage{charter}
\usepackage{amsmath,amssymb,amsfonts,amsopn,amsthm}
\usepackage{mathtools}
\usepackage{dsfont}
\usepackage{bm}
\usepackage{bbm}
\usepackage{nicefrac}
\usepackage{mathrsfs}
\usepackage{graphicx}
\usepackage{epstopdf}
\usepackage{caption}
\usepackage{subcaption}         
\usepackage{overpic}
\usepackage{float}
\usepackage{chngcntr}
\counterwithout{figure}{section}
\usepackage[section]{placeins}
\usepackage{tikz}
\usepackage{pgfplots}
\pgfplotsset{compat=1.17}
\usepackage{tikz-3dplot}
\usetikzlibrary{%
	patterns, shapes, positioning, shadings, calc,
	decorations.pathreplacing, decorations.pathmorphing,
	angles, quotes, backgrounds, intersections, plotmarks}
\tikzset{snake it/.style={decorate, decoration=snake}}
\usepgfplotslibrary{fillbetween}
\usepackage{etoolbox}
\apptocmd{\thebibliography}{%
	\setlength{\itemsep}{-1pt}%
	\setlength{\parsep}{0pt}%
	\setlength{\parskip}{0pt}%
}{}{}
\usepackage{titlesec}
\titleformat{\section}
{\large\bfseries}{\thesection.}{0.6em}{}
\titleformat{\subsection}
{\normalsize\bfseries}{\thesubsection.}{0.5em}{}
\titleformat{\subsubsection}
{\normalsize\itshape}{\thesubsubsection.}{0.5em}{}
\titlespacing*{\section}{0pt}{1.8ex plus .5ex minus .2ex}{0.8ex plus .2ex}
\titlespacing*{\subsection}{0pt}{0.2ex plus .1ex minus .1ex}{0.3ex plus .1ex}
\usepackage[numbers,sort&compress]{natbib}

\usepackage{algorithm}
\usepackage{algpseudocode}
\algrenewcommand\algorithmicrequire{\textbf{Input:}}
\algrenewcommand\algorithmicensure{\textbf{Output:}}
\makeatletter
\algnewcommand{\Why}[1]{\Statex\hskip\ALG@thistlm\(\triangleright\)\ \textit{#1}}
\makeatother

\usepackage{hyperref}
\hypersetup{
	colorlinks = true,
	linkcolor  = blue,
	citecolor  = red,
	urlcolor   = red!20,
	filecolor  = red!20,
}
\usepackage{enumitem}
\setlist{
	topsep=2pt,
	itemsep=2pt,
	parsep=0pt,
	rightmargin=0pt
}
\usepackage{longtable}
\usepackage{supertabular}
\usepackage{siunitx}
\usepackage{textcomp}
\usepackage{indentfirst}
\usepackage{url}
\usepackage{ifthen}
\usepackage[normalem]{ulem}     
\usepackage{listings}
\usepackage{authblk}

\definecolor{navyblue}{rgb}{0.0,0.0,0.5}
\definecolor{skyblue}{RGB}{135,206,235}
\definecolor{deepskyblue}{RGB}{0,191,255}

\numberwithin{equation}{section}
\newtheorem{theorem}{Theorem}[section]
\newtheorem{lemma}[theorem]{Lemma}
\newtheorem{corollary}[theorem]{Corollary}

\newtheorem{definition}[theorem]{Definition}

\newtheorem{remark}[theorem]{Remark}

\def\md{{\mathrm{d}}}

\def\O{\Omega}

\def\e{\varepsilon}

\def\fU{{\mathfrak{U}}}

\def\fC{\mathfrak C}

\def\fW{\mathfrak{W}}

\def\Dc{\mathcal{D}}

\def\N{{\mathbb{N}}}

\def\R{{\mathbb{R}}}

\def\Bc{{\mathcal{B}}}
\def\Ec{{\mathcal{E}}}

\def\Qc{{\mathcal{Q}}}
\def\Sc{{\mathcal{S}}}

\def\Kc{{\mathcal{K}}}

\def\Rc{{\mathcal{R}}}

\def\Jc{{\mathcal{J}}}

\def\wt{\widetilde}
\def\X{\times}
\def\div{{\mathrm{div}}}
\def\hat{\widehat}

\def\th{{\boldsymbol{\theta}}}

\def\bvp{\boldsymbol{\varphi}}
\def\Oc{\mathcal{O}}

\def\Gn{{\bf n}}
\def\Gu{{\bf u}}

\def\GH{{\bf H}}\def\GJ{{\bf J}}

\def\bv{\mathbf{v}}

\def\bw{\mathbf{w}}

\def\bn{\mathbf{n}}

\def\t{\boldsymbol{\theta}^\e}

\def\Ac{\mathcal{A}}
\newcommand{\Rey}{\mathrm{Re}}

\newcommand{\ocarrow}{\overset{\text{\rm oc}}{\rightharpoonup}}

\newcommand{\m}{\mathbf m}
\def\fX{\mathfrak{X}}

\DeclareCaptionType{scheme}[Scheme][List of Schemes]

\begin{document}
	%
	
	\title{Homogenization of an optimal control problem for a Navier–Stokes–Cahn–Hilliard system in a 3D perforated domain}

	\author{%
	Amartya Chakrabortty\thanks{
		Department of Mathematics, IIT Kharagpur,  721302, Kharagpur
		India; \texttt{amartyat26@kgpian.iitkgp.ac.in}},~~ 
	Haradhan Dutta\thanks{Department of Mathematics, IIT Kharagpur, 721302, Kharagpur India; \texttt{hdutta1412@kgpian.iitkgp.ac.in}},~~
	Hari Shankar Mahato\thanks{Department of Mathematics, IIT Kharagpur, 721302, Kharagpur India; \texttt{hsmahato@maths.iitkgp.ac.in}}%
}

	
	\date{} 
	
	\newcommand{\keywords}[1]{\par\smallskip\noindent\textbf{Keywords: }#1}
	\newcommand{\subjclass}[1]{\par\noindent\textbf{MSC 2020: }#1}
	
	\maketitle
	
	\vspace{-1.0cm}
\begin{abstract}
	We study the homogenization of an optimal control problem governed by a
	Navier--Stokes--Cahn--Hilliard system in a periodically perforated
	three-dimensional domain. The obstacles are Lipschitz sets of size
	$\e^\alpha$, $\alpha>3$, where $\e$ is the period. The fluid satisfies the
	no-slip condition on the outer boundary and the free-slip condition on the
	obstacles, the phase-field equation contains a source term, and the
	convective term carries a pore-scale Reynolds number that tends to zero as
	$\e\to0$. The control is a distributed body force, and the cost functional
	tracks the velocity in $L^2$ and the order parameter in $H^1$. Since
	uniqueness of weak solutions of the state system is not known for fixed
	$\e$, the control-to-state relation is multivalued, and we pose the
	microscopic problems over control--state pairs. We prove the existence of
	microscopic optimal pairs and the sequential $\Gamma$-convergence of the
	cost functionals to the cost functional of a control problem governed by the
	unsteady Stokes--Cahn--Hilliard system on the full domain. As a
	consequence, the optimal values converge, every cluster point of microscopic
	optimal pairs is optimal for the homogenized problem, and the optimal
	controls converge strongly along the corresponding subsequences. The
	$\liminf$ inequality follows from the homogenization of the state system,
	including the strong convergence of the gradients of the order parameter.
	Only the recovery sequence requires uniqueness of weak solutions of the
	homogenized system, which we prove when the outer boundary is of class
	$C^{2,1}$.
\end{abstract}

\\



\section{Introduction}\label{Sec01}

This paper studies the homogenization of an optimal control problem (OCP) whose
state system is a Navier--Stokes--Cahn--Hilliard (NSCH) system in a
periodically perforated three-dimensional domain. For fixed $\e$ the state
system is of three-dimensional Navier--Stokes type, and uniqueness of its weak
solutions is not known. A control may therefore correspond to several states,
so the control-to-state relation is multivalued and there is no reduced cost
functional. Linearizing a multivalued relation and deriving an adjoint system
and a first-order optimality system from it is non-trivial, and to the best of
our knowledge no such system is available for the present problem. We therefore
pass to the limit $\e\to0$ in the cost functionals themselves rather than in an
optimality system. We pose the microscopic problems over control--state pairs,
encode the state equation as the indicator of the admissible set, and prove the
sequential $\Gamma$-convergence of the resulting functionals, in the framework
of \cite{buttazzo1982gamma,denkowski1993asymptotic}; see also
\cite{DalMaso1993}. The limit state system is a Stokes--Cahn--Hilliard (SCH)
system on the full domain, for which we prove uniqueness.

\smallskip
\textbf{The state system.}
Let $\O\subset\R^3$ be a bounded domain with Lipschitz boundary and let $\e>0$
be the period. Every cell $\e(\kappa+\overline{Y})\subset\O$, $Y=(0,1)^3$, contains one
solid obstacle $\Oc_{\e,\kappa}$, a copy of a fixed reference obstacle $\Oc$ with
Lipschitz boundary, scaled by $a_\e=\e^\alpha$ with $\alpha>3$. The fluid domain
$\O_p^\e$ is $\O$ with the obstacles $\O_s^\e$ removed, and
$\Gamma_s^\e=\partial\O_s^\e$ is the union of their boundaries; the precise
construction is given in Section~\ref{sec:domain}. For a fixed time interval
$S=(0,T)$, the velocity $\Gu^\e$, the pressure $p^\e$, the order parameter
$\phi^\e$ and the chemical potential $\mu^\e$ solve
\begin{equation}\label{MainP01}
	\left\{\begin{aligned}
		& \partial_t \mathbf u^\e
		+ \Rey_\e\,(\mathbf u^\e\!\cdot\nabla)\mathbf u^\e
		- \div\!\big(2\nu D( \mathbf u^\e)\big)
		+ \nabla p^\e
		+ \lambda \phi^\e \nabla \mu^\e
		= \th^\e
		&& \text{in } S\times \Omega_p^\e,\\
		& \nabla\cdot \mathbf u^\e = 0
		&& \text{in } S\times \Omega_p^\e,\\
		& \mathbf u^\e = \mathbf 0\quad \text{on }  S\times \partial\Omega,\qquad
		\mathbf u^\e\cdot\Gn^\e = 0,\quad
		\big(2\nu D(\mathbf u^\e)\mathbf n^\e\big)\!\cdot\boldsymbol\tau^\e = 0
		&& \text{on }  S\times \Gamma^\e_s,\\
		& \partial_t \phi^\e + \mathbf u^\e\!\cdot\nabla \phi^\e
		- \Delta \mu^\e + G(\phi^\e) = 0
		&& \text{in } S\times \Omega_p^\e,\\
		& \mu^\e = -\Delta \phi^\e + F'(\phi^\e)
		&& \text{in } S\times \Omega_p^\e,\\
		& \mathbf n^\e\!\cdot\nabla \phi^\e = 0,\qquad
		\mathbf n^\e\!\cdot\nabla \mu^\e = 0
		&& \text{on } S\times \partial\Omega_p^\e,\\
		& \mathbf u^\e(0,x) = \mathbf u^\e_0(x),
		\qquad
		\phi^\e(0,x) = \phi^\e_0(x)
		&& \text{in } \Omega_p^\e,
	\end{aligned}\right.
\end{equation}
where $\nu,\lambda>0$, $F$ is the double-well potential \eqref{PotentialF}, $G$
is a source term satisfying \eqref{AssG01}, and the body force $\th^\e$ is the
control. The fluid satisfies the no-slip condition on $\partial\O$ and the
free-slip condition on $\Gamma_s^\e$: it does not penetrate the obstacles, and
the tangential stress vanishes there. The order parameter and the chemical
potential satisfy homogeneous Neumann conditions on $\partial\O_p^\e$. Because
of the source term, the mean of the order parameter is not conserved. The system
\eqref{MainP01} is the diffuse-interface description of an incompressible
two-phase flow \cite{anderson1998diffuse}. The convective term carries the
pore-scale Reynolds number $\Rey_\e$, which satisfies
\begin{equation}\label{ReynoldsAssumption}
	0<\Rey_\e\le1,
	\qquad
	\Rey_\e\longrightarrow0\quad\text{as }\e\to0 .
\end{equation}
The factor $\Rey_\e$ is a smallness assumption on the data: rescaling the
velocity and the pressure by $\Rey_\e$ turns \eqref{MainP01} into the NSCH
system driven by data of size $\Rey_\e$. Physically, this scaling describes the
low Reynolds number regime on the viscous diffusion timescale. Inertia is
negligible in the momentum balance, but the flow is not quasi-static, since the
unsteady term is left unscaled. The condition $\Rey_\e\to0$ removes the
convective term from the limit, so the limit system is of Stokes type. The
geometry enters the analysis through two parameters
(Remark~\ref{rem:alpha-role}). The ratio $a_\e/\e=\e^{\alpha-1}$ controls the
obstacle volume $|\O_s^\e|=O(\e^{3(\alpha-1)})$ and every $L^2$-level estimate.
The parameter $\rho_\e=a_\e/\e^3=\e^{\alpha-3}$ governs the cumulative
obstacle-scale viscous energy and tends to zero for $\alpha>3$, so the obstacles
are in the subcritical regime. Weak solutions of \eqref{MainP01} exist for every
$\e$ and every admissible control (Remark~\ref{rem:existence}).

\smallskip
\textbf{The control problem.}
Given a Tikhonov parameter $\beta>0$ and targets $\mathbf u_d$ and $\phi_d$ for
the velocity and the order parameter, we minimize
\[
\frac12\|\Gu^\e-\mathbf u_d\|^2_{L^2(S\X\O^\e_p)}
+\frac12\|\phi^\e-\phi_d\|^2_{L^2(S;H^1(\O^\e_p))}
+\frac{\beta}{2}\|\th^\e\|^2_{L^2(S\X\O^\e_p)}
\]
with respect to the control $\th^\e$, where $(\Gu^\e,\phi^\e,\mu^\e)$ is a weak
solution of \eqref{MainP01}. We minimize over control--state pairs; this is the
OCP \eqref{OPeps} of Section~\ref{SecMain}.

\smallskip
\textbf{Main results.}
We show that the optimal values of \eqref{OPeps} converge to the optimal value
of an OCP \eqref{OPlim} governed by the unsteady SCH system on $\O$, that every
cluster point of microscopic optimal pairs is optimal for \eqref{OPlim}, and
that the optimal controls converge strongly (Theorems~\ref{ThmGamma}
and~\ref{ThmOCPHomo}). The $\liminf$ inequality requires limits of admissible
pairs to be admissible. The recovery sequence requires every admissible pair of
\eqref{OPlim} to be a limit of admissible pairs of \eqref{OPeps}. The first
property is the homogenization of the state system
(Theorem~\ref{ThHomogenized}); it holds for every choice of microscopic weak
solutions and requires only the Lipschitz regularity of $\O$ and of $\Oc$. The
second follows from the uniqueness of weak solutions of the limit system
(Theorem~\ref{ThmUniq}), which we prove when $\partial\O\in C^{2,1}$, and this
is the only place in the paper where uniqueness is used. Because the order
parameter is tracked in $H^1$, the tracking terms converge only if the gradients
of the order parameter converge strongly; we prove this in
Corollary~\ref{CorGradPhiStrong}. Three further difficulties arise in the
proofs. The time derivative of the velocity is bounded only in the dual of an
$\e$-dependent space, so the Aubin--Lions--Simon theorem is not available. We
obtain strong compactness instead from a primal and a zero extension of each
field, together with the pairings of the velocity with fixed solenoidal fields.
These fields do not satisfy the impermeability condition on the obstacles, so
they are not admissible test functions; we construct a solenoidal restriction
operator that corrects them near every obstacle and, since $\alpha>3$,
converges strongly in $H^1$. Finally, the mean of the order parameter is not
conserved, and the uniqueness proof carries it as a separate component of the
energy.

\smallskip
For the Stokes equations in domains with obstacles of the size of the period,
Darcy's law was derived in \cite{Allaire89}. For the Stokes and stationary
Navier--Stokes equations in domains with small holes and no-slip conditions,
the additional zeroth-order term that appears at the critical hole size was
identified in \cite{AllaireTinyHolesI,AllaireTinyHolesII}. The evolutionary
Navier--Stokes system in perforated domains was homogenized in
\cite{lu2023homogenization}, slip conditions on the holes were considered in
\cite{slip1991allairehomogenization,hofer2022_ns_slip}, and uniform estimates
for the Stokes equations on perforated domains are developed in
\cite{LuUniformStokes21}.
For the NSCH system on a fixed domain, existence and well-posedness are studied
in
\cite{boyer1999mathematical,abels2009diffuse,colli2012global,abels2013incompressible,giorgini2020weak},
and uniqueness and regularity for the system with a singular potential were
established in \cite{giorgini2019uniqueness}. Cahn--Hilliard equations with a
nonlinear source term are studied in \cite{miranville2021source}, and an NSCH
system with chemotaxis and mass transport in \cite{giorgini2026chemotaxis}.
Effective SCH equations for periodic porous media were derived in
\cite{schmuck2012_stokesch}, and rigorous homogenization results for
Cahn--Hilliard, SCH and NSCH systems in porous media are obtained in
\cite{Hari01,Hari02,bunoiu2020homogenization,chakrabortty2026quantitative}. In
\cite{schmuck2012_stokesch,Hari01,Hari02,chakrabortty2026quantitative} the
solid inclusions are of the size of the period, and the effective systems
involve cell problems. The existence of weak solutions and the homogenization of
an NSCH system with an oscillating viscosity tensor, free slip and a source term
in a three-dimensional perforated domain are established in
\cite{chakrabortty2025navier} by periodic unfolding. Here we homogenize the
state system with extension and restriction operators, and the homogenization
also gives the strong convergence of the gradients of the order parameter that
the cost requires.

For evolutionary Navier--Stokes flows, including B\'enard-type problems, the
existence of optimal controls was proved in \cite{AbergelTemam1990}, where the
optimal control is not expected to be unique. Distributed optimal control of
NSCH systems on a fixed domain is studied in \cite{HintermullerWegner2014} for a
time-discretized system and in \cite{FrigeriRoccaSprekels2016} for a nonlocal
system in two dimensions. These works derive first-order optimality conditions,
which rely on a single-valued and differentiable control-to-state map.

Homogenization of elliptic OCPs with oscillating coefficients, possibly
appearing both in the state equation and in the cost functional, is treated in
\cite{MR1466916}. In perforated domains, OCPs were homogenized in
\cite{MR1666365,MuthukumarNandakumaran2009}, and the convergence of controls and
cost functionals when the homogenization produces a strange term was studied in
\cite{DiazPodolskiyShaposhnikova2022}. For the Stokes system in perforated
domains, OCPs were homogenized in
\cite{MuthukumarNandakumaran2008,GargSardar2023} for holes of the size of the
period, leading to a Darcy-type law and to a Stokes system on the full domain,
respectively, and in \cite{SaintJeanPaulinZoubairi2002} in a setting in which a
strange term appears. Boundary velocity suboptimal controls of an
incompressible flow in a cylindrically perforated domain were studied in
\cite{DApiceDeMaioKogut2009}. Semilinear OCPs in domains with oscillating
boundaries were studied in \cite{Nanda01,Nanda03,Nanda04,Nanda02}. The
homogenization of a semilinear nonlocal OCP in a high-contrast elastic domain is
studied in \cite{chakrabortty2026homogenization}, using a $\liminf$--$\limsup$
argument inspired by sequential $\Gamma$-convergence. In these works the state
system is uniquely solvable, and the limit is typically taken in the optimality
system, which again needs a single-valued control-to-state map with enough
regularity.

The $\Gamma$-convergence approach to OCPs was introduced in
\cite{buttazzo1982gamma}; see also \cite{DalMaso1993}. The problem is written as
the minimization of an extended-real-valued functional over control--state
pairs, with the state equation encoded in the indicator of the admissible set,
so the state need not be uniquely determined by the control. Denkowski and
Mortola \cite{denkowski1993asymptotic} apply it to parabolic differential
inclusions with linear, $G$-convergent principal parts. The controls are initial
data, the states are selection--trajectory pairs, and the admissible sets are the
graphs of the solution multifunctions. They show that the $\Gamma$-convergence of
the indicators of these graphs is equivalent to the Kuratowski convergence of the
solution sets, prove the upper half of this convergence by compactness and the
lower half under a Lipschitz condition on the right-hand sides in the Hausdorff
metric, and deduce the convergence of minimal values and of optimal solutions.
Kogut and Leugering developed the approach for the homogenization of OCPs,
through the variational convergence of constrained minimization problems; see
\cite{KogutLeugering2011} and the references therein. For state-constrained
semilinear problems with oscillating control regions, they identify in
\cite{NL01} the Kuratowski limit of the admissible control--state sets and the
$\Gamma$-limit of the extended cost functionals. Problems with exact partial
controllability constraints, including a perforated domain, are treated in
\cite{KogutLeugering2008}, and suboptimal boundary controls for elliptic
equations in perforated domains in \cite{DApiceDeMaioKogut2008}.

The pair formulation, the indicator of the admissible set and the conclusions
drawn from sequential $\Gamma$-convergence are within the framework of
\cite{buttazzo1982gamma,denkowski1993asymptotic}. Our main contribution is to
verify the two convergence properties that this framework requires, for a
three-dimensional evolutionary NSCH system in a perforated domain. That limits
of admissible pairs are admissible is the homogenization of the state system,
which holds for every choice of microscopic weak solutions. That every
admissible pair of the limit problem is a limit of admissible pairs is obtained
in \cite{denkowski1993asymptotic} from the Lipschitz condition and Gr\"onwall's
inequality; here it follows from the uniqueness of the limit system, which we
also prove by a Gr\"onwall argument. The setting differs from that of
\cite{denkowski1993asymptotic} in four respects. The state system is nonlinear,
and its non-uniqueness comes from the three-dimensional Navier--Stokes-type
structure rather than from a multivalued right-hand side. The function spaces
depend on $\e$, so no fixed Hilbert triple is available and each field has to be
extended to $\O$. The control is a distributed body force, and the order
parameter is tracked in $H^1$. In
\cite{MuthukumarNandakumaran2008,GargSardar2023,SaintJeanPaulinZoubairi2002}
the state system is a Stokes system, which is linear and uniquely solvable.
Finally, Theorem~\ref{ThmUniq} concerns an SCH system with a source term, whose
order parameter has a non-conserved mean; to our knowledge, this case is not
covered by the existing uniqueness results, in particular by
\cite{giorgini2019uniqueness}.

\smallskip
The paper is organized as follows. General notation and function spaces are
introduced in Section~\ref{Sec02}. In Section~\ref{Sec03}, we describe the
perforated domain, state the assumptions and define weak solutions of
\eqref{MainP01}. The OCPs and the main results are stated in
Section~\ref{SecMain}. Section~\ref{Sec04} contains the extension operators, the
uniform functional inequalities, the solenoidal restriction operator and the
proof of the a priori estimates; Section~\ref{SSec61} the compactness of the
extended solutions; Section~\ref{Sec06} the proofs of the homogenization and
uniqueness results; and Section~\ref{sec:control} the proofs of the results on
the OCPs.

\section{Notation and function spaces}\label{Sec02}

Throughout, $S=(0,T)$ is a fixed time interval, and $\O_p^\e$, $\O_s^\e$,
$\Gamma_s^\e$ and the outward unit normal $\Gn^\e$ on $\partial\O_p^\e$ are
those of Section~\ref{sec:domain}; $\boldsymbol\tau^\e$ denotes any unit
tangent vector on $\Gamma_s^\e$. We abbreviate $\O_T:=S\X\O$ and
$\Qc^\e_p:=S\X\O^\e_p$. For a vector field $\bv$, $D(\bv):=\frac12\big(\nabla\bv
+(\nabla\bv)^{\top}\big)$ is the symmetric gradient.

Let $X\subset\R^3$ be a bounded domain with Lipschitz boundary. We denote by
$L^r(X)$ and $W^{k,r}(X)$ the usual Lebesgue and Sobolev spaces with their
standard norms, and set $H^k(X)=W^{k,2}(X)$. The subspace of $H^1(X)$ with zero
trace on $\partial X$ is $H^1_0(X)$, and the mean-zero subspace of $L^2(X)$ is
$L^2_0(X)$. We write $H^1_{0,\div}(X)^3$ for the divergence-free subspace of
$H^1_0(X)^3$, with dual $H^{-1}_{0,\div}(X)^3:=(H^1_{0,\div}(X)^3)^\ast$, and
set $H^{-1}(X):=(H^1(X))^\ast$; note that this is the dual of $H^1(X)$, not of
$H^1_0(X)$. For duality pairings between a space and its dual we write
$\langle\cdot,\cdot\rangle$, and $(\cdot,\cdot)$ denotes the inner product of
$L^2(X)$ or $L^2(X)^3$. The space of $m$-times continuously differentiable
functions on $\overline X$ is $C^m(\overline X)$; $C^\infty_c(X)$ is the space
of smooth compactly supported functions in $X$, $C^\infty_{c,\div}(X)^3$ its
divergence-free subspace, and $\Dc'(X)$ the space of distributions. For a
Hilbert space $H$, $C_w([0,T];H)$ denotes the space of weakly continuous
functions from $[0,T]$ into $H$. We write $A\Subset B$ both for a compact
embedding and for a compactly contained set.

On the fluid domain we use the velocity spaces adapted to the mixed
no-slip/free-slip conditions,
\begin{equation}\label{eq:set}
	\begin{aligned}
		\GH^1(\O_p^\e)&:=\bigl\{\bv\in H^1(\O_p^\e)^3:
		\bv=0\ \text{on }\partial\O,\ \bv\cdot\Gn^\e=0\ \text{on }\Gamma_s^\e\bigr\},\\
		\GH^1_{\div}(\O_p^\e)&:=\bigl\{\bv\in\GH^1(\O_p^\e):
		\div\,\bv=0\ \text{in }\O_p^\e\bigr\},
	\end{aligned}
\end{equation}
with dual $\GH^{-1}_{\div}(\O_p^\e):=\big(\GH^1_{\div}(\O_p^\e)\big)^\ast$, and
the spaces of solenoidal fields
\[
H^\sigma:=\overline{C^\infty_{c,\div}(\O)^3}^{\,L^2(\O)^3},
\qquad
H_\e^\sigma:=\overline{\GH^1_{\div}(\O_p^\e)}^{\,L^2(\O_p^\e)^3}.
\]

For a field $w$ defined on $\O_p^\e$ we write
\begin{equation}\label{eq:extensions-notation}
	\wt w:=E^\e w\ \ \text{(resp.\ }\Ec^\e w\text{ for vector fields)},
	\qquad
	\overline{w}:=
	\begin{cases}
		w & \text{in }\O_p^\e,\\
		0 & \text{in }\O_s^\e,
	\end{cases}
\end{equation}
for the primal and the zero extension of $w$ to $\O$, the extension operators
$E^\e$ and $\Ec^\e$ being those of Lemmas~\ref{Emain} and~\ref{LemD02}. The
characteristic function of a measurable set $A$ is denoted by $\chi_A$.

For $\psi$ defined on $\Qc^\e_p$ we use the norm
\begin{equation}\label{H1PoreNorm}
	\|\psi\|^2_{L^2(S;H^1(\O^\e_p))}
	=\|\psi\|^2_{L^2(\Qc^\e_p)}+\|\nabla\psi\|^2_{L^2(\Qc^\e_p)} ,
\end{equation}
and analogously on $\O_T$. In estimates we write $L^2(X)$ for $L^2(X)^3$ or
$L^2(X)^{3\times3}$, indicating the full product space only when stating weak or
strong convergence. We use the Einstein summation convention. The symbol $C$
denotes a generic positive constant independent of $\e$, whose value may change
from line to line.
	
	\section{Problem setting}\label{Sec03}
	
\subsection{Domain description}\label{sec:domain}
Let $\O\subset\mathbb{R}^3$ be a bounded domain with Lipschitz boundary and let
$Y=(0,1)^3$ be the unit reference cell. Let $\Oc\subset\mathbb{R}^3$ be a bounded domain with Lipschitz boundary such that $0\in\Oc$ and $\mathbb{R}^3\setminus
\overline{\Oc}$ is connected; $\Oc$ is the centered reference obstacle. Fix a
placement point $y_0\in Y$ and set
$d_0:=\operatorname{dist}(y_0,\partial Y)>0$ and
$R_{\Oc}:=\sup_{z\in\Oc}|z|<\infty$.

Fix $\alpha>3$. The physical size of the obstacles and their size relative to the
periodicity scale are $a_\e:=\e^{\alpha}$ and
$\delta_\e:=a_\e/\e=\e^{\alpha-1}$; since $\alpha>1$, we have
$\delta_\e\to0$ as $\e\to0$. For $0<\e<1$ set
$\Kc_\e:=\{\kappa\in\mathbb Z^3:\e(\kappa+\overline Y)\subset\O\}$. Since $\O$ is
bounded, $\#\Kc_\e<\infty$ for every $\e$, and $\#\Kc_\e=O(\e^{-3})$.

For $\kappa\in\Kc_\e$ let $x_{\e,\kappa}:=\e(\kappa+y_0)$ be the obstacle center
in the cell $\e(\kappa+Y)$ and let $\Oc_{\e,\kappa}:=x_{\e,\kappa}+a_\e\Oc$ be the
corresponding solid obstacle. Each obstacle has diameter of order $a_\e$,
whereas neighboring centers are at distance of order $\e$; the ratio of these two
lengths is $\delta_\e\to0$. The solid and fluid domains are
\begin{equation}\label{eq:domains}
	\O_s^\e:=\bigcup_{\kappa\in\Kc_\e}\Oc_{\e,\kappa},
	\qquad
	\O_p^\e:=\O\setminus\overline{\O_s^\e},
	\qquad
	\Gamma_s^\e:=\partial\O_s^\e
	=\bigcup_{\kappa\in\Kc_\e}\bigl(x_{\e,\kappa}+a_\e\,\partial\Oc\bigr),
\end{equation}
so that $\partial\O_p^\e=\partial\O\cup\Gamma_s^\e$ and
$\overline\O=\overline{\O_p^\e}\cup\overline{\O_s^\e}$.

Let $\e_0\in(0,1)$ be such that $\delta_{\e_0}R_{\Oc}\leq d_0/2$; since
$\e\mapsto\delta_\e$ is increasing on $(0,1)$ for $\alpha>1$, the same bound holds
for all $0<\e<\e_0$. For such $\e$ the rescaled obstacle
$Y_{s,\e}:=y_0+\delta_\e\Oc$ satisfies $\overline{Y_{s,\e}}\subset Y$ with
$\operatorname{dist}(\overline{Y_{s,\e}},\partial Y)\geq d_0-\delta_\e R_{\Oc}\geq
d_0/2$. Equivalently, $\overline{\Oc_{\e,\kappa}}\subset\e(\kappa+Y)$ with
\begin{equation}\label{eq:obstacle-separation}
	\operatorname{dist}\bigl(\Oc_{\e,\kappa},\partial(\e(\kappa+Y))\bigr)\geq c_0\e,
	\qquad c_0:=d_0/2,
\end{equation}
for all $\kappa\in\Kc_\e$. In particular, the obstacles are mutually disjoint,
uniformly separated at the periodicity scale, and do not meet $\partial\O$.
Hence, for every fixed $\e\in(0,\e_0)$, the set $\O_p^\e$ is a bounded Lipschitz
domain, being $\O$ with finitely many disjoint Lipschitz obstacles removed.

Two consequences of this construction enter the analysis. First, the total
obstacle volume satisfies
\begin{equation}\label{eq:obstacle-volume}
	|\O_s^\e|=\#\Kc_\e\,a_\e^3\,|\Oc|
	=O\bigl(\e^{3(\alpha-1)}\bigr)\longrightarrow0 ,
\end{equation}
so the perforation is asymptotically negligible in measure, and the
contribution of $\O_s^\e$ to the weak formulation vanishes in the quantitative
form of Lemma~\ref{lem:hole-error}. Second, the parameter governing the
cumulative obstacle-scale viscous energy is
\begin{equation}\label{eq:subcritical-parameter}
	\rho_\e:=\frac{a_\e}{\e^{3}}=\e^{\alpha-3}\longrightarrow0
	\qquad\text{since }\alpha>3 ,
\end{equation}
so the obstacles are in the subcritical regime. The strong convergence of the
solenoidal correction of Lemma~\ref{lem:solenoidal-restriction} to the identity
relies on \eqref{eq:subcritical-parameter}; \eqref{eq:obstacle-volume} alone
gives only convergence in $L^2$ (see Remark~\ref{rem:alpha-role}).

Throughout we write
\begin{equation}\label{eq:char-functions}
	\chi_p^\e:=\chi_{\O_p^\e},
	\qquad
	\chi_s^\e:=\chi_{\O_s^\e}=1-\chi_p^\e
\end{equation}
for the characteristic functions of the fluid and the solid part; we use the
same symbols for their time-independent extensions to $\O_T$. By
\eqref{eq:obstacle-volume}, for every $q\in[1,\infty)$,
\begin{equation}\label{eq:char-convergence}
	\|\chi_p^\e-1\|_{L^q(\O)}
	=\|\chi_s^\e\|_{L^q(\O)}
	=|\O_s^\e|^{1/q}
	\leq C\,\e^{3(\alpha-1)/q}
	\longrightarrow0 ,
\end{equation}
so $\chi_p^\e\to1$ strongly in $L^q(\O)$ for every $q<\infty$. Only the
hypothesis $\alpha>1$ is used here.
	
\subsection{Model assumptions}\label{SSec32}

We collect the assumptions used throughout the paper.

\textbf{Initial data.}
Let $\phi_0\in H^1(\O)$ and set $\phi_0^\e:=\phi_0|_{\O_p^\e}$. Moreover, let
$\Gu_0^\e\in H_\e^\sigma$ (see Section \ref{Sec02}) and let
$\overline\Gu_0^\e$ denote the zero extension of $\Gu_0^\e$ to $\O$. We assume
that there exists $\Gu_0\in H^\sigma$ such that
\begin{equation}\label{ICA01}
	\overline\Gu_0^\e\longrightarrow\Gu_0
	\quad\text{strongly in }L^2(\O)^3,
	\qquad
	\|\Gu_0^\e\|_{L^2(\O_p^\e)^3}\leq M_1,
	\qquad
	\|\phi_0\|_{H^1(\O)}\leq M_2,
\end{equation}
for constants $M_1, M_2>0$ independent of $\e$.

\textbf{Double-well potential.}
The double-well potential $F:\mathbb{R}\to\mathbb{R}$ is
\begin{equation}\label{PotentialF}
	F(s):=\tfrac{1}{4}(s^2-1)^2,\qquad s\in\mathbb{R},
\end{equation}
with derivative $F'(s)=s^3-s$. It satisfies
\begin{equation}\label{F01}
	\tfrac{1}{2}s^2-\tfrac{3}{4}\leq F(s)\leq\tfrac{1}{4}s^4+\tfrac{1}{4},
	\qquad
	|F'(s)|\leq|s|^3+|s|\leq C(1+|s|^3),
	\qquad s\in\mathbb{R}.
\end{equation}

\textbf{Source term.}
The source term $G:\mathbb{R}\to\mathbb{R}$ satisfies
\begin{equation}\label{AssG01}
	G\in C^1(\mathbb{R}),\qquad G(0)=0,\qquad
	0<c_1\leq G'(s)\leq c_2\quad\text{for all }s\in\mathbb{R},
\end{equation}
for constants $c_1,c_2>0$. Consequently, by the mean value theorem,
\begin{equation}\label{SourceGrowth}
	c_1|s|\leq|G(s)|\leq c_2|s|,
	\qquad
	c_1s^2\leq G(s)s\leq c_2s^2,
	\qquad s\in\mathbb{R}.
\end{equation}

\textbf{Reynolds number.}
The pore-scale Reynolds number $\Rey_\e$ in the convective term of
\eqref{MainP01} satisfies \eqref{ReynoldsAssumption}.

\subsection{Weak formulation}\label{SSec33}
The admissible spaces are
\[
\begin{aligned}
	\mathfrak{U}^\e&=L^\infty(S;L^2(\O_p^\e))^3\cap L^2(S;\GH^1_{\div}(\O_p^\e))\cap W^{1,4/3}(S;\GH^{-1}_{\div}(\O_p^\e)),\\
	\mathfrak{C}^\e&=L^\infty(S;H^1(\O_p^\e))\cap H^1(S;H^{-1}(\O_p^\e)),\qquad
	\mathfrak{W}^\e=L^2(S;H^1(\O_p^\e)),
\end{aligned}
\]
with macroscopic counterparts
\[
\begin{aligned}
	\mathfrak{U}&=L^\infty(S;L^2(\O))^3\cap L^2(S;H^1_{0,\div}(\O)^3)\cap C_w([0,T];H^\sigma),\\
	\mathfrak{C}&=L^\infty(S;H^1(\O))
	\cap
	C([0,T];L^2(\O)),\qquad
	\mathfrak{W}=L^2(S;H^1(\O)).
\end{aligned}
\]
\begin{definition}[Weak solution]\label{Def01}
	A triplet $(\mathbf u^\e,\phi^\e,\mu^\e)\in\mathfrak{U}^\e\times\mathfrak{C}^\e\times\mathfrak{W}^\e$ is called a weak solution of
	\eqref{MainP01} if
	$\mathbf u^\e(0)=\mathbf u_0^\e$ and $\phi^\e(0)=\phi^\e_0$ in $\O_p^\e$ are attained,
	and if the identities \eqref{weak-momentum1}--\eqref{weak-mu1} hold for all
	$\bvp_1\in L^4(S;\GH^1_{\div}(\O_p^\e))$ and
	$\varphi_2,\varphi_3\in L^2(S;H^1(\O_p^\e))$, together with the energy
	inequality \eqref{weak-energy1}:
	\begin{align}
		\label{weak-momentum1}
		&\int_S\!\langle\partial_t\mathbf{u}^\e,\bvp_1\rangle\,\md t
		+\Rey_\e\int_{S\times\O_p^\e} b^\e(\mathbf u^\e,\mathbf u^\e,\bvp_1)\,\md(x,t)
		+2\nu\int_{S\times\O_p^\e} D(\mathbf{u}^\e):D(\bvp_1)\,\md(x,t)\\
		&\hspace{5cm}
		+\int_{S\times\O_p^\e}\lambda\phi^\e\nabla\mu^\e\cdot\bvp_1\,\md(x,t)
		=\int_{S\times\O_p^\e}\th^\e\cdot\bvp_1\,\md(x,t),\nonumber\\[1ex]
		\label{weak-phi1}
		&\int_S\!\langle\partial_t\phi^\e,\varphi_2\rangle\,\md t
		+\int_{S\times\O_p^\e}\hskip -3mm G(\phi^\e)\varphi_2\,\md(x,t)
		+\int_{S\times\O_p^\e}\hskip -3mm\nabla\mu^\e\cdot\nabla\varphi_2\,\md(x,t)
		+\int_{S\times\O_p^\e}\hskip -3mm(\mathbf{u}^\e\cdot\nabla \phi^\e)\varphi_2\,\md(x,t)=0,\\[1ex]
		\label{weak-mu1}
		&\int_{S\times\O_p^\e}\mu^\e\varphi_3\,\md(x,t)
		=\int_{S\times\O_p^\e}\nabla\phi^\e\cdot\nabla\varphi_3\,\md(x,t)
		+\int_{S\times\O_p^\e}F'(\phi^\e)\varphi_3\,\md(x,t),
	\end{align}
	where the trilinear form of the convective term is written in the
	antisymmetric form
	\begin{equation}\label{eq:trilinear-form}
		b^\e(\mathbf v,\mathbf w,\mathbf z)
		:=-\,\mathbf v\otimes\mathbf w:\nabla\mathbf z
		=-\,v_iw_j\,\partial_iz_j .
	\end{equation}
	Setting
	\begin{equation}\label{energy-functional}
		\mathcal{E}^\e(t):=\frac12\int_{\O_p^\e}|\mathbf u^\e(t)|^2\,\md x
		+\lambda\int_{\O_p^\e}\Big(\frac12|\nabla\phi^\e(t)|^2+F(\phi^\e(t))\Big)\md x,
	\end{equation}
	the energy inequality reads
	\begin{multline}
		\label{weak-energy1}
		\mathcal{E}^\e(t)
		+2\nu\int_0^t\!\!\int_{\O_p^\e}D(\mathbf u^\e):D(\mathbf u^\e)\,\md(x,s)
		+\lambda\int_0^t\!\!\int_{\O_p^\e}|\nabla\mu^\e|^2\,\md(x,s)\\
		\le\ \mathcal{E}^\e(0)
		+\int_0^t\!\!\int_{\O_p^\e}\th^\e\cdot\mathbf u^\e\,\md(x,s)
		-\lambda\int_0^t\!\!\int_{\O_p^\e}G(\phi^\e)\mu^\e\,\md(x,s)
	\end{multline}
	for a.e.\ $t\in S$, where
	$\mathcal{E}^\e(0)=\frac12\|\mathbf u_0^\e\|_{L^2(\O_p^\e)}^2
	+\lambda\big(\frac12\|\nabla\phi_0\|_{L^2(\O_p^\e)}^2
	+\int_{\O_p^\e}F(\phi_0)\,\md x\big)$.
\end{definition}

\begin{remark}[Existence of microscopic weak solutions]
	\label{rem:existence}
	For fixed $\e\in(0,\e_0)$, the domain $\O_p^\e$ is a bounded
	Lipschitz domain. Hence, the perforations enter the existence argument
	only through the fixed geometry, the associated function spaces, and
	the constants in the estimates, which may depend on $\e$.
	For every $\th^\e\in L^2(\Qc_p^\e)^3$, problem \eqref{MainP01}
	admits at least one weak solution
	$$
	(\Gu^\e,\phi^\e,\mu^\e)
	\in \fU^\e\times\fC^\e\times\fW^\e
	$$
	in the sense of Definition~\ref{Def01}. Indeed, a Faedo--Galerkin
	approximation and the energy estimate \eqref{weak-energy1} yield
	bounds uniform with respect to the Galerkin dimension. These bounds
	ensure global continuation of the approximate solutions and provide
	the compactness required to pass to the limit in the transport,
	capillary, source, and potential terms.
	Existence results for local NSCH systems are established in
	\cite{boyer1999mathematical,abels2009diffuse,
		giorgini2019uniqueness}; the nonlocal case is treated in
	\cite{colli2012global}, while nonhomogeneous and degenerate-mobility
	models are considered in
	\cite{abels2013incompressible,giorgini2020weak}.
	Cahn--Hilliard equations with nonlinear source terms are studied in
	\cite{miranville2021source}; see also
	\cite{giorgini2026chemotaxis} for an NSCH system incorporating a
	phase-field mass source. A closely related existence result for a
	three-dimensional NSCH system with an oscillating viscosity tensor,
	a nonlinear source, and mixed no-slip/free-slip boundary conditions
	in a dilute perforated domain is established in
	\cite[Theorem~4.5]{chakrabortty2025navier}.
	No uniqueness of microscopic weak solutions is asserted.
\end{remark}

\begin{remark}\label{Rem-energy}
	For sufficiently smooth solutions, \eqref{weak-energy1} holds with equality.
	Indeed, taking $\bvp_1=\Gu^\e$ in \eqref{weak-momentum1},
	$\varphi_2=\lambda\mu^\e$ in \eqref{weak-phi1} and
	$\varphi_3=\lambda\partial_t\phi^\e$ in \eqref{weak-mu1}, and adding the
	resulting identities, the convective term drops and the two coupling terms cancel by
	$\div\Gu^\e=0$ and the boundary conditions in \eqref{MainP01},
	while the choice $\varphi_3=\lambda\partial_t\phi^\e$ turns the right-hand
	side of \eqref{weak-mu1} into
	$\lambda\frac{\md}{\md t}\big(\frac12\|\nabla\phi^\e\|_{L^2(\O_p^\e)}^2
	+\int_{\O_p^\e}F(\phi^\e)\,\md x\big)$.
	At the weak level $\partial_t\phi^\e$ is only an element of
	$L^2(S;H^{-1}(\O_p^\e))$ and is therefore not an admissible test function in
	\eqref{weak-mu1}. Accordingly, \eqref{weak-energy1} is imposed as part of the
	definition: it is obtained by writing the above identity for the solutions of a
	Galerkin approximation, for which the test function is
	admissible, and passing to the limit.
\end{remark}

\section{The control problems and main results}\label{SecMain}

Throughout this section, $\alpha>3$ and the assumptions \eqref{ICA01},
\eqref{AssG01} and \eqref{ReynoldsAssumption} hold. Let
$\beta>0$, and let $\mathbf u_d\in L^2(\O_T)^3$ and $\phi_d\in L^2(S;H^1(\O))$
be a prescribed velocity and order parameter. We pose the control problem at the
level $\e$ over control--state pairs (Section~\ref{SSecMicroOCP}), prove that it
has optimal pairs whose controls and states are bounded uniformly in $\e$
(Section~\ref{SSecExist}), identify the limits of such pairs by homogenizing the
state system (Section~\ref{SSecHomState}), pose the limit control problem on the
full domain, where the state system is uniquely solvable
(Section~\ref{SSecLimOCP}), and compare the two problems by sequential
$\Gamma$-convergence (Section~\ref{SSecGamma}).

\subsection{The microscopic control problem}\label{SSecMicroOCP}

For a control $\th^\e\in L^2(\Qc^\e_p)^3$ and a corresponding weak solution
$(\Gu^\e,\phi^\e,\mu^\e)$ of \eqref{MainP01} in the sense of
Definition~\ref{Def01}, we consider the cost functional
\begin{equation}\label{Cost01}
	\GJ^\e(\th^\e,\Gu^\e,\phi^\e):=
	\frac12\|\Gu^\e-\mathbf u_d\|^2_{L^2(\Qc^\e_p)}
	+\frac12\|\phi^\e-\phi_d\|^2_{L^2(S;H^1(\O^\e_p))}
	+\frac{\beta}{2}\|\th^\e\|^2_{L^2(\Qc^\e_p)} ,
\end{equation}
with the norm \eqref{H1PoreNorm}. The control is distributed over the pore space
and is not subject to pointwise constraints; the velocity is tracked in $L^2$
and the order parameter in $H^1$. Since uniqueness of weak solutions of
\eqref{MainP01} is not known, the control-to-state correspondence
\begin{equation}\label{MultiSolOp}
	\Sc^\e:\ L^2(\Qc^\e_p)^3 \ \rightrightarrows\
	\fU^\e\X\fC^\e\X\fW^\e,
	\qquad
	\Sc^\e(\th^\e):=\bigl\{(\Gu^\e,\phi^\e,\mu^\e)\ \text{weak solution of
		\eqref{MainP01} with body force }\th^\e\bigr\},
\end{equation}
is multivalued, and the map $\th^\e\mapsto\GJ^\e(\th^\e,\Sc^\e(\th^\e))$ is not a
function of $\th^\e$. The minimization is therefore carried out over the set of
admissible control--state pairs
\begin{equation}\label{AdmSetEps}
	\Ac^\e:=\Bigl\{\bigl(\th^\e,(\Gu^\e,\phi^\e,\mu^\e)\bigr):\
	\th^\e\in L^2(\Qc^\e_p)^3,\
	(\Gu^\e,\phi^\e,\mu^\e)\in\Sc^\e(\th^\e)\Bigr\},
\end{equation}
which is nonempty by Remark~\ref{rem:existence}, and the microscopic problem
reads
\begin{equation}\label{OPeps}\tag{$\mathrm{OP}^\e$}
	\text{minimize }\ \GJ^\e(\th^\e,\Gu^\e,\phi^\e)
	\quad\text{over all }\
	\bigl(\th^\e,(\Gu^\e,\phi^\e,\mu^\e)\bigr)\in\Ac^\e .
\end{equation}
Since $\Ac^\e$ is a relation and not a graph, \eqref{OPeps} selects a state
together with the control. Its value is
\begin{equation}\label{InfEps}
	\m_\e:=\inf\Bigl\{\GJ^\e(\th^\e,\Gu^\e,\phi^\e):\
	\bigl(\th^\e,(\Gu^\e,\phi^\e,\mu^\e)\bigr)\in\Ac^\e\Bigr\}\ \ge0 .
\end{equation}
To pass to the limit $\e\to0$, we first need optimal pairs, and bounds on them
that do not depend on $\e$.

\subsection{Existence and uniform bounds}\label{SSecExist}

The Tikhonov term bounds the controls of all pairs whose cost is close to the
infimum.

\begin{lemma}\label{LemBoundControl}
	There exist $C_0>0$ and $K>0$, both independent of $\e$, such that
	\begin{equation}\label{InfBound01}
		\m_\e\le C_0
		\qquad\text{for every }\e\in(0,\e_0),
	\end{equation}
	and such that every admissible pair
	$\bigl(\th^\e,(\Gu^\e,\phi^\e,\mu^\e)\bigr)\in\Ac^\e$ with
	$\GJ^\e(\th^\e,\Gu^\e,\phi^\e)\le C_0+1$ satisfies
	\begin{equation}\label{ControlBound01}
		\|\th^\e\|_{L^2(\Qc^\e_p)}\le K .
	\end{equation}
	In particular \eqref{ControlBound01} holds for every optimal control and, from
	some index on, for every term of any minimizing sequence.
\end{lemma}

\begin{theorem}\label{ThmExistOPe}
	For every $\e\in(0,\e_0)$ the problem \eqref{OPeps} admits at least one
	solution; in particular the infimum $\m_\e$ of \eqref{InfEps} is attained.
\end{theorem}

Controls bounded in $L^2$ give states bounded uniformly in $\e$, in the
$\e$-dependent spaces.

\begin{lemma}\label{Est01}
	Let $(\Gu^\e,\phi^\e,\mu^\e)\in \fU^\e\X\fC^\e\X\fW^\e$ be a weak solution of
	\eqref{MainP01} in the sense of Definition~\ref{Def01}, and assume that
	$\|\th^\e\|_{L^2(\Qc^\e_p)}\le C$ with $C$ independent of $\e$. Then
	\begin{equation}\label{MainE01}
		\begin{aligned}
			\|\Gu^\e\|_{L^\infty(S;L^2(\O^\e_p))}
			+\|\Gu^\e\|_{L^2(S;H^1(\O^\e_p))}
			+\|\partial_t\Gu^\e\|_{L^{4/3}(S;\GH_{\div}^{-1}(\O^\e_p))}
			&\le C,\\[0.2cm]
			\|\phi^\e\|_{L^\infty(S;H^1(\O^\e_p))}
			+\|\partial_t\phi^\e\|_{L^2(S;H^{-1}(\O^\e_p))}
			+\|\mu^\e\|_{L^2(S;H^1(\O^\e_p))}
			&\le C,
		\end{aligned}
	\end{equation}
	with $C>0$ independent of $\e$.
\end{lemma}

By Lemma~\ref{LemBoundControl}, Lemma~\ref{Est01} applies to optimal pairs and
to all pairs with bounded cost. We now identify the limits of such pairs.

\subsection{Homogenization of the state system}\label{SSecHomState}

The solutions live on the $\e$-dependent domain $\O_p^\e$, so limits are taken
through appropriate extensions to $\O$. Each field is extended in two ways. The primal
extension $\wt w$ of \eqref{eq:extensions-notation}, a Sobolev extension for the
scalar fields (Lemma \ref{lem:scalar-two-extensions}) and a solenoidal extension for the velocity (Lemma \ref{LemD02}), is bounded uniformly
in the energy space and carries the weak limits of the gradients. The zero
extension $\overline w$ carries the distributional time derivative and, for the
velocity, the divergence-free constraint. The two extensions differ in $L^2$ by
a term of order $a_\e/\e$, so they have the same $L^2$ limit. For the velocity,
$\partial_t\Gu^\e$ is bounded only in the dual of an $\e$-dependent space, and a
fixed divergence-free field on $\O$ is not an admissible test function, since it
does not satisfy the impermeability condition on $\Gamma_s^\e$. The solenoidal
restriction operator $\Rc^\e$ of Lemma~\ref{lem:solenoidal-restriction} provides
admissible test fields; combined with the two extensions, the pairings of
$\Gu^\e$ with these fields give strong compactness of the velocity and identify
the initial conditions.

\begin{lemma}[Compactness of the solution sequence]
	\label{lem:limit-compactness}
	Let $\th^\e\in L^2(\Qc^\e_p)^3$ satisfy
	$\|\th^\e\|_{L^2(\Qc^\e_p)}\leq C$, and let
	$(\Gu^\e,\phi^\e,\mu^\e)\in\fU^\e\times\fC^\e\times\fW^\e$ be a weak
	solution of \eqref{MainP01} with body force $\th^\e$ in the sense of
	Definition~\ref{Def01}. Then, up to a subsequence, there exists
	$(\Gu,\phi,\mu)\in\fU\times\fC\times\fW$ such that the following assertions
	hold.
	
	\textbf{(i) Convergence of the extensions.}
	For every $2\leq s<6$,
	\begin{equation}\label{TSL0}
		\begin{aligned}
			\wt\Gu^\e,\ \overline{\Gu^\e}&\to\Gu
			\quad\text{strongly in }L^2(S;L^s(\O)^3),
			&&\nabla\wt\Gu^\e,\ \overline{\nabla\Gu^\e}\rightharpoonup\nabla\Gu
			\quad\text{weakly in }L^2(\O_T)^{3\times3},\\
			\wt\phi^\e,\ \overline{\phi^\e}&\to\phi
			\quad\text{strongly in }C([0,T];L^s(\O)),
			&&\nabla\wt\phi^\e,\ \overline{\nabla\phi^\e}\rightharpoonup\nabla\phi
			\quad\text{weakly in }L^2(\O_T)^3,\\
			\wt\mu^\e,\,\overline{\mu^\e}&\rightharpoonup\mu
			\quad\text{weakly in }L^2(\O_T),\quad&&
			\nabla\wt\mu^\e,\ \overline{\nabla\mu^\e}\rightharpoonup\nabla\mu
			\quad\text{weakly in }L^2(\O_T)^3 .
		\end{aligned}
	\end{equation}
	Moreover, for every
	$\zeta\in C^\infty_{c,\div}(\O)^3$,
	\begin{equation}\label{TSL-pairing}
		(\Gu^\e(\cdot),\Rc^\e\zeta)_{L^2(\O_p^\e)}\longrightarrow
		(\Gu,\zeta)_{L^2(\O)}\qquad\text{strongly in }C([0,T]).
	\end{equation}
	
	\textbf{(ii) Nonlinear terms.}
	For every $1\leq r<3/2$,
	\begin{equation}\label{TSL02}
		\begin{aligned}
			\wt\phi^\e\,\wt\Gu^\e&\to\phi\Gu\quad
			\text{strongly in }L^2(\O_T)^3,\quad
			&&F(\wt\phi^\e)\to F(\phi)\quad\text{strongly in }L^1(\O_T),\\
			F'(\wt\phi^\e)&\rightharpoonup F'(\phi)\quad\text{weakly in }L^2(\O_T),\quad
			&&G(\wt\phi^\e)\to G(\phi)\quad\text{strongly in }L^2(\O_T),\\
			\wt\phi^\e\,\nabla\wt\mu^\e&\rightharpoonup\phi\nabla\mu\quad\text{weakly in }L^2(S;L^r(\O)^3).
		\end{aligned}
	\end{equation}
	
	\textbf{(iii) Identification of the initial conditions.}
	The limit fields satisfy
	\begin{equation}\label{eq:limit-initial-conditions}
		\Gu(0)=\Gu_0\quad\text{in }H^\sigma,
		\qquad
		\phi(0)=\phi_0\quad\text{in }L^2(\O).
	\end{equation}
\end{lemma}

The limit is identified for forces that converge weakly. On the body force we
assume that $\th^\e\in L^2(\Qc^\e_p)^3$ satisfies
\begin{equation}\label{ForceConvergence}
	\|\th^\e\|_{L^2(\Qc^\e_p)}\leq C
	\qquad\text{and}\qquad
	\overline{\th^\e}\rightharpoonup\th_0
	\quad\text{weakly in }L^2(\O_T)^3
\end{equation}
for some $\th_0\in L^2(\O_T)^3$, where $\overline{\th^\e}$ denotes the zero
extension of $\th^\e$ to $\O_T$.

	\begin{theorem}[Homogenized system]
	\label{ThHomogenized}\label{ThCritical}
	Assume \eqref{ForceConvergence} and let $(\Gu,\phi,\mu)$ be the limit
	fields obtained in Lemma~\ref{lem:limit-compactness}. Then
	$(\Gu,\phi,\mu)\in\fU\X\fC\X\fW$ with
	\begin{equation}\label{LimitTimeRegularity}
		\partial_t\Gu\in L^2\bigl(S;H^{-1}_{0,\div}(\O)^3\bigr),
		\qquad
		\partial_t\phi\in L^2\bigl(S;H^{-1}(\O)\bigr),
	\end{equation}
	and $(\Gu,\phi,\mu)$ is a weak solution of the Stokes--Cahn--Hilliard
	system on $\O$ in the sense that
	\begin{align}
		\int_S\langle\partial_t\Gu,\bvp_1\rangle\,\md t
		+
		2\nu\int_{\O_T}D\Gu:D\bvp_1\,\md(x,t)
		+
		\lambda\int_{\O_T}\phi\,\nabla\mu\cdot\bvp_1\,\md(x,t)
		=
		\int_{\O_T}\th_0\cdot\bvp_1\,\md(x,t),
		\label{HomogenizedMomentum}\\[1mm]
		\int_S\langle\partial_t\phi,\varphi_2\rangle\,\md t
		-
		\int_{\O_T}\phi\,\Gu\cdot\nabla\varphi_2\,\md(x,t)
		+
		\int_{\O_T}\nabla\mu\cdot\nabla\varphi_2\,\md(x,t)
		+
		\int_{\O_T}G(\phi)\varphi_2\,\md(x,t)
		=0,
		\label{HomogenizedPhase}\\[1mm]
		\int_{\O_T}\mu\varphi_3\,\md(x,t)
		-
		\int_{\O_T}\nabla\phi\cdot\nabla\varphi_3\,\md(x,t)
		-
		\int_{\O_T}F'(\phi)\varphi_3\,\md(x,t)
		=0,
		\label{HomogenizedPotential}
	\end{align}
	for every $\bvp_1\in L^2(S;H^1_{0,\div}(\O)^3)$ and every
	$\varphi_2,\varphi_3\in L^2(S;H^1(\O))$, with the initial conditions
	\eqref{eq:limit-initial-conditions}. Moreover, the energy estimate
	\begin{multline}
		\label{weak-energy2}
		\mathcal{E}(t)
		+2\nu\int_0^t\!\!\int_{\O}D(\Gu):D(\Gu)\,\md(x,s)
		+\lambda\int_0^t\!\!\int_{\O}|\nabla\mu|^2\,\md(x,s)\\
		\le\ \mathcal{E}(0)
		+\int_0^t\!\!\int_{\O}\th_0\cdot\Gu\,\md(x,s)
		-\lambda\int_0^t\!\!\int_{\O}G(\phi)\mu\,\md(x,s)
	\end{multline}
	holds for a.e.\ $t\in S$, where
	\[
	\mathcal{E}(t)
	:=\frac12\|\Gu(t)\|_{L^2(\O)}^2
	+\lambda\Bigl(\frac12\|\nabla\phi(t)\|_{L^2(\O)}^2
	+\int_{\O}F(\phi(t))\,\md x\Bigr),
	\]
	and $\mathcal{E}(0)=\frac12\|\Gu_0\|_{L^2(\O)}^2
	+\lambda\bigl(\frac12\|\nabla\phi_0\|_{L^2(\O)}^2
	+\int_{\O}F(\phi_0)\,\md x\bigr)$.
\end{theorem}
%
Since the order parameter is tracked in $H^1$,
the cost also requires the strong convergence of its gradient.

\begin{corollary}\label{CorGradPhiStrong}
	Under the hypotheses of Theorem~\ref{ThHomogenized}, along the subsequence of
	Lemma~\ref{lem:limit-compactness},
	\begin{equation}\label{GradPhiStrong}
		\overline{\nabla\phi^\e}=\chi_p^\e\nabla\wt\phi^\e\longrightarrow\nabla\phi
		\qquad\text{strongly in }L^2(\O_T)^3 .
	\end{equation}
\end{corollary}

No uniqueness has been used so far: Theorem~\ref{ThHomogenized} holds for every
choice of weak solutions at the level $\e$. To recover every admissible pair of
the limit problem, the limit system must be uniquely solvable.

\subsection{The homogenized control problem}\label{SSecLimOCP}

The scaling \eqref{ReynoldsAssumption} is chosen for this purpose. The a priori
estimates hold for any $\Rey_\e\le1$, since the convective term is orthogonal to
$\Gu^\e$ and absent from the energy identity. The condition $\Rey_\e\to0$
removes the convective term from the limit; if it survived, the homogenized
system would inherit the open uniqueness problem for the three-dimensional
Navier--Stokes equations. With $\Rey_\e\to0$, the limit system is uniquely
solvable when $\partial\O\in C^{2,1}$. Since the mean of the order parameter is
not conserved, it is carried as a separate component of the energy in the
proof; the regularity of $\partial\O$ enters through elliptic regularity for the
Neumann Laplacian.

\begin{theorem}[Regularity and uniqueness for the homogenized system]
	\label{ThmUniq}
	Let $\O\subset\R^3$ be a bounded domain with $\partial\O\in C^{2,1}$, 
	let the
	initial data satisfy \eqref{ICA01}, and let $\th_0\in L^2(\O_T)^3$. Then every
	weak solution $(\Gu,\phi,\mu)$ of the homogenized system in the sense of
	\eqref{HomogenizedMomentum}--\eqref{HomogenizedPotential} satisfies the
	additional regularity
	\begin{equation}\label{ExtraReg01}
		\phi\in L^2(S;H^3(\O)),
	\end{equation}
	and the weak solution is unique.
\end{theorem}

By Theorems~\ref{ThHomogenized} and~\ref{ThmUniq}, for every
$\th\in L^2(\O_T)^3$ the system
\eqref{HomogenizedMomentum}--\eqref{HomogenizedPotential} with body force $\th$
and initial data \eqref{eq:limit-initial-conditions} has a unique weak solution,
denoted by
\[
\Sc(\th):=(\Gu_\th,\phi_\th,\mu_\th)\in\fU\X\fC\X\fW .
\]
The homogenized cost functional is
\begin{equation}\label{Cost02}
	\GJ(\th,\Gu,\phi):=
	\frac12\|\Gu-\mathbf u_d\|^2_{L^2(\O_T)}
	+\frac12\|\phi-\phi_d\|^2_{L^2(S;H^1(\O))}
	+\frac{\beta}{2}\|\th\|^2_{L^2(\O_T)} ,
\end{equation}
the set of admissible pairs is
\begin{multline}\label{AdmSetLim}
	\Ac:=\Bigl\{\bigl(\th,(\Gu,\phi,\mu)\bigr):\ \th\in L^2(\O_T)^3,\\
	(\Gu,\phi,\mu)\ \text{weak solution of
		\eqref{HomogenizedMomentum}--\eqref{HomogenizedPotential}
		with body force }\th\text{ and initial data
		\eqref{eq:limit-initial-conditions}}\Bigr\},
\end{multline}
and the homogenized problem reads
\begin{equation}\label{OPlim}\tag{$\mathrm{OP}$}
	\text{minimize }\ \GJ(\th,\Gu,\phi)
	\quad\text{over all }\ \bigl(\th,(\Gu,\phi,\mu)\bigr)\in\Ac .
\end{equation}
Unlike $\Ac^\e$, the set $\Ac$ is the graph of the control-to-state map
$\Sc:L^2(\O_T)^3\to\fU\X\fC\X\fW$, and \eqref{OPlim} may equivalently be written
in the reduced form
\begin{equation}\label{ReducedOP}
	\text{minimize}\quad
	\GJ_{\mathrm{red}}(\th):=\GJ\bigl(\th,\Gu_\th,\phi_\th\bigr)
	\quad\text{over }\th\in L^2(\O_T)^3 .
\end{equation}
Its value is
\begin{equation}\label{InfLim}
	\m:=\inf\Bigl\{\GJ(\th,\Gu,\phi):\ \bigl(\th,(\Gu,\phi,\mu)\bigr)\in\Ac\Bigr\}\ \ge0 .
\end{equation}
The two problems are posed on different domains, and \eqref{OPeps} is not in
reduced form. We compare them through zero extensions of control--state pairs.

\subsection{\texorpdfstring{$\Gamma$}{Gamma}-convergence and homogenization of the control problem}\label{SSecGamma}

A control--state pair
$\bigl(\th^\e,(\Gu^\e,\phi^\e,\mu^\e)\bigr)\in\Ac^\e$ is written as the
quadruple
\begin{equation}\label{IdentificationMap}
	\Xi^\e:=\bigl(\th^\e,\Gu^\e,\phi^\e,\mu^\e\bigr)
	\ \in\
	\fX^\e:=L^2(\Qc^\e_p)^3\X\fU^\e\X\fC^\e\X\fW^\e ,
\end{equation}
and we write $\Xi^\e\in\Ac^\e$. Since the components live on the
$\e$-dependent domain $\O_p^\e$, limits are taken through the zero extensions
\eqref{eq:extensions-notation}, in the sense of the following definition.

\begin{definition}[Convergence of control--state quadruples]\label{DefOCConv}
	Let $\Lambda^\e=(\vartheta^\e,\bv^\e,\psi^\e,\nu^\e)\in\fX^\e$ and
	$\Lambda=(\vartheta,\bv,\psi,\nu)\in\fX:=L^2(\O_T)^3\X\fU\X\fC\X\fW$. We
	write $\Lambda^\e\ocarrow\Lambda$ as $\e\to0$ if
	\begin{alignat}{2}
		&\overline{\vartheta^\e}\rightharpoonup\vartheta
		\quad\text{weakly in }L^2(\O_T)^3, \label{OCConv1}\\
		&\overline{\bv^\e}\to\bv
		\quad\text{strongly in }L^2(\O_T)^3,\qquad
		&& \overline{\nabla\bv^\e}\rightharpoonup \nabla\bv
		\quad\text{weakly in }L^2(\O_T)^{3\X3},\label{OCConv2}\\
		&\overline{\psi^\e}\to\psi
		\quad\text{strongly in }C([0,T];L^2(\O)),\qquad
		&&\overline{\nabla\psi^\e}\to\nabla\psi
		\quad\text{strongly in }L^2(\O_T)^3, \label{OCConv3}\\
		&\overline{\nu^\e}\rightharpoonup\nu
		\quad\text{weakly in }L^2(\O_T),\qquad
		&&\overline{\nabla\nu^\e}\rightharpoonup \nabla\nu
		\quad\text{weakly in }L^2(\O_T)^3 . \label{OCConv4}
	\end{alignat}
\end{definition}
All
statements below are sequential, and the fundamental theorem of
$\Gamma$-convergence is used only in its sequential form; no metrizability is
required.

	Since $\overline{w}$ vanishes on $\O^\e_s$ and coincides with $w$ on
$\O^\e_p$, one has $\chi^\e_p\overline{w}=\overline{w}$ and
\begin{equation}\label{eq:ideC}
	\|w-w_d\|^2_{L^2(\Qc^\e_p)}
	=\bigl\|\chi^\e_p(\overline{w}-w_d)\bigr\|^2_{L^2(\O_T)}
	=\bigl\|\overline{w}-\chi^\e_p w_d\bigr\|^2_{L^2(\O_T)}
\end{equation}
for every $w_d\in L^2(\O_T)$. Applying \eqref{eq:ideC} to $\Gu^\e$,
$\phi^\e$ and, componentwise, $\partial_i\phi^\e$, and using that the zero
extension is an isometry of $L^2(\O_p^\e)$ into $L^2(\O)$, the cost
\eqref{Cost01} is written on the fixed domain as
\begin{equation}\label{GammaCost}
	\GJ^\e(\th^\e,\Gu^\e,\phi^\e)
	=\frac12\|\overline{\Gu^\e}-\chi^\e_p\mathbf u_d\|^2_{L^2(\O_T)}
	+\frac12\|\overline{\phi^\e}-\chi^\e_p\phi_d\|^2_{L^2(\O_T)}
	+\frac12\|\overline{\nabla\phi^\e}-\chi^\e_p\nabla\phi_d\|^2_{L^2(\O_T)}
	+\frac{\beta}{2}\|\overline{\th^\e}\|^2_{L^2(\O_T)} .
\end{equation}
We define $\Jc^\e:\fX^\e\to[0,+\infty]$ and $\Jc:\fX\to[0,+\infty]$ by
\begin{equation}\label{GammaFunctional}
	\Jc^\e(\Lambda^\e):=
	\begin{cases}
		\GJ^\e(\vartheta^\e,\bv^\e,\psi^\e)
		&\text{if }\Lambda^\e\in\Ac^\e,\\[1mm]
		+\infty &\text{otherwise},
	\end{cases}
	\qquad
	\Jc(\Lambda):=
	\begin{cases}
		\GJ(\vartheta,\bv,\psi)
		&\text{if }\Lambda\in\Ac,\\[1mm]
		+\infty &\text{otherwise} .
	\end{cases}
\end{equation}
An element $\Lambda\in\fX$ with $\Jc(\Lambda)<\infty$ is called
admissible. For admissible $\Lambda$ the value $\GJ(\vartheta,\bv,\psi)$
is finite, since $\psi\in\fC\subset L^2(S;H^1(\O))$. By
Theorems~\ref{ThHomogenized} and~\ref{ThmUniq}, the admissible
quadruples are exactly $(\th,\Sc(\th))$ with $\th\in L^2(\O_T)^3$; the
definition \eqref{GammaFunctional} refers to admissibility alone, which
keeps Theorem~\ref{ThmGamma}(ii) free of any uniqueness argument.
\begin{theorem}[Sequential $\Gamma$-convergence]\label{ThmGamma}
	Let $\O\subset\R^3$ be a bounded domain with $\partial\O\in C^{2,1}$.
	The following hold as $\e\to0$.
	\begin{enumerate}[label=\textup{(\roman*)}]
		\item \textup{(Equicoercivity.)} If $\Jc^\e(\Lambda^\e)\le M$ for some
		$M\ge0$ and every $\e$ in a sequence tending to $0$, then every
		subsequence has a further subsequence with $\Lambda^\e\ocarrow\Xi$ for
		some $\Xi\in\Ac$.
		\item \textup{($\liminf$ inequality.)} If $\Lambda^\e\ocarrow\Lambda$ and
		$\liminf_{\e\to0}\Jc^\e(\Lambda^\e)<\infty$, then $\Lambda\in\Ac$ and
		\[
		\Jc(\Lambda)\le\liminf_{\e\to0}\Jc^\e(\Lambda^\e) .
		\]
		\item \textup{($\limsup$ inequality.)} For every $\Xi\in\Ac$ there exists a
		family $\{\Xi^\e\}_\e$ with $\Xi^\e\in\Ac^\e$ such that
		\[
		\Xi^\e\ocarrow\Xi
		\qquad\text{and}\qquad
		\lim_{\e\to0}\Jc^\e(\Xi^\e)=\Jc(\Xi) .
		\]
	\end{enumerate}
	In particular, $\{\Jc^\e\}$ is equicoercive and $\Gamma$-converges to $\Jc$
	with respect to $\ocarrow$, in the sequential sense.
\end{theorem}

Parts (i) and (ii) follow from Section~\ref{SSecHomState} and use no uniqueness.
Part (iii) is the only point of the paper at which uniqueness is used, and it is
uniqueness for the homogenized system (Theorem~\ref{ThmUniq}): it identifies the
limit of the states obtained from the restricted controls.

\begin{theorem}[Homogenization of the optimal control problem]\label{ThmOCPHomo}
	Let $\O\subset\R^3$ be a bounded domain with $\partial\O\in C^{2,1}$ and for every $\e\in(0,\e_0)$, 
	$\hat\Xi_\e:=(\hat\th^\e,\hat\Gu^\e,\hat\phi^\e,\hat\mu^\e)\in\Ac^\e$ be
	optimal for \eqref{OPeps}. Then:
	\begin{enumerate}[label=\textup{(\roman*)}]
		\item every subsequence of $\{\hat\Xi_\e\}$ has a further subsequence
		converging, with respect to $\ocarrow$, to some $\Xi\in\Ac$;
		\item the values converge,
		\begin{equation}\label{MinConvergence}
			\m_\e\longrightarrow\m \qquad\text{as }\e\to0 ;
		\end{equation}
		\item every cluster point $\Xi=(\th,\Gu,\phi,\mu)$ of $\{\hat\Xi_\e\}$
		with respect to $\ocarrow$ is optimal for \eqref{OPlim}, that is,
		$\Jc(\Xi)=\m$. In particular, the infimum \eqref{InfLim} is attained;
		\item if $\hat\Xi_\e\ocarrow\Xi=(\th,\Gu,\phi,\mu)$ along a
		subsequence, then, along that subsequence,
		\begin{equation}\label{StrongControlConvergence}
			\overline{\hat\th^\e}\longrightarrow\th
			\qquad\text{strongly in }L^2(\O_T)^3 .
		\end{equation}
	\end{enumerate}
\end{theorem}

\begin{remark}\label{RemNonUnique}
	Theorem~\ref{ThmOCPHomo} shows that every cluster point of microscopic
	optimal pairs is optimal for \eqref{OPlim}. The converse, that every
	optimal pair of \eqref{OPlim} arises as such a limit, is not asserted:
	the recovery sequence of Theorem~\ref{ThmGamma}(iii) consists of
	admissible, not optimal, pairs for \eqref{OPeps}, and $\Gamma$-convergence
	gives no more. If \eqref{OPlim} has a unique solution, the subsequence
	argument of Step~3 of the proof of Theorem~\ref{ThmGamma} shows that the
	convergences of Theorem~\ref{ThmOCPHomo} hold along the whole family.
	
	Uniqueness cannot be expected in general. The admissible set of
	\eqref{OPlim} is a graph by Theorem \ref{ThCritical} and \ref{ThmUniq}, so the reduced
	functional \eqref{ReducedOP} is well defined, but $\Sc$ is nonlinear:
	the state depends on the control through the capillary term
	$\lambda\phi\nabla\mu$ and the transport term $\Gu\cdot\nabla\phi$, both
	quadratic in the state, and through $F'$. Hence the reduced functional
	need not be convex, and the Tikhonov term does not by itself ensure a
	unique minimizer.
\end{remark}

\begin{remark}\label{RemNoAdjoint}
	Theorem~\ref{ThmOCPHomo} compares the two problems at the level of minimum
	values and minimizers, not of optimality systems. By
	Theorems~\ref{ThHomogenized} and~\ref{ThmUniq} the homogenized problem is
	reduced, so a first-order optimality system for \eqref{OPlim} could be
	derived once the differentiability of $\Sc$ is established. We do not pursue
	this: by the discussion in Section~\ref{SSecMicroOCP}, no optimality system is
	available at the level $\e$, so such a system could not be obtained as a
	limit.
\end{remark}

\section{Preliminary results and a priori estimates}\label{Sec04}

\subsection{Extension results and uniform inequalities}
The passage to the limit uses the extension operators constructed in this
subsection and the uniform functional inequalities they yield.

\begin{lemma}[Extension on the primal level]\label{Emain}
	There exist $\e_0\in(0,1)$ and $C>0$, independent of $\e\in(0,\e_0)$, such
	that, for every $\e\in(0,\e_0)$, there is a bounded linear operator
	$E^\e:H^1(\O_p^\e)\longrightarrow H^1(\O)$ such that
	\[
	E^\e u=u \quad\text{a.e.\ in }\O_p^\e .
	\]
	Moreover,
	\begin{equation}\label{eq:extension-primal-estimates}
		\|E^\e u\|_{L^2(\O)}\leq C\|u\|_{L^2(\O_p^\e)},
		\qquad
		\|\nabla E^\e u\|_{L^2(\O)}\leq C\|\nabla u\|_{L^2(\O_p^\e)}
	\end{equation}
	for every $u\in H^1(\O_p^\e)$. Furthermore, if $u=0$ on $\partial\O$ in the
	sense of traces, then $E^\e u\in H^1_0(\O)$. The same operator, applied
	componentwise, extends vector fields, and there exists a further operator
	$\mathscr{E}^\e$ with the same properties that satisfies, in addition, the
	symmetric-gradient bound
	\begin{equation}\label{eq:extension-korn}
		\|D(\mathscr{E}^\e\bv)\|_{L^2(\O)}
		\leq C\|D(\bv)\|_{L^2(\O_p^\e)}
		\qquad\text{for every }\bv\in H^1(\O_p^\e)^3 .
	\end{equation}
\end{lemma}
The operator $E^\e$ is constructed at the scale of the obstacles, following
\cite{CioranescuSJP1979,ConcaDonato1988}; see \cite[Lemma~3.1]{LuPokorny2021}
for this construction in domains with tiny holes. Fix $R>R_{\Oc}$, so that
$\overline{\Oc}\subset B_R$. Decreasing $\e_0$ if necessary, we may assume
$2Ra_\e<c_0\e$ for $\e<\e_0$; then, by \eqref{eq:obstacle-separation}, the
balls $B_{2Ra_\e}(x_{\e,\kappa})$, $\kappa\in\Kc_\e$, are mutually disjoint and
contained in $\O$ (see Figure~\ref{fig:solenoidal-restriction}). On each ball,
$u$ is written as its mean over
$B_{2Ra_\e}(x_{\e,\kappa})\setminus\overline{\Oc_{\e,\kappa}}$ plus a function of
zero mean; the latter is rescaled to the fixed perforated ball
$B_{2R}\setminus\overline{\Oc}$, extended there, and rescaled back. This yields,
with $C$ depending only on $\Oc$ and $R$,
\[
\|\nabla E^\e u\|_{L^2(\Oc_{\e,\kappa})}\le C\|\nabla u\|_{L^2(B_{2Ra_\e}(x_{\e,\kappa})\setminus \Oc_{\e,\kappa})},
\qquad
\|E^\e u\|_{L^2(\Oc_{\e,\kappa})}\le C\|u\|_{L^2(B_{2Ra_\e}(x_{\e,\kappa})\setminus \Oc_{\e,\kappa})},
\]
and summing over $\kappa$ gives \eqref{eq:extension-primal-estimates};
moreover, $E^\e u=u$ near $\partial\O$. In \cite{LuPokorny2021} the holes are
of class $C^2$; this is used only to obtain an extension operator on the
rescaled perforated ball uniformly with respect to different hole shapes. Here
all obstacles are homothetic to the single Lipschitz set $\Oc$, and on the fixed
Lipschitz domain $B_{2R}\setminus\overline{\Oc}$ an extension operator bounded
in $L^2$ and in $H^1$ exists \cite{AdamsFournier03}. The operator
$\mathscr E^\e$ is obtained in the same way, with the mean replaced by the
$L^2(B_{2Ra_\e}(x_{\e,\kappa})\setminus \Oc_{\e,\kappa})$-projection onto rigid
displacements and the Poincar\'e--Wirtinger inequality replaced by Korn's
inequality on $B_{2R}\setminus\overline{\Oc}$ \cite{oleinik01}. Rigid
displacements are invariant under the rescaling, which gives
\eqref{eq:extension-korn} with a constant independent of $\e$.

\begin{lemma}[Scalar fields]
	\label{lem:scalar-two-extensions}
	Let $v^\e\in\fC^\e$ satisfy
	\[
	\|v^\e\|_{L^\infty(S;H^1(\O_p^\e))}
	+\|\partial_tv^\e\|_{L^2(S;H^{-1}(\O_p^\e))}\leq C .
	\]
	Define the primal extension $\wt v^\e:=E^\e v^\e$ and the zero extension
	$\overline v^\e$ to $S\X\O$. Then
	\begin{equation}\label{ScalarPrimalBound}
		\|\wt v^\e\|_{L^\infty(S;H^1(\O))}\leq C,
		\qquad
		\overline v^\e\in H^1(S;H^{-1}(\O)),
		\qquad
		\|\partial_t\overline v^\e\|_{L^2(S;H^{-1}(\O))}\leq C .
	\end{equation}
	Moreover,
	\begin{equation}\label{ScalarExtensionsDifference}
		\|\wt v^\e-\overline v^\e\|_{L^\infty(S;L^2(\O))}
		\leq C\frac{a_\e}{\e} .
	\end{equation}
	Consequently, there exist $v\in\fC$ and a subsequence such that
	\begin{equation}\label{ScalarTwoExtensionStrongCompactness}
		\wt v^\e\longrightarrow v,
		\qquad
		\overline v^\e\longrightarrow v
		\qquad\text{strongly in }C([0,T];L^2(\O)).
	\end{equation}
\end{lemma}

\begin{proof}
	The primal bound in \eqref{ScalarPrimalBound}$_1$ is
	\eqref{eq:extension-primal-estimates} applied for a.e.\ $t\in S$.
	
	The time derivative of the zero extension is obtained by restriction. For
	$\eta\in H^1(\O)$, set
	\[
	\bigl\langle\partial_t\overline v^\e,\eta\bigr\rangle_{H^{-1}(\O),H^1(\O)}
	:=
	\bigl\langle\partial_tv^\e,\eta|_{\O_p^\e}\bigr\rangle
	_{H^{-1}(\O_p^\e),H^1(\O_p^\e)} .
	\]
	Since $\O_s^\e$ is time-independent, this functional is the distributional
	derivative of $\overline v^\e$. From
	$\|\eta|_{\O_p^\e}\|_{H^1(\O_p^\e)}\leq\|\eta\|_{H^1(\O)}$ we obtain
	\[
	\|\partial_t\overline v^\e\|_{H^{-1}(\O)}
	\leq\|\partial_tv^\e\|_{H^{-1}(\O_p^\e)},
	\]
	which gives \eqref{ScalarPrimalBound}$_3$. Since the zero extension is an
	isometry from $L^2(\O_p^\e)$ into $L^2(\O)$, we also have
	$\overline v^\e\in L^\infty(S;L^2(\O))\subset L^2(S;H^{-1}(\O))$, and hence
	$\overline v^\e\in H^1(S;H^{-1}(\O))$, which is
	\eqref{ScalarPrimalBound}$_2$.
	
	Since $\wt v^\e=v^\e$ in $\O_p^\e$, the difference $\wt v^\e-\overline v^\e$
	vanishes in $\O_p^\e$ and equals $\wt v^\e$ in $\O_s^\e$. Hence, by H\"older's
	inequality with $\tfrac12=\tfrac13+\tfrac16$ and the Sobolev embedding
	$H^1(\O)\hookrightarrow L^6(\O)$,
	\[
	\|\wt v^\e-\overline v^\e\|_{L^\infty(S;L^2(\O))}
	\leq|\O_s^\e|^{1/3}\|\wt v^\e\|_{L^\infty(S;L^6(\O))}
	\leq C|\O_s^\e|^{1/3} .
	\]
	By \eqref{eq:obstacle-volume}, $|\O_s^\e|^{1/3}\leq C\e^{\alpha-1}=Ca_\e/\e$,
	which is \eqref{ScalarExtensionsDifference}.
	
	By the standard Hilbert-triple result, $v^\e\in C([0,T];L^2(\O_p^\e))$; since
	both extension operators are time-independent and bounded on the corresponding
	$L^2$ spaces, $\wt v^\e,\overline v^\e\in C([0,T];L^2(\O))$. Set
	\[
	d_\e:=\|\wt v^\e-\overline v^\e\|_{L^\infty(S;L^2(\O))},
	\]
	so that $d_\e\to0$ by \eqref{ScalarExtensionsDifference}. For $t,s\in[0,T]$,
	the interpolation inequality
	$\|q\|^2_{L^2(\O)}\leq\|q\|_{H^1(\O)}\|q\|_{H^{-1}(\O)}$ applied to
	$q=\wt v^\e(t)-\wt v^\e(s)$ gives, using the uniform $H^1$ bound in
	\eqref{ScalarPrimalBound},
	\begin{align*}
		\|\wt v^\e(t)-\wt v^\e(s)\|^2_{L^2(\O)}
		&\leq C\|\wt v^\e(t)-\wt v^\e(s)\|_{H^{-1}(\O)}
		\leq C\bigl(\|\overline v^\e(t)-\overline v^\e(s)\|_{H^{-1}(\O)}
		+2d_\e\bigr)\\
		&\leq C\bigl(|t-s|^{1/2}+2d_\e\bigr),
	\end{align*}
	the last step by the time-derivative bound in \eqref{ScalarPrimalBound} and
	H\"older's inequality in time. Thus $\{\wt v^\e\}$ is asymptotically uniformly
	equicontinuous in $L^2(\O)$: given $\gamma>0$, choose $\e_1$ with
	$2Cd_\e<\gamma^2/2$ for $\e<\e_1$ and then $|t-s|$ small enough that
	$C|t-s|^{1/2}<\gamma^2/2$; along the sequence $\e\to0$, only finitely many
	indices satisfy $\e\geq\e_1$, and these are equicontinuous individually. Since
	$\{\wt v^\e\}$ is uniformly bounded in $L^\infty(S;H^1(\O))$ and
	$H^1(\O)\Subset L^2(\O)$, the Arzel\`a--Ascoli theorem gives, after
	extraction,
	\[
	\wt v^\e\longrightarrow v \qquad\text{strongly in }C([0,T];L^2(\O)).
	\]
	The uniform bound in $L^\infty(S;H^1(\O))$ also gives, after a further
	extraction, $\wt v^\e\overset{*}{\rightharpoonup}v$ in
	$L^\infty(S;H^1(\O))$, whence $v\in\fC$. The convergence of
	$\overline v^\e$ follows from \eqref{ScalarExtensionsDifference}.
\end{proof}

\begin{lemma}[Solenoidal fields]\label{LemD02}
	There exist $\e_0\in(0,1)$ and $C>0$, independent of $\e\in(0,\e_0)$, such that
	for every $\e\in(0,\e_0)$ there is a linear operator
	$\Ec^\e:\GH^1_{\div}(\O_p^\e)\to H^1_{0,\div}(\O)^3$ with
	\begin{equation}\label{SolenoidalPrimalExtensionIdentity}
		(\Ec^\e v)|_{\O_p^\e}=v\qquad\forall\,v\in \GH^1_{\div}(\O_p^\e),
	\end{equation}
	\begin{equation}\label{SolenoidalPrimalExtensionBounds}
		\|\Ec^\e v\|_{L^2(\O)}
		\leq C\bigl(\|v\|_{L^2(\O_p^\e)}+a_\e\|\nabla v\|_{L^2(\O_p^\e)}\bigr),
		\qquad
		\|\nabla\Ec^\e v\|_{L^2(\O)}\leq C\|\nabla v\|_{L^2(\O_p^\e)},
	\end{equation}
	and, in particular,
	\begin{equation}\label{SolenoidalPrimalExtensionH1Bound}
		\|\Ec^\e v\|_{H^1(\O)}\leq C\|v\|_{H^1(\O_p^\e)} .
	\end{equation}
	For $v\in \GH^1_{\div}(\O_p^\e)$, let $\overline v$ be its zero extension
	through the obstacles. Then
	\begin{equation}\label{SolenoidalZeroExtension}
		\|\overline v\|_{L^2(\O)}=\|v\|_{L^2(\O_p^\e)},
		\qquad
		\div\,\overline v=0\ \text{ in }\Dc'(\O) .
	\end{equation}
	Moreover, for every $v^\e\in L^2(S;\GH^1_{\div}(\O_p^\e))$,
	\begin{equation}\label{PrimalZeroExtensionDifference}
		\|\Ec^\e v^\e-\overline v^\e\|_{L^2(S\X\O)}
		\leq C\,\frac{a_\e}{\e}\,\|v^\e\|_{L^2(S;H^1(\O_p^\e))} ,
	\end{equation}
	so that if $\sup_{\e\in(0,\e_0)}\|v^\e\|_{L^2(S;H^1(\O_p^\e))}<\infty$, then
	$\Ec^\e v^\e-\overline v^\e\to0$ strongly in $L^2(S\X\O)^3$.
\end{lemma}

\begin{proof}
	\textbf{Step 1: A Bogovski\u{\i} operator on the obstacles.}
	Recall that $\Oc_{\e,\kappa}=x_{\e,\kappa}+a_\e\Oc$ with
	$x_{\e,\kappa}=\e(\kappa+y_0)$ and $a_\e=\e^\alpha$. Let
	\[
	L^2_{0,\mathrm{comp}}(\O_s^\e)
	:=\Bigl\{g\in L^2(\O_s^\e):
	\int_{\Oc_{\e,\kappa}}g\,\md x=0\ \text{for every }\kappa\in\Kc_\e\Bigr\}
	\]
	be the space of functions with zero mean on each connected component of
	$\O_s^\e$; then $L^2_{0,\mathrm{comp}}(\O_s^\e)\subset L^2_0(\O_s^\e)$. We
	claim that there is a linear operator
	\[
	\Bc_s^\e:L^2_{0,\mathrm{comp}}(\O_s^\e)\longrightarrow H^1_0(\O_s^\e)^3
	\]
	such that $\Psi:=\Bc_s^\e(g)$ satisfies $\div\Psi=g$ in $\O_s^\e$ and
	$\Psi=0$ on $\partial\O_s^\e$, together with, for some $C_s>0$ independent of
	$\e$, $\kappa$ and $g$,
	\begin{equation}\label{Bogov02+-gradient}
		\|\nabla\Bc_s^\e(g)\|_{L^2(\O_s^\e)}\leq C_s\|g\|_{L^2(\O_s^\e)},
		\qquad
		\|\Bc_s^\e(g)\|_{L^2(\O_s^\e)}\leq C_s a_\e\|g\|_{L^2(\O_s^\e)} .
	\end{equation}
	Indeed, let $\Bc_{\Oc}:L^2_0(\Oc)\to H^1_0(\Oc)^3$ be a fixed Bogovski\u{\i}
	operator on the reference obstacle, satisfying $\div_z\Bc_{\Oc}(f)=f$ and
	$\|\Bc_{\Oc}(f)\|_{H^1(\Oc)}\leq C\|f\|_{L^2(\Oc)}$; such an operator exists
	since $\Oc$ is a bounded Lipschitz domain \cite{BogovM01,galdi2011}. For
	$\kappa\in\Kc_\e$ set $g_{\e,\kappa}(z):=g(x_{\e,\kappa}+a_\e z)$,
	$z\in\Oc$, which belongs to $L^2_0(\Oc)$ by the componentwise mean-zero
	condition, and define, for $x\in \Oc_{\e,\kappa}$,
	\[
	\Bc^\e_s(g)(x)
	:=a_\e\,\Bc_{\Oc}(g_{\e,\kappa})
	\Bigl(\frac{x-x_{\e,\kappa}}{a_\e}\Bigr) .
	\]
	This is linear, maps into $H^1_0(\O_s^\e)^3$, and satisfies
	$\div\Bc^\e_s(g)=g$ in $\O_s^\e$. A change of variables on each
	obstacle turns the reference estimate into
	\[
	\|\nabla\Bc^\e_s(g)\|_{L^2(\Oc_{\e,\kappa})}
	\leq C\|g\|_{L^2(\Oc_{\e,\kappa})},
	\qquad
	\|\Bc^\e_s(g)\|_{L^2(\Oc_{\e,\kappa})}
	\leq Ca_\e\|g\|_{L^2(\Oc_{\e,\kappa})},
	\]
	with $C$ depending only on $\Oc$, since all obstacles are homothetic copies of
	$\Oc$ with the same ratio $a_\e$. Squaring and summing over
	$\kappa\in\Kc_\e$, which is legitimate because the obstacles are mutually
	disjoint, yields \eqref{Bogov02+-gradient} with $C_s$ independent of $\e$.
	
	\smallskip
	\textbf{Step 2: The solenoidal extension.}
	Let $v\in \GH^1_{\div}(\O_p^\e)$ and set $w:=(E^\e v_1,E^\e v_2,E^\e v_3)$,
	with $E^\e$ the scalar extension of Lemma~\ref{Emain}. Then
	$w|_{\O_p^\e}=v$, $w=0$ on $\partial\O$, and
	\begin{equation}\label{VectorScalarExtensionEstimates}
		\|w\|_{L^2(\O)}\leq C\|v\|_{L^2(\O_p^\e)},
		\qquad
		\|\nabla w\|_{L^2(\O)}\leq C\|\nabla v\|_{L^2(\O_p^\e)} .
	\end{equation}
	Set $g:=(\div w)|_{\O_s^\e}$. For each $\kappa\in\Kc_\e$ the traces of $w$
	from the solid side and of $v$ from the fluid side agree on
	$\partial \Oc_{\e,\kappa}$, so with $\Gn_s^\e=-\Gn^\e$ and the impermeability
	condition $v\cdot\Gn^\e=0$ on $\Gamma_s^\e$,
	\[
	\int_{\Oc_{\e,\kappa}}g\,\md x
	=\int_{\partial \Oc_{\e,\kappa}}w\cdot\Gn_s^\e\,\md S
	=-\int_{\partial \Oc_{\e,\kappa}}v\cdot\Gn^\e\,\md S=0 ,
	\]
	hence $g\in L^2_{0,\mathrm{comp}}(\O_s^\e)$. This is the first use of the
	impermeability condition in the definition of $\GH^1_{\div}(\O_p^\e)$, and it
	is what makes the Bogovski\u{\i} correction of Step~1 available.
	
	Let $z_\e:=\Bc_s^\e(g)\in H^1_0(\O_s^\e)^3$ and extend $z_\e$ by zero
	to $\O$; since $\O_s^\e\Subset\O$ and $z_\e$ has zero trace on
	$\partial\O_s^\e$, this zero extension belongs to $H^1_0(\O)^3$. The bounds
	\eqref{Bogov02+-gradient} together with
	$\|g\|_{L^2(\O_s^\e)}\leq C\|\nabla w\|_{L^2(\O_s^\e)}$ give
	\begin{equation}\label{SolidCorrectionH1Bounds}
		\|\nabla z_\e\|_{L^2(\O_s^\e)}\leq C\|\nabla w\|_{L^2(\O_s^\e)},
		\qquad
		\|z_\e\|_{L^2(\O_s^\e)}\leq Ca_\e\|\nabla w\|_{L^2(\O_s^\e)} .
	\end{equation}
	Define $\Ec^\e v:=w-z_\e$. Since $z_\e=0$ in $\O_p^\e$, we have
	$(\Ec^\e v)|_{\O_p^\e}=v$, which is
	\eqref{SolenoidalPrimalExtensionIdentity}, and
	\[
	\div(\Ec^\e v)=
	\begin{cases}
		\div v=0, & \text{in }\O_p^\e,\\
		g-\div z_\e=0, & \text{in }\O_s^\e .
	\end{cases}
	\]
	Since $w,z_\e\in H^1(\O)^3$, we have $\Ec^\e v\in H^1(\O)^3$, so the two
	identities above imply $\div(\Ec^\e v)=0$ in $\Dc'(\O)$; no interface term
	arises, because $\Ec^\e v$ has no jump across $\Gamma_s^\e$. Moreover,
	$\Ec^\e v=0$ on $\partial\O$, since $w=0$ there and $z_\e$ is supported in
	$\O_s^\e\Subset\O$. Thus $\Ec^\e v\in H^1_{0,\div}(\O)^3$, and $\Ec^\e$ is
	linear by the linearity of $E^\e$, of the divergence, and of $\Bc_s^\e$.
	Combining \eqref{VectorScalarExtensionEstimates} and
	\eqref{SolidCorrectionH1Bounds},
	\[
	\|\Ec^\e v\|_{L^2(\O)}\leq\|w\|_{L^2(\O)}+\|z_\e\|_{L^2(\O_s^\e)}
	\leq C\bigl(\|v\|_{L^2(\O_p^\e)}+a_\e\|\nabla v\|_{L^2(\O_p^\e)}\bigr),
	\qquad
	\|\nabla\Ec^\e v\|_{L^2(\O)}\leq C\|\nabla v\|_{L^2(\O_p^\e)},
	\]
	which is \eqref{SolenoidalPrimalExtensionBounds}; adding these and using
	$a_\e\leq1$ gives \eqref{SolenoidalPrimalExtensionH1Bound}.
	
	\smallskip
	\textbf{Step 3: The zero extension.}
	The norm identity in \eqref{SolenoidalZeroExtension} is immediate. For
	$\varphi\in\Dc(\O)$, integration by parts on $\O_p^\e$ gives
	\[
	\int_\O\overline v\cdot\nabla\varphi\,\md x
	=\int_{\O_p^\e}v\cdot\nabla\varphi\,\md x
	=-\int_{\O_p^\e}(\div v)\varphi\,\md x
	+\int_{\Gamma_s^\e}(v\cdot\Gn^\e)\varphi\,\md S
	=0 ,
	\]
	both terms vanishing by $\div v=0$ and $v\cdot\Gn^\e=0$ on $\Gamma_s^\e$; the
	contribution of $\partial\O$ is absent because $\varphi$ has compact support.
	This is \eqref{SolenoidalZeroExtension}$_2$, and it is the second use of the
	impermeability condition: without it, the zero extension would fail to be
	divergence-free across $\Gamma_s^\e$.
	
	Finally, since $\Ec^\e v^\e=v^\e$ in $\O_p^\e$, the difference
	$\Ec^\e v^\e-\overline v^\e$ vanishes on $\O_p^\e$ and equals $\Ec^\e v^\e$ on
	$\O_s^\e$. By H\"older's inequality with $\tfrac12=\tfrac13+\tfrac16$, the
	Sobolev embedding $H^1(\O)\hookrightarrow L^6(\O)$ and
	\eqref{SolenoidalPrimalExtensionH1Bound},
	\[
	\|\Ec^\e v^\e-\overline v^\e\|_{L^2(S\X\O)}
	=\|\Ec^\e v^\e\|_{L^2(S\X\O_s^\e)}
	\leq|\O_s^\e|^{1/3}\|\Ec^\e v^\e\|_{L^2(S;L^6(\O))}
	\leq C|\O_s^\e|^{1/3}\|v^\e\|_{L^2(S;H^1(\O_p^\e))} .
	\]
	By \eqref{eq:obstacle-volume}, $|\O_s^\e|^{1/3}\leq Ca_\e/\e$, which gives
	\eqref{PrimalZeroExtensionDifference}. As $a_\e/\e=\e^{\alpha-1}\to0$, the
	stated convergence follows.
\end{proof}

\begin{lemma}
	\label{lem:uniform-ineq}
	There exist $\e_0\in(0,1)$ and $C>0$, independent of $\e\in(0,\e_0)$, such
	that for every $\e\in(0,\e_0)$, every $\bv\in\GH^1(\O_p^\e)$, every
	$u\in H^1(\O^\e_p)$ and every $r\in[2,6]$,
	\begin{alignat}{2}
		\|\bv\|_{L^2(\O_p^\e)}&\leq C\|\nabla\bv\|_{L^2(\O_p^\e)},
		&\qquad
		\|\bv\|_{L^2(\O_p^\e)}+\|\nabla\bv\|_{L^2(\O_p^\e)}
		&\leq C\|D(\bv)\|_{L^2(\O_p^\e)},
		\label{eq:Poincare-Korn}\\
		\|u\|_{L^2(\O^\e_p)}&\leq C\|\nabla u\|_{L^2(\O^\e_p)}
		\quad\text{if }u\in L^2_0(\O^\e_p),
		&\qquad
		\|u\|_{L^r(\O^\e_p)}&\leq C\|u\|_{H^1(\O^\e_p)} .
		\label{eq:PW-Sobolev}
	\end{alignat}
\end{lemma}

\begin{proof}
	For \eqref{eq:Poincare-Korn}$_1$ and \eqref{eq:PW-Sobolev}$_2$, we extend by
	$E^\e$ and apply the corresponding inequality on the fixed domain $\O$.
	For $\bv\in\GH^1(\O_p^\e)$ we have $E^\e\bv\in H^1_0(\O)^3$ by
	Lemma~\ref{Emain}, since $\bv=0$ on $\partial\O$, so Poincar\'e's inequality
	on $\O$ and \eqref{eq:extension-primal-estimates} give
	\[
	\|\bv\|_{L^2(\O_p^\e)}\leq\|E^\e\bv\|_{L^2(\O)}
	\leq C\|\nabla E^\e\bv\|_{L^2(\O)}
	\leq C\|\nabla\bv\|_{L^2(\O_p^\e)} ,
	\]
	and similarly for the Sobolev embedding. For \eqref{eq:PW-Sobolev}$_1$, note
	that $E^\e u$ need not have zero mean on $\O$. Writing
	$m_\e:=|\O|^{-1}\int_\O E^\e u\,\md x$ and applying the Poincar\'e--Wirtinger
	inequality to $E^\e u-m_\e$ gives
	$\|E^\e u-m_\e\|_{L^2(\O)}\leq C\|\nabla u\|_{L^2(\O_p^\e)}$, while
	$\int_{\O_p^\e}u\,\md x=0$ yields
	$|m_\e|\,|\O_p^\e|\leq\|E^\e u-m_\e\|_{L^2(\O)}|\O|^{1/2}$. Since
	$|\O_p^\e|\geq\tfrac12|\O|$ for $\e$ small by \eqref{eq:obstacle-volume},
	the two estimates combine to give \eqref{eq:PW-Sobolev}$_1$.
	
	The uniform Korn inequality \eqref{eq:Poincare-Korn}$_2$ does not follow from
	$E^\e$, which is bounded with respect to the full gradient only. It follows
	from the extension operator $\mathscr{E}^\e$ of \eqref{eq:extension-korn}.
	See
	\cite{slip1991allairehomogenization,LuUniformStokes21,oleinik01,Ext01,Ext02}
	for these and related inequalities on periodically perforated domains.
\end{proof}

\subsection{A solenoidal restriction operator}
\label{sec:test_functions}

A fixed field $\zeta\in C^\infty_{c,\div}(\O)^3$ is not admissible in the
momentum equation \eqref{weak-momentum1}: its restriction to $\O_p^\e$ is
divergence-free and vanishes on $\partial\O$, but it does not satisfy the
impermeability condition $\zeta\cdot\Gn^\e=0$ on $\Gamma_s^\e$, so it does not
belong to $\GH^1_{\div}(\O_p^\e)$. The following lemma provides a solenoidal
correction of $\zeta$ that vanishes near every obstacle and converges strongly
to $\zeta$ in $H^1(\O)^3$.

\begin{figure}[ht]
	\centering
	\begin{tikzpicture}[scale=1.0,>=latex]
		\draw[thick] (-2.6,-2.6) rectangle (2.6,2.6);
		\node at (-1.15,2.25) {$\e(\kappa+Y)$};
		
		\fill[gray!12] (0,0) circle (2.0);
		\draw[thick] (0,0) circle (2.0);
		
		\fill[white] (0,0) circle (1.1);
		\draw[thick] (0,0) circle (1.1);
		
		\fill[pattern=north east lines, pattern color=black!60]
		plot[smooth cycle,tension=0.9]
		coordinates {(-0.42,0.10) (-0.16,0.44) (0.34,0.32)
			(0.46,-0.12) (0.10,-0.44) (-0.34,-0.30)};
		\draw[thick]
		plot[smooth cycle,tension=0.9]
		coordinates {(-0.42,0.10) (-0.16,0.44) (0.34,0.32)
			(0.46,-0.12) (0.10,-0.44) (-0.34,-0.30)};
		
		\fill (0,0) circle (1.1pt);
		
		\draw[->] (0,0) -- (78:1.1);
		\node at (66:0.72) {\footnotesize $Ra_\e$};
		\draw[->] (0,0) -- (140:2.0);
		\node at (146:1.35) {\footnotesize $2Ra_\e$};
		
		\node at (0,-0.85) {\footnotesize $\psi^\e=0$};
		\node[align=center] at (0,1.55)
		{\footnotesize $\psi^\e=\eta_{\e,\kappa}\zeta-c_{\e,\kappa}$};
		\node at (1.75,-2.15) {\footnotesize $\psi^\e=\zeta$};
		
		\draw[->,shorten >=2pt] (3.55,1.15) -- (1.55,0.45);
		\node[anchor=west] at (3.55,1.15) {\footnotesize $A_{\e,\kappa}$};
		
		\draw[->,shorten >=2pt] (3.55,0.15) -- (0.50,0.05);
		\node[anchor=west] at (3.55,0.15) {\footnotesize $\Oc_{\e,\kappa}$};
		
		\draw[->,shorten >=2pt] (3.55,-0.85) -- (0.06,-0.06);
		\node[anchor=west] at (3.55,-0.85) {\footnotesize $x_{\e,\kappa}$};
	\end{tikzpicture}
	\caption{Cross-section of the cell $\e(\kappa+Y)$ used in the construction
		of $\Rc^\e$ (not to scale: $a_\e\ll\e$).}
	\label{fig:solenoidal-restriction}
\end{figure}

\begin{lemma}[Subcritical solenoidal restriction]
	\label{lem:solenoidal-restriction}
	There exist $\e_0\in(0,1)$ and, for every
	$\e\in(0,\e_0)$, a linear map
	\[
	\Rc^\e:C^\infty_{c,\div}(\O)^3\longrightarrow\GH^1_{\div}(\O_p^\e),
	\]
	independent of time, such that for every $\zeta\in C^\infty_{c,\div}(\O)^3$
	the field $\Rc^\e\zeta$ vanishes on a neighborhood of $\Gamma_s^\e$, its
	zero extension $\overline{\Rc^\e\zeta}$ belongs to $H^1_{0,\div}(\O)^3$, and
	\begin{equation}\label{SolenoidalRestrictionConvergence}
		\overline{\Rc^\e\zeta}\longrightarrow\zeta
		\quad\text{strongly in }H^1(\O)^3,
		\qquad
		\|\Rc^\e\zeta\|_{H^1(\O_p^\e)}\leq C_\zeta ,
	\end{equation}
	with $C_\zeta>0$ independent of $\e$.
\end{lemma}

\begin{proof}
	Fix $\zeta\in C^\infty_{c,\div}(\O)^3$ and $R>R_{\Oc}$, so that
	$\overline{\Oc}\subset B_R$. For $\kappa\in\Kc_\e$ recall
	$x_{\e,\kappa}=\e(\kappa+y_0)$ and set
	\[
	A_{\e,\kappa}
	:=
	B_{2Ra_\e}(x_{\e,\kappa})
	\setminus\overline{B_{Ra_\e}(x_{\e,\kappa})} ,
	\qquad
	\mathscr N_\e:=\bigcup_{\kappa\in\Kc_\e}B_{2Ra_\e}(x_{\e,\kappa}) .
	\]
	Since $a_\e/\e=\e^{\alpha-1}\to0$, there is $\e_0\in(0,1)$ such that
	$2Ra_\e<c_0\e$ for every $\e\in(0,\e_0)$, with $c_0$ as in
	\eqref{eq:obstacle-separation}. For such $\e$ the balls
	$B_{2Ra_\e}(x_{\e,\kappa})$, $\kappa\in\Kc_\e$, are mutually disjoint, are
	contained in their corresponding $\e$-cells, and do not meet $\partial\O$.
	Since $\#\Kc_\e\leq C\e^{-3}$,
	\begin{equation}\label{eq:small-modification-set}
		|\mathscr N_\e|\leq C\frac{a_\e^3}{\e^3}
		=C\e^{3(\alpha-1)} .
	\end{equation}
	
	\textbf{Step 1: The cutoff and its divergence.}
	Let $\eta\in C^\infty\bigl(\overline{B_{2R}\setminus B_R}\bigr)$ satisfy
	$0\leq\eta\leq1$, $\eta=0$ on $\partial B_R$ and $\eta=1$ on
	$\partial B_{2R}$, and set
	\[
	\eta_{\e,\kappa}(x):=\eta\Bigl(\frac{x-x_{\e,\kappa}}{a_\e}\Bigr),
	\]
	so that
	\begin{equation}\label{eq:subcritical-cutoff-gradient}
		\eta_{\e,\kappa}=0
		\ \text{ on }\partial B_{Ra_\e}(x_{\e,\kappa}),
		\qquad
		\eta_{\e,\kappa}=1
		\ \text{ on }\partial B_{2Ra_\e}(x_{\e,\kappa}),
		\qquad
		|\nabla\eta_{\e,\kappa}|\leq\frac{C}{a_\e} .
	\end{equation}
	The field $\eta_{\e,\kappa}\zeta$ has the required traces on the two
	boundaries of the annulus but is not divergence-free: since
	$\nabla\cdot\zeta=0$,
	\[
	\nabla\cdot(\eta_{\e,\kappa}\zeta)
	=\nabla\eta_{\e,\kappa}\cdot\zeta
	=:f_{\e,\kappa}
	\qquad\text{in }A_{\e,\kappa}.
	\]
	By the divergence theorem, $\nabla\cdot\zeta=0$ and the boundary values in
	\eqref{eq:subcritical-cutoff-gradient},
	\[
	\int_{A_{\e,\kappa}}f_{\e,\kappa}\,\md x
	=\int_{A_{\e,\kappa}}\nabla\cdot(\eta_{\e,\kappa}\zeta)\,\md x
	=\int_{\partial B_{2Ra_\e}(x_{\e,\kappa})}\zeta\cdot\bn\,\md S
	=\int_{B_{2Ra_\e}(x_{\e,\kappa})}\nabla\cdot\zeta\,\md x
	=0 ,
	\]
	so that $f_{\e,\kappa}\in L^2_0(A_{\e,\kappa})$.
	
	\textbf{Step 2: The local Bogovski\u{\i} correction.}
	Let
	$\mathcal B_{A_{\e,\kappa}}:L^2_0(A_{\e,\kappa})\to H^1_0(A_{\e,\kappa})^3$
	be the Bogovski\u{\i} operator obtained by translation and rescaling from a
	fixed Bogovski\u{\i} operator on $B_{2R}\setminus\overline{B_R}$; see, for
	example, \cite[Lemma~III.3.1]{galdi2011}. Set
	$c_{\e,\kappa}:=\mathcal B_{A_{\e,\kappa}}[f_{\e,\kappa}]$, so that
	$c_{\e,\kappa}\in H^1_0(A_{\e,\kappa})^3$ and
	$\nabla\cdot c_{\e,\kappa}=f_{\e,\kappa}$ in $A_{\e,\kappa}$. The operator
	is linear, and rescaling the estimate on the reference annulus gives, with
	$C$ independent of $\e$ and $\kappa$,
	\begin{equation}\label{eq:local-Bogovskii-estimate}
		\|\nabla c_{\e,\kappa}\|_{L^2(A_{\e,\kappa})}
		\leq C\|f_{\e,\kappa}\|_{L^2(A_{\e,\kappa})}
		\leq\frac{C}{a_\e}\|\zeta\|_{L^2(A_{\e,\kappa})},
	\end{equation}
	where the last inequality uses
	\eqref{eq:subcritical-cutoff-gradient}. The rescaled Poincar\'e inequality
	then gives
	\begin{equation}\label{eq:local-Bogovskii-L2-estimate}
		\|c_{\e,\kappa}\|_{L^2(A_{\e,\kappa})}
		\leq Ca_\e\|\nabla c_{\e,\kappa}\|_{L^2(A_{\e,\kappa})}
		\leq C\|\zeta\|_{L^2(A_{\e,\kappa})} .
	\end{equation}
	
	\textbf{Step 3: The corrected field.}
	Define $\psi^\e$ on $\O$ by (see Figure~\ref{fig:solenoidal-restriction})
	\begin{equation}\label{eq:defpsi}
		\psi^\e(x)
		:=
		\begin{cases}
			0,
			& x\in B_{Ra_\e}(x_{\e,\kappa}),\ \kappa\in\Kc_\e,
			\\[2mm]
			\eta_{\e,\kappa}(x)\zeta(x)-c_{\e,\kappa}(x),
			& x\in A_{\e,\kappa},\ \kappa\in\Kc_\e,
			\\[2mm]
			\zeta(x),
			& x\notin\mathscr N_\e .
		\end{cases}
	\end{equation}
	The definition is unambiguous because the balls are mutually disjoint. On
	the inner boundary of each annulus $\eta_{\e,\kappa}\zeta-c_{\e,\kappa}=0$,
	while on the outer boundary $\eta_{\e,\kappa}\zeta-c_{\e,\kappa}=\zeta$;
	thus the traces of the three pieces agree and $\psi^\e\in H^1_0(\O)^3$. In
	every annulus,
	\[
	\nabla\cdot\psi^\e
	=\nabla\eta_{\e,\kappa}\cdot\zeta
	+\eta_{\e,\kappa}\nabla\cdot\zeta
	-\nabla\cdot c_{\e,\kappa}
	=f_{\e,\kappa}-f_{\e,\kappa}=0 ,
	\]
	and $\psi^\e$ is divergence-free inside the inner balls and outside
	$\mathscr N_\e$ as well. Since the traces agree across the interfaces, no
	distributional surface term appears, and hence
	$\psi^\e\in H^1_{0,\div}(\O)^3$.
	
	Since $\overline{\Oc_{\e,\kappa}}\subset B_{Ra_\e}(x_{\e,\kappa})$ by the
	choice of $R$, the field $\psi^\e$ vanishes on a neighborhood of every
	obstacle, and in particular on a neighborhood of $\Gamma_s^\e$. We set
	\[
	\Rc^\e\zeta:=\psi^\e|_{\O_p^\e}
	\in H^1_{0,\div}(\O_p^\e)^3\subset\GH^1_{\div}(\O_p^\e),
	\]
	whose zero extension satisfies $\overline{\Rc^\e\zeta}=\psi^\e$ in $\O$, so
	that $\overline{\Rc^\e\zeta}\in H^1_{0,\div}(\O)^3$. The map
	$\zeta\mapsto\Rc^\e\zeta$ is linear by the linearity of the cutoff
	multiplication and of $\mathcal B_{A_{\e,\kappa}}$, and the construction
	involves no time variable.
	
	\textbf{Step 4: Convergence.}
	Since $\psi^\e=\zeta$ outside $\mathscr N_\e$, the difference
	$\psi^\e-\zeta$ is supported in $\mathscr N_\e$, where it equals
	$(\eta_{\e,\kappa}-1)\zeta-c_{\e,\kappa}$ on $A_{\e,\kappa}$ and $-\zeta$
	on the inner ball. The definition \eqref{eq:defpsi}, together with
	\eqref{eq:subcritical-cutoff-gradient}--\eqref{eq:local-Bogovskii-L2-estimate},
	$a_\e\leq1$ and the disjointness of the balls, gives
	\begin{equation}\label{eq:local-restriction-error}
		\|\psi^\e-\zeta\|_{H^1(\O)^3}
		\leq
		C\Bigl(
		\|\nabla\zeta\|_{L^2(\mathscr N_\e)}
		+\frac{1}{a_\e}\|\zeta\|_{L^2(\mathscr N_\e)}
		\Bigr).
	\end{equation}
	Since $\zeta$ and $\nabla\zeta$ are bounded, \eqref{eq:small-modification-set}
	gives
	\[
	\|\nabla\zeta\|_{L^2(\mathscr N_\e)}
	\leq\|\nabla\zeta\|_{L^\infty(\O)}|\mathscr N_\e|^{1/2}
	\leq C_\zeta\,\e^{3(\alpha-1)/2}\longrightarrow0
	\]
	and
	\[
	\frac{1}{a_\e}\|\zeta\|_{L^2(\mathscr N_\e)}
	\leq
	\frac{\|\zeta\|_{L^\infty(\O)}}{a_\e}|\mathscr N_\e|^{1/2}
	\leq
	C_\zeta\Bigl(\frac{a_\e}{\e^3}\Bigr)^{1/2}
	=C_\zeta\,\e^{(\alpha-3)/2}\longrightarrow0 ,
	\]
	the last convergence precisely because $\alpha>3$. Hence
	\begin{equation}\label{eq:subcritical-strong-restriction}
		\|\overline{\Rc^\e\zeta}-\zeta\|_{H^1(\O)^3}
		=\|\psi^\e-\zeta\|_{H^1(\O)^3}\longrightarrow0 ,
	\end{equation}
	which is \eqref{SolenoidalRestrictionConvergence}$_1$. The uniform bound
	follows, since
	\[
	\|\Rc^\e\zeta\|_{H^1(\O_p^\e)}
	\leq\|\psi^\e\|_{H^1(\O)^3}
	\leq\|\zeta\|_{H^1(\O)^3}+\|\psi^\e-\zeta\|_{H^1(\O)^3}
	\leq C_\zeta .
	\]
	This completes the proof.
\end{proof}

\begin{remark}[Role of $\alpha$ in the construction]
	\label{rem:alpha-role}
	The hypothesis $\alpha>3$ is used only in the last step of the proof, and
	only for the gradient. The construction and the $L^2$ convergence require
	only $\alpha>1$; the $H^1$ bound and the $H^1$ convergence require
	$\alpha\geq3$ and $\alpha>3$, respectively.
	
	\textbf{(i) Geometry and well-definedness: $\alpha>1$.} The disjointness of
	the balls $B_{2Ra_\e}(x_{\e,\kappa})$, their inclusion in the corresponding
	$\e$-cells and the fact that they do not meet $\partial\O$ follow from
	$2Ra_\e<c_0\e$, that is, from $a_\e/\e=\e^{\alpha-1}\to0$; this is the same
	type of condition as $\delta_\e\to0$ in the construction of the domain in
	Section~\ref{sec:domain}. Steps~1--3 use nothing further: the cutoff, the
	local Bogovski\u{\i} correction, the gluing and the conclusion
	$\psi^\e\in H^1_{0,\div}(\O)^3$ are cellwise statements, invariant under
	the rescaling $x\mapsto(x-x_{\e,\kappa})/a_\e$.
	
	\textbf{(ii) Strong $L^2$ convergence: $\alpha>1$.} On the inner ball
	$\psi^\e-\zeta=-\zeta$, and on the annulus
	$\psi^\e-\zeta=(\eta_{\e,\kappa}-1)\zeta-c_{\e,\kappa}$. The first two
	contributions are bounded pointwise by $\|\zeta\|_{L^\infty(\O)}$, and the
	third is bounded by $C\|\zeta\|_{L^2(A_{\e,\kappa})}$ by
	\eqref{eq:local-Bogovskii-L2-estimate}, the factor $a_\e$ of the rescaled
	Poincar\'e inequality cancelling the factor $a_\e^{-1}$ of
	\eqref{eq:subcritical-cutoff-gradient}. Hence, by
	\eqref{eq:small-modification-set},
	\begin{equation}\label{eq:restriction-L2-rate}
		\|\overline{\Rc^\e\zeta}-\zeta\|_{L^2(\O)^3}
		\leq C_\zeta|\mathscr N_\e|^{1/2}
		\leq C_\zeta\Bigl(\frac{a_\e}{\e}\Bigr)^{3/2}
		=C_\zeta\,\e^{3(\alpha-1)/2},
	\end{equation}
	with no condition beyond $\alpha>1$. This rate is better than the rate
	$a_\e/\e$ of the primal--zero extension gap in Lemma~\ref{LemD02}.
	
	\textbf{(iii) The gradient: $\alpha\geq3$ for boundedness, $\alpha>3$ for
		convergence.} Both gradient contributions on $\mathscr N_\e$, namely
	$\zeta\otimes\nabla\eta_{\e,\kappa}$ and $\nabla c_{\e,\kappa}$, are of
	size $a_\e^{-1}\|\zeta\|_{L^2(\mathscr N_\e)}$ by
	\eqref{eq:subcritical-cutoff-gradient} and
	\eqref{eq:local-Bogovskii-estimate}, whence
	\[
	\frac{1}{a_\e}\|\zeta\|_{L^2(\mathscr N_\e)}
	\leq C_\zeta\Bigl(\frac{a_\e}{\e^3}\Bigr)^{1/2}
	=C_\zeta\,\rho_\e^{1/2},
	\]
	with $\rho_\e$ the obstacle-scale parameter of
	\eqref{eq:subcritical-parameter}. Thus the family $\{\Rc^\e\zeta\}_\e$ is
	bounded in $H^1$ as soon as $\rho_\e\lesssim1$, that is, for $\alpha\geq3$,
	while \eqref{eq:subcritical-strong-restriction} requires $\rho_\e\to0$, that
	is, $\alpha>3$.
	
	Accordingly, the parameters $a_\e/\e$ of \eqref{eq:obstacle-volume} and
	$\rho_\e$ of \eqref{eq:subcritical-parameter} play different roles:
	$a_\e/\e$ governs every $L^2$-level statement of this section, whereas
	$\rho_\e$ governs the gradient alone.
\end{remark}


		\subsection{A priori estimates}\label{SSec51}
				This subsection contains the proof of Lemma~\ref{Est01}.
		\begin{proof}[Proof of Lemma \ref{Est01}]
			The proof is divided into six steps. Throughout, $C>0$ denotes a generic constant
			independent of $\e$, whose value may change from line to line. We shall use
			repeatedly that
			\begin{equation}\label{VolBound01}
				|\O^\e_s|\leq Ca_\e^3/\e^3,
				\qquad\text{hence}\qquad
				\tfrac12|\O|\le|\O^\e_p|\le|\O|
				\quad\text{for all $\e$ small enough,}
			\end{equation}
			so that in particular $|\O^\e_p|^{-1}$ is bounded uniformly in $\e$.
			
			\medskip
			{\bf Step 1. The energy inequality.}
			By Definition \ref{Def01}, the weak solution satisfies the energy inequality
			\eqref{weak-energy1}. 
			Moving the source term to the left-hand side, this reads, for a.e.\ $t\in S$,
			\begin{multline}\label{4.5}
				\frac12\int_{\O^\e_p}|\Gu^\e(t)|^2\,\md x
				+2\nu\int_0^t\!\!\int_{\O^\e_p}|D(\Gu^\e)|^2\,\md(x,s)
				+\lambda\int_0^t\!\!\int_{\O^\e_p}|\nabla\mu^\e|^2\,\md(x,s)\\
				+\frac{\lambda}{2}\int_{\O^\e_p}|\nabla\phi^\e(t)|^2\,\md x
				+\lambda\int_{\O^\e_p}F(\phi^\e(t))\,\md x
				+\lambda\int_0^t\!\!\int_{\O^\e_p}G(\phi^\e)\mu^\e\,\md(x,s)\\
				\le\ \Ec^\e(0)+\int_0^t\!\!\int_{\O^\e_p}\t\cdot\Gu^\e\,\md(x,s).
			\end{multline}
			By the assumptions on the initial data \eqref{ICA01}, the growth bound \eqref{F01}
			on $F$ and \eqref{VolBound01},
			\begin{equation}\label{InitEnergy01}
				\Ec^\e(0)
				=\frac12\|\Gu^\e_0\|^2_{L^2(\O^\e_p)}
				+\lambda\Bigl(\frac12\|\nabla\phi_0\|^2_{L^2(\O^\e_p)}
				+\int_{\O^\e_p}F(\phi_0)\,\md x\Bigr)
				\le C.
			\end{equation}
For the force term, H\"older's and Young's inequalities give, for any
$\kappa_5>0$ and a.e.\ $t\in S$,
\begin{equation}\label{ForceEst01}
	\left|\int_0^t\!\!\int_{\O^\e_p}\t\cdot\Gu^\e\,\md(x,s)\right|
	\le\frac{1}{2\kappa_5}\|\t\|^2_{L^2(\Qc^\e_p)}
	+\frac{\kappa_5}{2}\int_0^t\|\Gu^\e\|^2_{L^2(\O^\e_p)}\,\md s .
\end{equation}
Since $\Gu^\e(s)\in\GH^1_\div(\O^\e_p)$ for a.e.\ $s\in S$, the uniform Korn
inequality \eqref{eq:Poincare-Korn}$_2$ yields
\begin{equation}\label{Korn01}
	\|\Gu^\e\|_{L^2(\O^\e_p)}+\|\nabla\Gu^\e\|_{L^2(\O^\e_p)}
	\le C_K\|D(\Gu^\e)\|_{L^2(\O^\e_p)}
	\qquad\text{for a.e. }s\in S,
\end{equation}
with $C_K$ independent of $\e$. Choosing $\kappa_5>0$ with
$C_K^2\kappa_5/2\le\nu$, the last term in \eqref{ForceEst01} is absorbed into
the viscous term of \eqref{4.5}. Together with \eqref{InitEnergy01} and the
bound on $\t$ we obtain, for a.e.\ $t\in S$,
\begin{multline}\label{4.7}
	\frac12\|\Gu^\e(t)\|^2_{L^2(\O^\e_p)}
	+\nu\int_0^t\|D(\Gu^\e)\|^2_{L^2(\O^\e_p)}\,\md s
	+\lambda\int_0^t\|\nabla\mu^\e\|^2_{L^2(\O^\e_p)}\,\md s
	+\frac{\lambda}{2}\|\nabla\phi^\e(t)\|^2_{L^2(\O^\e_p)}\\
	+\lambda\int_{\O^\e_p}F(\phi^\e(t))\,\md x
	+\lambda\int_0^t\!\!\int_{\O^\e_p}G(\phi^\e)\mu^\e\,\md(x,s)
	\ \le\ C .
\end{multline}
			Note that the last term on the left-hand side has no sign a priori; it is estimated
			from below in Step 2.
			
			{\bf Step 2. The source contribution.}
			Since $G$ is Lipschitz with $G(0)=0$ by \eqref{AssG01}, and since
			$\phi^\e\in L^\infty(S;H^1(\O^\e_p))$, the composition $G(\phi^\e)$ belongs to
			$L^2(S;H^1(\O^\e_p))$, with $\nabla G(\phi^\e)=G'(\phi^\e)\nabla\phi^\e$. Hence
			$\varphi_3=G(\phi^\e)$ is an admissible test function in \eqref{weak-mu1}, and we
			obtain, for a.e.\ $t\in(0,T)$,
			\begin{equation}\label{E01}
				\int_{\O^\e_p}G(\phi^\e)\mu^\e\,\md x
				=\underbrace{\int_{\O^\e_p}G'(\phi^\e)|\nabla\phi^\e|^2\,\md x}_{=:I_1}
				+\underbrace{\int_{\O^\e_p}G(\phi^\e)F'(\phi^\e)\,\md x}_{=:I_2}.
			\end{equation}
			By assumption \eqref{AssG01}, namely $0<c_1\le G'(s)\le c_2$ for all $s\in\R$,
			\begin{equation}\label{E02}
				c_1\|\nabla\phi^\e\|^2_{L^2(\O^\e_p)}\le I_1\le c_2\|\nabla\phi^\e\|^2_{L^2(\O^\e_p)}.
			\end{equation}
			We turn to $I_2$. Since $F'(s)=s^3-s$, we have,
			\[
			G(s)F'(s)=G(s)s\,(s^2-1).
			\]
			If $|s|\ge1$, then $s^2-1\ge0$ and the lower bound in \eqref{SourceGrowth} gives
			$G(s)F'(s)\ge c_1s^2(s^2-1)$. If $|s|<1$, then $s^2-1<0$ and the upper bound in
			\eqref{SourceGrowth} gives $G(s)F'(s)\ge c_2s^2(s^2-1)\ge-\frac{c_2}{4}$, whereas
			$c_1s^2(s^2-1)\le0$; hence in this case as well
			$G(s)F'(s)\ge c_1s^2(s^2-1)-\frac{c_2}{4}$. Using the identity
			$s^2(s^2-1)=\bigl(s^2-\frac12\bigr)^2-\frac14$, we conclude that
			\begin{equation}\label{E023}
				G(s)F'(s)\ \ge\ c_1\Bigl(s^2-\frac12\Bigr)^2-\frac{c_1+c_2}{4},
				\qquad\forall s\in\R,
			\end{equation}
			and therefore, by \eqref{VolBound01},
			\begin{equation}\label{E03}
				I_2\ \ge\ c_1\int_{\O^\e_p}\Bigl((\phi^\e)^2-\frac12\Bigr)^2\md x
				-\frac{c_1+c_2}{4}\,|\O|.
			\end{equation}
			Combining \eqref{E01}--\eqref{E03} and inserting the result into \eqref{4.7}, we
			obtain
			\begin{multline}\label{E05}
				\frac12\|\Gu^\e\|^2_{L^\infty(S;L^2(\O^\e_p))}
				+\nu\|D(\Gu^\e)\|^2_{L^2(\Qc^\e_p)}
				+\lambda\|\nabla\mu^\e\|^2_{L^2(\Qc^\e_p)}
				+\frac{\lambda}{2}\|\nabla\phi^\e\|^2_{L^\infty(S;L^2(\O^\e_p))}\\
				+\lambda\|F(\phi^\e)\|_{L^\infty(S;L^1(\O^\e_p))}
				+\lambda c_1\|\nabla\phi^\e\|^2_{L^2(\Qc^\e_p)}
				+\lambda c_1\int_{\Qc^\e_p}\Bigl((\phi^\e)^2-\frac12\Bigr)^2\md(x,t)
				\ \le\ C.
			\end{multline}
			All terms on the left-hand side are now nonnegative; taking the supremum over
			$t\in S$ gives
			\begin{equation}\label{U131}
				\begin{aligned}
					\|\Gu^\e\|_{L^\infty(S;L^2(\O^\e_p))}+\|D(\Gu^\e)\|_{L^2(\Qc^\e_p)}+
					\|\nabla\mu^\e\|_{L^2(\Qc^\e_p)}
					+\|\nabla\phi^\e\|_{L^\infty(S;L^2(\O^\e_p))}&\le C,\\
					\|F(\phi^\e)\|_{L^\infty(S;L^1(\O^\e_p))}
					+\int_{\Qc^\e_p}\Bigl((\phi^\e)^2-\frac12\Bigr)^2\md(x,t)&\le C.
				\end{aligned}
			\end{equation}
			Applying the uniform Korn inequality \eqref{Korn01} once more, the first line of
			\eqref{U131} upgrades to \eqref{MainE01}$_{1,2}$.
			Finally, the lower bound $F(s)\ge\frac12s^2-\frac34$ from \eqref{F01} together with
			the fifth term in \eqref{E05} and \eqref{VolBound01} gives
			$\|\phi^\e\|_{L^\infty(S;L^2(\O^\e_p))}\le C$, whence, with
			\eqref{U131}$_4$, we get \eqref{MainE01}$_4$.

{\bf Step 3. The transport term.}
Since $\nabla\cdot\Gu^\e=0$ in $\O^\e_p$, $\Gu^\e=0$ on $\partial\O$ and
$\Gu^\e\cdot\Gn^\e=0$ on $\Gamma^\e_s$, integration by parts gives, for every
$\varphi\in H^1(\O^\e_p)$ and a.e.\ $t\in S$,
\begin{equation}\label{TransportIBP}
	\int_{\O^\e_p}(\Gu^\e\cdot\nabla\phi^\e)\,\varphi\,\md x
	=-\int_{\O^\e_p}\phi^\e\,\Gu^\e\cdot\nabla\varphi\,\md x ,
\end{equation}
the boundary integral vanishing on $\partial\O$ because $\Gu^\e=0$ there, and on
$\Gamma^\e_s$ because of the impermeability condition.
			
			{\bf Step 4. Estimate for $\partial_t\Gu^\e$ in
	$L^{4/3}(S;\GH_\div^{-1}(\O^\e_p))$.}
Let $\psi\in\GH^1_\div(\O^\e_p)$. By the weak momentum formulation
\eqref{weak-momentum1} and the antisymmetric form
\eqref{eq:trilinear-form} of the convective term,
\begin{equation}\label{TimeEst01}
	\bigl\langle\partial_t\Gu^\e,\psi\bigr\rangle
	=\Rey_\e\int_{\O^\e_p}\Gu^\e\otimes\Gu^\e:\nabla\psi\,\md x
	-2\nu\int_{\O^\e_p}D(\Gu^\e):D(\psi)\,\md x
	-\lambda\int_{\O^\e_p}\phi^\e\nabla\mu^\e\cdot\psi\,\md x
	+\int_{\O^\e_p}\th^\e\cdot\psi\,\md x .
\end{equation}
For the convective term, H\"older's inequality and $\Rey_\e\le1$ from
\eqref{ReynoldsAssumption} give
\begin{equation}\label{eq:convective-bound}
	\left|\Rey_\e\int_{\O^\e_p}\Gu^\e\otimes\Gu^\e:\nabla\psi\,\md x\right|
	\le\|\Gu^\e\|^2_{L^4(\O^\e_p)}\|\nabla\psi\|_{L^2(\O^\e_p)}
	\le C\|\Gu^\e\|_{L^2(\O^\e_p)}^{1/2}
	\|\Gu^\e\|_{H^1(\O^\e_p)}^{3/2}\|\psi\|_{H^1(\O^\e_p)},
\end{equation}
the last step by the interpolation inequality
$\|v\|_{L^4}\le\|v\|_{L^2}^{1/4}\|v\|_{L^6}^{3/4}$, valid on any domain,
and the uniform Sobolev embedding
$H^1(\O^\e_p)\hookrightarrow L^6(\O^\e_p)$ of \eqref{eq:PW-Sobolev}$_2$.
The viscous term is estimated by Cauchy--Schwarz,
\[
\left|2\nu\int_{\O^\e_p}D(\Gu^\e):D(\psi)\,\md x\right|
\le 2\nu\|D(\Gu^\e)\|_{L^2(\O^\e_p)}\|D(\psi)\|_{L^2(\O^\e_p)}
\le C\|D(\Gu^\e)\|_{L^2(\O^\e_p)}\|\psi\|_{H^1(\O^\e_p)}.
\]
For the capillary term, H\"older's inequality together with the uniform Sobolev
embedding $H^1(\O^\e_p)\hookrightarrow L^4(\O^\e_p)$ from \eqref{eq:PW-Sobolev}$_2$ gives
\[
\left|\lambda\int_{\O^\e_p}\phi^\e\nabla\mu^\e\cdot\psi\,\md x\right|
\le\lambda\|\phi^\e\|_{L^4(\O^\e_p)}\|\nabla\mu^\e\|_{L^2(\O^\e_p)}
\|\psi\|_{L^4(\O^\e_p)}
\le C\|\phi^\e\|_{H^1(\O^\e_p)}\|\nabla\mu^\e\|_{L^2(\O^\e_p)}\|\psi\|_{H^1(\O^\e_p)},
\]
while $\bigl|\int_{\O^\e_p}\th^\e\cdot\psi\,\md x\bigr|
\le\|\th^\e\|_{L^2(\O^\e_p)}\|\psi\|_{H^1(\O^\e_p)}$. Hence
\begin{equation}\label{TimeEst05}
	\|\partial_t\Gu^\e\|_{\GH^{-1}_\div(\O^\e_p)}
	\le C\Bigl(
	\|\Gu^\e\|_{L^2(\O^\e_p)}^{1/2}\|\Gu^\e\|_{H^1(\O^\e_p)}^{3/2}
	+\|D(\Gu^\e)\|_{L^2(\O^\e_p)}
	+\|\phi^\e\|_{H^1(\O^\e_p)}\|\nabla\mu^\e\|_{L^2(\O^\e_p)}
	+\|\th^\e\|_{L^2(\O^\e_p)}\Bigr).
\end{equation}
It is the first term that fixes the integrability in time. Raising
\eqref{TimeEst05} to the power $4/3$ and integrating over $S$, the
convective contribution is
\[
\int_S\|\Gu^\e\|_{L^2(\O^\e_p)}^{2/3}
\|\Gu^\e\|_{H^1(\O^\e_p)}^{2}\,\md t
\le
\|\Gu^\e\|^{2/3}_{L^\infty(S;L^2(\O^\e_p))}
\|\Gu^\e\|^{2}_{L^2(S;H^1(\O^\e_p))}
\le C
\]
by \eqref{MainE01}$_{1,2}$, and this is the largest exponent for which
the right-hand side is finite: the product
$\|\Gu^\e\|_{L^2}^{1/2}\|\Gu^\e\|_{H^1}^{3/2}$ belongs to $L^{4/3}(S)$
and to no better Lebesgue space in time, the velocity being bounded only
in $L^\infty(S;L^2)\cap L^2(S;H^1)$. The three remaining terms of
\eqref{TimeEst05} belong to $L^2(S)\subset L^{4/3}(S)$, by
\eqref{U131}, \eqref{MainE01}$_2$, the bound
$\phi^\e\in L^\infty(S;H^1(\O^\e_p))$ and the hypothesis on $\th^\e$.
We conclude \eqref{MainE01}$_3$.
			
			{\bf Step 5. Estimate for $\partial_t\phi^\e$ in $L^2(S;H^{-1}(\O^\e_p))$.}
			Let $\varphi_2\in H^1(\O^\e_p)$. From \eqref{weak-phi1}, and after integrating the
			transport term by parts as in Step 3,
			\begin{equation}\label{PhiTime01}
				\langle\partial_t\phi^\e,\varphi_2\rangle
				=-\int_{\O^\e_p}G(\phi^\e)\varphi_2\,\md x
				-\int_{\O^\e_p}\nabla\mu^\e\cdot\nabla\varphi_2\,\md x
				+\int_{\O^\e_p}\phi^\e\,\Gu^\e\cdot\nabla\varphi_2\,\md x .
			\end{equation}
			By \eqref{SourceGrowth}, $|G(s)|\le c_2|s|$, so the first term is bounded by
			$C\|\phi^\e\|_{L^2(\O^\e_p)}\|\varphi_2\|_{H^1(\O^\e_p)}$, and the second by
			$\|\nabla\mu^\e\|_{L^2(\O^\e_p)}\|\varphi_2\|_{H^1(\O^\e_p)}$. For the transport term,
			H\"older's inequality and \eqref{eq:PW-Sobolev}$_2$ give
			\[
			\left|\int_{\O^\e_p}\phi^\e\,\Gu^\e\cdot\nabla\varphi_2\,\md x\right|
			\le\|\phi^\e\|_{L^4(\O^\e_p)}\|\Gu^\e\|_{L^4(\O^\e_p)}
			\|\nabla\varphi_2\|_{L^2(\O^\e_p)}
			\le C\|\phi^\e\|_{H^1(\O^\e_p)}\|\Gu^\e\|_{H^1(\O^\e_p)}\|\varphi_2\|_{H^1(\O^\e_p)}.
			\]
			Hence
			\begin{equation}\label{PhiTime03}
				\|\partial_t\phi^\e\|_{(H^{-1}(\O^\e_p))}
				\le C\Bigl(\|\phi^\e\|_{L^2(\O^\e_p)}+\|\nabla\mu^\e\|_{L^2(\O^\e_p)}
				+\|\phi^\e\|_{H^1(\O^\e_p)}\|\Gu^\e\|_{H^1(\O^\e_p)}\Bigr).
			\end{equation}
			Squaring, integrating over $S$, and using $\phi^\e\in L^\infty(S;H^1(\O^\e_p))$ together
			with $\Gu^\e\in L^2(S;H^1(\O^\e_p))$ from \eqref{MainE01}$_2$ and $\nabla\mu^\e\in
			L^2(\Qc^\e_p)$ from \eqref{U131}, we obtain \eqref{MainE01}$_5$.
			
			\smallskip
			{\bf Step 6. Estimate for $\mu^\e$ in $L^2(S;H^1(\O^\e_p))$.}
			For a.e.\ $t\in(0,T)$ we decompose
			$\mu^\e(t)=m^\e(\mu^\e)(t)+\bigl(\mu^\e(t)-m^\e(\mu^\e)(t)\bigr)$, where
			\[
			m^\e(\mu^\e)(t):=\frac{1}{|\O^\e_p|}\int_{\O^\e_p}\mu^\e(t,x)\,\md x .
			\]
			The second summand has zero mean over $\O^\e_p$, so the uniform
			Poincar\'e--Wirtinger inequality \eqref{eq:PW-Sobolev}$_1$ yields
			\begin{equation}\label{pw-mu}
				\|\mu^\e-m^\e(\mu^\e)\|_{L^2(S;H^1(\O^\e_p))}
				\le C\|\nabla\mu^\e\|_{L^2(\Qc^\e_p)} .
			\end{equation}
			Taking $\varphi_3=1$ in \eqref{weak-mu1} gives
			\[
			m^\e(\mu^\e)(t)=\frac{1}{|\O^\e_p|}\int_{\O^\e_p}F'(\phi^\e(t,x))\,\md x .
			\]
			By the growth condition $|F'(s)|\le C(1+|s|^3)$ from \eqref{F01}, the uniform
			embedding $H^1(\O^\e_p)\hookrightarrow L^3(\O^\e_p)$ of \eqref{eq:PW-Sobolev}$_2$ and
			\eqref{VolBound01},
			\[
			|m^\e(\mu^\e)(t)|\le C\bigl(1+\|\phi^\e(t)\|^3_{H^1(\O^\e_p)}\bigr),
			\]
			whence, by \eqref{MainE01}$_4$,
			\begin{equation}\label{mu-bar-bound}
				\|m^\e(\mu^\e)\|_{L^\infty(0,T)}\le C .
			\end{equation}
			Combining \eqref{pw-mu}, \eqref{mu-bar-bound} and \eqref{U131} we conclude \eqref{MainE01}$_6$. This completes the proof.
		\end{proof}

\section{Compactness in the perforated domain}\label{SSec61}

This section contains the proof of Lemma~\ref{lem:limit-compactness}. We use
the primal and zero extensions \eqref{eq:extensions-notation} and recall from
\eqref{eq:obstacle-volume} that
\begin{equation}\label{ObstacleVolume}
	|\O_s^\e|^{1/2}\leq C\e^{3(\alpha-1)/2}\longrightarrow0 .
\end{equation}

\subsection{Contribution of the obstacle set}

\begin{lemma}\label{lem:hole-error}
	Let $\{g^\e\}_\e\subset L^2(\O_T)$ satisfy
	$\|g^\e\|_{L^2(\O_T)}\leq C$ uniformly in $\e$. Then the following hold.
	\begin{enumerate}[label=\textnormal{(\roman*)}]
		\item For every $\psi\in L^\infty(\O_T)$,
		\begin{equation}\label{HoleErrorEstimate}
			\left|
			\int_{S\times\O_s^\e} g^\e\psi\,\md(x,t)
			\right|
			\leq
			C\|\psi\|_{L^\infty(\O_T)}\,|\O_s^\e|^{1/2}
			\leq
			C\|\psi\|_{L^\infty(\O_T)}\,\e^{3(\alpha-1)/2}
			\longrightarrow0 .
		\end{equation}
		Consequently, if $g^\e\rightharpoonup g$ weakly in $L^2(\O_T)$, then
		\[
		\int_{S\times\O_p^\e} g^\e\psi\,\md(x,t)
		\longrightarrow
		\int_{\O_T} g\psi\,\md(x,t)
		\qquad\text{for every }\psi\in L^\infty(\O_T).
		\]
		\item If $g^\e\rightharpoonup g$ weakly in $L^2(\O_T)$, then
		\begin{equation}\label{HoleErrorWeak}
			\chi^\e_p\,g^\e\rightharpoonup g
			\qquad\text{weakly in }L^2(\O_T).
		\end{equation}
		\item If $g^\e\to g$ strongly in $L^2(\O_T)$, then
		$\chi^\e_p g^\e\to g$ strongly in $L^2(\O_T)$.
	\end{enumerate}
	The same statements hold componentwise for vector- and matrix-valued
	sequences.
\end{lemma}

\begin{proof}
	(i) By Hölder's inequality,
	\[
	\int_{S\times\O_s^\e}|g^\e\psi|\,\md(x,t)
	\leq
	\|\psi\|_{L^\infty(\O_T)}\|g^\e\|_{L^2(\O_T)}
	\,|S\times\O_s^\e|^{1/2},
	\]
	and $|S\times\O_s^\e|^{1/2}=T^{1/2}|\O_s^\e|^{1/2}$, which gives
	\eqref{HoleErrorEstimate} by \eqref{ObstacleVolume}. The second assertion
	follows by writing
	$\int_{S\times\O_p^\e}=\int_{\O_T}-\int_{S\times\O_s^\e}$, the first
	integral converging by weak--strong pairing and the second by
	\eqref{HoleErrorEstimate}.
	
	(ii) Let $\psi\in L^2(\O_T)$ be fixed. Then
	\[
	\left|
	\int_{S\times\O_s^\e}g^\e\psi\,\md(x,t)
	\right|
	\leq
	\|g^\e\|_{L^2(\O_T)}
	\|\psi\|_{L^2(S\times\O_s^\e)}
	\longrightarrow0,
	\]
	the last convergence by \eqref{eq:char-convergence}, since
$\|\psi\|_{L^2(S\times\O_s^\e)}=\|\chi_s^\e\psi\|_{L^2(\O_T)}\to0$
by dominated convergence, $\psi$ being independent of $\e$.
	
	(iii) By \eqref{eq:char-convergence},
\[
\|\chi_p^\e g^\e-g\|_{L^2(\O_T)}
\leq\|\chi_p^\e(g^\e-g)\|_{L^2(\O_T)}+\|\chi_s^\e g\|_{L^2(\O_T)}
\leq\|g^\e-g\|_{L^2(\O_T)}+\|\chi_s^\e g\|_{L^2(\O_T)},
\]
and both terms tend to zero, the first by hypothesis and the second by
dominated convergence.
\end{proof}

\subsection{Compactness properties}

\begin{lemma}
	\label{lem:unfolding-compactness}
	The following assertions\footnote{Throughout this subsection, a subsequence
		of $\{\e\}_\e$, not relabeled, is extracted whenever necessary.} hold.
	
	\begin{enumerate}[label=\textnormal{(\roman*)}]
		\item Let $v^\e\in L^2(S;H^1(\O_p^\e))$ satisfy
		$\|v^\e\|_{L^2(S;H^1(\O_p^\e))}\leq C$. Then, up to a subsequence, there
		exists $v\in L^2(S;H^1(\O))$ such that
		\begin{equation}\label{PeriodicScalarCompactness}
			\wt v^\e\rightharpoonup v
			\quad\text{weakly in }L^2(S;H^1(\O)),
			\qquad
			\nabla\wt v^\e\rightharpoonup\nabla v
			\quad\text{weakly in }L^2(\O_T)^3 .
		\end{equation}
		If, in addition,
		\[
		\|v^\e\|_{L^\infty(S;H^1(\O_p^\e))}
		+\|\partial_tv^\e\|_{L^2(S;H^{-1}(\O_p^\e))}\leq C,
		\]
		then
		\begin{equation}\label{PeriodicScalarStrongCompactness}
			\wt v^\e\to v,
			\quad
			\overline{v^\e}\to v
			\qquad\text{strongly in }C([0,T];L^s(\O))
			\quad\text{for every }2\leq s<6 .
		\end{equation}
		
		\item Let $\bv^\e\in L^2(S;\GH^1_{\div}(\O_p^\e))$ satisfy
		\begin{equation}\label{VelocityCompactnessAssumptions}
			\|\bv^\e\|_{L^\infty(S;L^2(\O_p^\e)^3)}
			+\|\bv^\e\|_{L^2(S;H^1(\O_p^\e)^3)}
			+\|\partial_t\bv^\e\|_{L^{4/3}(S;\GH^{-1}_{\div}(\O_p^\e))}
			\leq C .
		\end{equation}
		Then, up to a subsequence, there exists
		$\bv\in L^\infty(S;L^2(\O)^3)\cap L^2(S;H^1_{0,\div}(\O)^3)$ such that
		\begin{equation}\label{PeriodicVelocityCompactness}
			\begin{aligned}
				\wt\bv^\e,\ \overline\bv^\e
				&\longrightarrow\bv
				&&\text{strongly in }L^2(S;L^s(\O)^3),\quad 2\leq s<6,\\
				\nabla\wt\bv^\e
				&\rightharpoonup\nabla\bv
				&&\text{weakly in }L^2(\O_T)^{3\times3},
			\end{aligned}
		\end{equation}
		and $\nabla\cdot\bv=0$ in $\mathcal D'(\O_T)$. Moreover, for every
		$\zeta\in C^\infty_{c,\div}(\O)^3$, with $\Rc^\e\zeta$ the solenoidal
		restriction of Lemma~\ref{lem:solenoidal-restriction},
		\begin{equation}\label{VelocityPairingConvergence}
			A_\e^\zeta:=(\bv^\e(\cdot),\Rc^\e\zeta)_{L^2(\O_p^\e)}
			\longrightarrow
			(\bv,\zeta)_{L^2(\O)}
			\qquad\text{strongly in }C([0,T]),
		\end{equation}
		and $\bv$ admits a representative in $C_w([0,T];H^\sigma)$.
	\end{enumerate}
\end{lemma}

\begin{proof}	
	\textbf{Step 1: Scalar compactness.}
	By Lemma~\ref{Emain}, $\{\wt v^\e\}$ is bounded in $L^2(S;H^1(\O))$, which
	gives \eqref{PeriodicScalarCompactness} after extraction; the second
	convergence there is the first one composed with the (linear, bounded)
	gradient.
	
	Assume now in addition the $L^\infty$-in-time and time-derivative bounds.
	Lemma~\ref{lem:scalar-two-extensions} gives, after a further extraction,
	$\wt v^\e\to v$ strongly in $C([0,T];L^2(\O))$, together with
	$\|\wt v^\e-\overline{v^\e}\|_{L^\infty(S;L^2(\O))}\leq Ca_\e/\e\to0$, so
	that $\overline{v^\e}$ has the same limit; this limit coincides with the weak limit in \eqref{PeriodicScalarCompactness},
	both being limits in $L^2(\O_T)$; the uniform bound in
	$L^\infty(S;H^1(\O))$ also gives $v\in L^\infty(S;H^1(\O))$ after a
	weak-$*$ extraction.
	
	For the upgrade to $L^s$, $2\leq s<6$, we interpolate. The primal extension
	is bounded in $L^\infty(S;H^1(\O))\hookrightarrow L^\infty(S;L^6(\O))$ by
	\eqref{eq:extension-primal-estimates},
	whereas the zero extension is not covered by that bound and is estimated
	directly using the uniform Sobolev embedding
	\eqref{eq:PW-Sobolev}$_2$. 
	
	\textbf{Step 2: Weak velocity compactness.}
	By Lemma~\ref{LemD02} and \eqref{VelocityCompactnessAssumptions},
	\begin{equation}\label{VelocityPrimalAndZeroBounds}
		\|\wt\bv^\e\|_{L^2(S;H^1_{0,\div}(\O)^3)}
		+\|\overline\bv^\e\|_{L^\infty(S;L^2(\O)^3)}
		\leq C,
	\end{equation}
	the second bound being the norm identity \eqref{SolenoidalZeroExtension}$_1$,
	and
	\begin{equation}\label{VelocityPrimalZeroCloseness}
		\wt\bv^\e-\overline\bv^\e\longrightarrow0
		\qquad\text{strongly in }L^2(\O_T)^3 ,
	\end{equation}
	by \eqref{PrimalZeroExtensionDifference}.
	Consequently there exist a subsequence and
	$\bv\in L^2(S;H^1_{0,\div}(\O)^3)$ with
	\begin{equation}\label{VelocityWeakFixedDomainConvergences}
		\wt\bv^\e\rightharpoonup\bv
		\qquad\text{weakly in }L^2(S;H^1_0(\O)^3).
	\end{equation}
	In particular $\nabla\cdot\bv=0$, since $\div\,\wt\bv^\e=0$ in $\O$ for
	every $\e$ by Lemma~\ref{LemD02} and the divergence is weakly continuous.
	
	The bound \eqref{VelocityPrimalAndZeroBounds} does {not} yield strong
	compactness by the Aubin--Lions--Simon theorem: the time derivative in
	\eqref{VelocityCompactnessAssumptions} is controlled only in the dual of
	$\GH^1_{\div}(\O_p^\e)$, a space that varies with $\e$, and no
	$\e$-uniform bound on $\partial_t\wt\bv^\e$ in a fixed space is available,
	the extension operator $\Ec^\e$ of Lemma~\ref{LemD02} carrying no dual
	estimate. Steps~3 and 4 replace that argument by a compactness statement
	for the scalar pairings against a fixed basis, which is enough because the
	spatial bound \eqref{VelocityPrimalAndZeroBounds} already lives on the
	fixed domain $\O$.
	
	\textbf{Step 3: Compactness of the velocity pairings.}
	A fixed $\zeta\in C^\infty_{c,\div}(\O)^3$ does not satisfy
	$\zeta\cdot\Gn^\e=0$ on $\Gamma_s^\e$, hence does not belong to
	$\GH^1_{\div}(\O_p^\e)$ and is not an admissible pairing for
	$\partial_t\bv^\e$; we use instead the solenoidal restriction
	$\Rc^\e\zeta\in\GH^1_{\div}(\O_p^\e)$ of
	Lemma~\ref{lem:solenoidal-restriction} and set
	\[
	A_\e^\zeta(t):=(\bv^\e(t),\Rc^\e\zeta)_{L^2(\O_p^\e)} .
	\]
	Since $\Rc^\e\zeta$ is independent of time, it may be taken out of the
	difference quotients defining the distributional derivative, so that
	\[
	\frac{\md}{\md t}A_\e^\zeta
	=\bigl\langle\partial_t\bv^\e,\Rc^\e\zeta\bigr\rangle_
	{\GH^{-1}_{\div}(\O_p^\e),\GH^1_{\div}(\O_p^\e)}
	\qquad\text{in }\mathcal D'(S).
	\]
	By Cauchy--Schwarz,
	\eqref{VelocityCompactnessAssumptions} and the uniform bound
	\eqref{SolenoidalRestrictionConvergence}$_2$,
	\[
	\|A_\e^\zeta\|_{L^\infty(S)}
	\leq\|\bv^\e\|_{L^\infty(S;L^2(\O_p^\e))}
	\|\Rc^\e\zeta\|_{L^2(\O_p^\e)}\leq C_\zeta ,
	\quad
	\Bigl\|\frac{\md}{\md t}A_\e^\zeta\Bigr\|_{L^{4/3}(S)}
	\leq
	\|\partial_t\bv^\e\|_{L^{4/3}(S;\GH^{-1}_{\div}(\O_p^\e))}
	\|\Rc^\e\zeta\|_{\GH^1_{\div}(\O_p^\e)}
	\leq C_\zeta .
	\]
	 Hence
	$\|A_\e^\zeta\|_{W^{1,4/3}(S)}\leq C_\zeta$, and since $4/3>1$ the
	embedding $W^{1,4/3}(S)\hookrightarrow C^{0,1/4}([0,T])$ holds, so
	$\{A_\e^\zeta\}_\e$ is bounded and equicontinuous on $[0,T]$. By the
	Arzel\`a--Ascoli theorem it is relatively compact in $C([0,T])$.
	
	Next, set
$B_\e^\zeta(t):=(\overline\bv^\e(t),\zeta)_{L^2(\O)}
=(\bv^\e(t),\zeta)_{L^2(\O_p^\e)}$, the second identity holding because
$\overline\bv^\e$ vanishes on $\O_s^\e$. For a.e.\ $t\in S$,
\[
|B_\e^\zeta(t)-A_\e^\zeta(t)|
=\bigl|(\bv^\e(t),\zeta-\Rc^\e\zeta)_{L^2(\O_p^\e)}\bigr|
\leq
\|\bv^\e\|_{L^\infty(S;L^2(\O_p^\e))}
\|\zeta-\overline{\Rc^\e\zeta}\|_{L^2(\O)} ,
\]
where the integral over $\O_p^\e$ has been rewritten over $\O$, since
$\overline\bv^\e$ vanishes on $\O_s^\e$.
By
\eqref{SolenoidalRestrictionConvergence}$_1$,
\begin{equation}\label{eq:AB-difference}
	\|B_\e^\zeta-A_\e^\zeta\|_{L^\infty(S)}\longrightarrow0 .
\end{equation}
Since $(\wt\bv^\e,\zeta)_{L^2(\O)}-B_\e^\zeta
=(\wt\bv^\e-\overline\bv^\e,\zeta)_{L^2(\O)}$, the bound
\eqref{eq:AB-difference} and \eqref{VelocityPrimalZeroCloseness} give
\begin{equation}\label{eq:tilde-A}
	\bigl\|(\wt\bv^\e,\zeta)_{L^2(\O)}-A_\e^\zeta\bigr\|_{L^2(S)}
	\leq T^{1/2}\|B_\e^\zeta-A_\e^\zeta\|_{L^\infty(S)}
	+\|\zeta\|_{L^2(\O)}\|\wt\bv^\e-\overline\bv^\e\|_{L^2(\O_T)}
	\longrightarrow0 .
\end{equation}
By the relative compactness of $\{A_\e^\zeta\}$ in $C([0,T])$, every
subsequence has a further subsequence converging in $C([0,T])$, hence in
$L^2(S)$; by \eqref{eq:tilde-A} the corresponding subsequence of
$(\wt\bv^\e,\zeta)_{L^2(\O)}$ converges in $L^2(S)$ to the same limit.
Thus $\{(\wt\bv^\e(\cdot),\zeta)_{L^2(\O)}\}_\e$ is relatively compact in
$L^2(S)$ for every $\zeta\in C^\infty_{c,\div}(\O)^3$.
	
	\textbf{Step 4: Strong velocity compactness.}

The space $H^\sigma$ is the closure of $C^\infty_{c,\div}(\O)^3$ in
$L^2(\O)^3$. It is therefore a separable Hilbert space, and
$C^\infty_{c,\div}(\O)^3$ contains a countable subset which is dense in
$H^\sigma$. Since $C^\infty_{c,\div}(\O)^3$ is a vector space, the
Gram--Schmidt procedure applied to this subset yields an orthonormal basis
$\{\zeta_k\}_{k\geq1}\subset C^\infty_{c,\div}(\O)^3$ of $H^\sigma$. For
$N\geq1$ let
\[
P_N z:=\sum_{k=1}^N(z,\zeta_k)_{L^2(\O)}\,\zeta_k ,\qquad z\in H^\sigma ,
\]
be the orthogonal projection onto
$\operatorname{span}\{\zeta_1,\ldots,\zeta_N\}$. Then
$\|P_Nz\|_{L^2(\O)}\leq\|z\|_{L^2(\O)}$,
$\|(I-P_N)z\|_{L^2(\O)}\leq\|z\|_{L^2(\O)}$, and by Parseval's identity
\begin{equation}\label{eq:PN-pointwise}
	\|(I-P_N)z\|_{L^2(\O)}^2=\sum_{k>N}|(z,\zeta_k)_{L^2(\O)}|^2
	\longrightarrow0
	\qquad\text{as }N\to\infty,\ \text{ for every }z\in H^\sigma .
\end{equation}
Since $\wt\bv^\e(t)\in H^1_{0,\div}(\O)^3\subset H^\sigma$ for a.e.\ $t\in S$,
and likewise $\bv(t)\in H^1_{0,\div}(\O)^3$, the projections
$P_N\wt\bv^\e(t)$ and $P_N\bv(t)$ are defined.

By the last assertion of Step~3, for every $k\geq1$ the sequence
$\{(\wt\bv^\e(\cdot),\zeta_k)_{L^2(\O)}\}_\e$ is relatively compact in
$L^2(S)$. For $\chi\in L^2(S)$ the map
$\bw\mapsto\int_S\chi(t)(\bw(t),\zeta_k)_{L^2(\O)}\,\md t$ is a continuous
linear functional on $L^2(S;H^1_0(\O)^3)$, so
\eqref{VelocityWeakFixedDomainConvergences} gives
$(\wt\bv^\e,\zeta_k)_{L^2(\O)}\rightharpoonup(\bv,\zeta_k)_{L^2(\O)}$
weakly in $L^2(S)$. Since strong convergence implies weak convergence,
every $L^2(S)$-cluster point of the sequence equals
$(\bv,\zeta_k)_{L^2(\O)}$; by relative compactness the whole sequence
converges:
\begin{equation}\label{eq:mode-convergence}
	(\wt\bv^\e(\cdot),\zeta_k)_{L^2(\O)}\longrightarrow(\bv(\cdot),\zeta_k)_{L^2(\O)}
	\qquad\text{strongly in }L^2(S)\ \text{ for every }k\geq1 .
\end{equation}
By the orthonormality of $\{\zeta_k\}$ and \eqref{eq:mode-convergence},
for every fixed $N$,
\begin{equation}\label{eq:projection-convergence}
	\|P_N\wt\bv^\e-P_N\bv\|_{L^2(\O_T)}^2
	=\sum_{k=1}^N\bigl\|(\wt\bv^\e,\zeta_k)_{L^2(\O)}
	-(\bv,\zeta_k)_{L^2(\O)}\bigr\|_{L^2(S)}^2
	\longrightarrow0 .
\end{equation}
Set
\[
\gamma_N:=
\sup_{\substack{z\in H^1_{0,\div}(\O)^3\\ \|z\|_{H^1(\O)}\leq1}}
\|(I-P_N)z\|_{L^2(\O)} .
\]
	We claim that $\gamma_N\to0$. Let $\mathbf B$ denote the closed unit ball
of $H^1_{0,\div}(\O)^3$. By the Rellich embedding
$H^1_0(\O)^3\Subset L^2(\O)^3$, and since $\mathbf B\subset H^\sigma$ with
$H^\sigma$ closed in $L^2(\O)^3$, the closure $\overline{\mathbf B}$ is a
compact subset of $H^\sigma$. Since $\|I-P_N\|\leq1$ and $I-P_N\to0$
pointwise on $H^\sigma$ by \eqref{eq:PN-pointwise}, the convergence is
uniform on compact subsets of $H^\sigma$. Hence
\begin{equation}\label{eq:gN}
	\gamma_N\leq\sup_{z\in\overline{\mathbf B}}\|(I-P_N)z\|_{L^2(\O)}\to0.
\end{equation}
Using the definition of $\gamma_N$,
$\|(I-P_N)z\|_{L^2(\O)}\leq\gamma_N\|z\|_{H^1(\O)}$ for every
$z\in H^1_{0,\div}(\O)^3$. Let
$C_0:=\sup_\e\|\wt\bv^\e\|_{L^2(S;H^1(\O))}$, which is finite by
\eqref{VelocityPrimalAndZeroBounds}; by
\eqref{VelocityWeakFixedDomainConvergences} and weak lower semicontinuity
of the norm, also $\|\bv\|_{L^2(S;H^1(\O))}\leq C_0$. Applying the last
inequality with $z=\wt\bv^\e(t)$ and $z=\bv(t)$ and integrating over $S$,
\begin{equation}\label{eq:tail-uniform}
	\|(I-P_N)\wt\bv^\e\|_{L^2(\O_T)}\leq C_0\gamma_N ,
	\qquad
	\|(I-P_N)\bv\|_{L^2(\O_T)}\leq C_0\gamma_N ,
\end{equation}
uniformly in $\e$.

	Writing
$\wt\bv^\e-\bv=P_N(\wt\bv^\e-\bv)+(I-P_N)\wt\bv^\e-(I-P_N)\bv$ and using
\eqref{eq:tail-uniform}, we obtain for every $N$ and every $\e$
\begin{equation}\label{eq:split-estimate}
	\|\wt\bv^\e-\bv\|_{L^2(\O_T)}
	\leq\|P_N\wt\bv^\e-P_N\bv\|_{L^2(\O_T)}+2C_0\gamma_N .
\end{equation}
Taking the $\limsup$ as $\e\to0$ and using
\eqref{eq:projection-convergence}, we get
$\limsup_{\e\to0}\|\wt\bv^\e-\bv\|_{L^2(\O_T)}\leq2C_0\gamma_N$ for every
$N$. Letting $N\to\infty$ and using \eqref{eq:gN} gives
\begin{equation}\label{VelocityStrongFixedDomainConvergence}
	\wt\bv^\e\longrightarrow\bv
	\qquad\text{strongly in }L^2(\O_T)^3 .
\end{equation}
Combining this with \eqref{VelocityPrimalZeroCloseness},
\begin{equation}\label{eq:zero-ext-strong}
	\|\overline\bv^\e-\bv\|_{L^2(\O_T)}
	\leq\|\overline\bv^\e-\wt\bv^\e\|_{L^2(\O_T)}
	+\|\wt\bv^\e-\bv\|_{L^2(\O_T)}\longrightarrow0 .
\end{equation}
By \eqref{VelocityPrimalAndZeroBounds}, $\{\overline\bv^\e\}$ is bounded
in $L^\infty(S;L^2(\O)^3)=\bigl(L^1(S;L^2(\O)^3)\bigr)^*$; since
$L^2(\O_T)^3\subset L^1(S;L^2(\O)^3)$, every weak-$*$ cluster point
coincides with the strong limit in \eqref{eq:zero-ext-strong}. Hence
$\overline\bv^\e\overset{*}{\rightharpoonup}\bv$ in $L^\infty(S;L^2(\O)^3)$
and $\bv\in L^\infty(S;L^2(\O)^3)$.

The sequences $\{\wt\bv^\e\}$ and $\{\overline\bv^\e\}$ are bounded in
$L^2(S;L^6(\O)^3)$, by \eqref{VelocityPrimalAndZeroBounds} with
$H^1_0(\O)\hookrightarrow L^6(\O)$ and by the uniform Sobolev embedding
\eqref{eq:PW-Sobolev}$_2$ on $\O_p^\e$, respectively; so is
$\bv\in L^2(S;H^1_0(\O)^3)$. For $2<s<6$ we interpolate in space and H\"older's
inequality in time give $\|\wt\bv^\e-\bv\|_{L^2(S;L^s(\O))}\longrightarrow0$
by \eqref{VelocityStrongFixedDomainConvergence}, and likewise for
$\overline\bv^\e$ by \eqref{eq:zero-ext-strong}. This is
\eqref{PeriodicVelocityCompactness}$_1$.

	\textbf{Step 5: The pairings and weak continuity in time.}
Fix $\zeta\in C^\infty_{c,\div}(\O)^3$. By \eqref{eq:tilde-A} and
\eqref{VelocityStrongFixedDomainConvergence},
\begin{equation}\label{eq:A-L2}
	\bigl\|A_\e^\zeta-(\bv,\zeta)_{L^2(\O)}\bigr\|_{L^2(S)}
	\leq\bigl\|A_\e^\zeta-(\wt\bv^\e,\zeta)_{L^2(\O)}\bigr\|_{L^2(S)}
	+\|\zeta\|_{L^2(\O)}\|\wt\bv^\e-\bv\|_{L^2(\O_T)}
	\longrightarrow0 .
\end{equation}
	By Step~3, every subsequence of $\{A_\e^\zeta\}$ has a further
subsequence converging in $C([0,T])$ to some $a\in C([0,T])$; by
\eqref{eq:A-L2}, $a=(\bv(\cdot),\zeta)_{L^2(\O)}$ a.e.\ in $S$. Since an
a.e.\ class contains at most one continuous function, all these limits
coincide, and the whole sequence $\{A_\e^\zeta\}$ converges in
$C([0,T])$ to the continuous representative of
$t\mapsto(\bv(t),\zeta)_{L^2(\O)}$.

	Let $\{\zeta_k\}$ be the basis of Step~4, $M:=\|\bv\|_{L^\infty(S;L^2(\O))}$,
and $a_k$ the continuous representative of $(\bv(\cdot),\zeta_k)_{L^2(\O)}$.
Outside a null set $E\subset S$, $\|\bv(t)\|_{L^2(\O)}\leq M$ and
$(\bv(t),\zeta_k)_{L^2(\O)}=a_k(t)$ for all $k$; hence, for $z\in H^\sigma$,
$s,t\in S\setminus E$ and $N\geq1$,
\begin{equation}\label{eq:weak-cont-density}
	|(\bv(t)-\bv(s),z)_{L^2(\O)}|
	\leq\sum_{k=1}^N|a_k(t)-a_k(s)|\,|(z,\zeta_k)_{L^2(\O)}|
	+2M\|(I-P_N)z\|_{L^2(\O)} .
\end{equation}
By \eqref{eq:weak-cont-density}, \eqref{eq:PN-pointwise} and the density
of $S\setminus E$ in $[0,T]$, weak limits along $S\setminus E$ define a
representative of $\bv$ in $C_w([0,T];H^\sigma)$. For this
representative, $(\bv(\cdot),\zeta)_{L^2(\O)}$ is the continuous function
obtained above, which proves \eqref{VelocityPairingConvergence}. This completes the proof.
\end{proof}

\subsection{Proof of Lemma \ref{lem:limit-compactness}}

\begin{proof}[Proof of the Lemma \ref{lem:limit-compactness}]
	\textbf{Step 1: Compactness.}
	By the bound on $\theta^\e$, Lemma~\ref{Est01} gives the estimates
	\eqref{MainE01} uniformly in $\e$. The bounds on $\Gu^\e$ in
	\eqref{MainE01} are the hypotheses \eqref{VelocityCompactnessAssumptions}
	of Lemma~\ref{lem:unfolding-compactness}(ii), which yields the
	convergences of $\wt\Gu^\e$, $\overline{\Gu^\e}$ and $\nabla\wt\Gu^\e$ in
	\eqref{TSL0}, the regularity $\Gu\in\fU\cap C_w([0,T];H^\sigma)$, and
	\eqref{TSL-pairing}. The bounds on $\phi^\e$ in \eqref{MainE01} are the
	hypotheses of the second part of Lemma~\ref{lem:unfolding-compactness}(i),
	which yields the convergences of $\wt\phi^\e$, $\overline{\phi^\e}$ and
	$\nabla\wt\phi^\e$, and $\phi\in\fC$ by
	Lemma~\ref{lem:scalar-two-extensions}. The bound on $\mu^\e$ in
	$L^2(S;H^1(\O_p^\e))$ is the hypothesis of the first part of
	Lemma~\ref{lem:unfolding-compactness}(i), which yields the convergences of
	$\wt\mu^\e$ and $\nabla\wt\mu^\e$, and $\mu\in\fW$.
	
	For the remaining zero extensions, note that
	$\overline{\nabla w^\e}=\chi_p^\e\nabla\wt w^\e$ for
	$w\in\{\Gu,\phi,\mu\}$ and $\overline{\mu^\e}=\chi_p^\e\wt\mu^\e$.
	Since $\nabla\wt w^\e$ and $\wt\mu^\e$ converge weakly in $L^2(\O_T)$,
	Lemma~\ref{lem:hole-error}(ii) gives the same weak limits for these zero
	extensions.
	
	\textbf{Step 2: Nonlinear terms.}
	By \eqref{MainE01}, \eqref{eq:extension-primal-estimates},
	\eqref{VelocityPrimalAndZeroBounds} and $H^1(\O)\hookrightarrow L^6(\O)$,
	\begin{equation}\label{TSL-bounds}
		\|\wt\phi^\e\|_{L^\infty(S;L^6(\O))}
		+\|\phi\|_{L^\infty(S;L^6(\O))}
		+\|\wt\Gu^\e\|_{L^2(S;L^6(\O))}
		+\|\nabla\wt\mu^\e\|_{L^2(\O_T)}\leq C .
	\end{equation}
	For the transport product, H\"older's inequality gives
	\[
	\|\wt\phi^\e\wt\Gu^\e-\phi\Gu\|_{L^2(\O_T)}
	\leq\|\wt\phi^\e-\phi\|_{L^\infty(S;L^4(\O))}
	\|\wt\Gu^\e\|_{L^2(S;L^4(\O))}
	+\|\phi\|_{L^\infty(S;L^4(\O))}\|\wt\Gu^\e-\Gu\|_{L^2(S;L^4(\O))},
	\]
	and both terms tend to zero by \eqref{TSL0} with $s=4$ and
	\eqref{TSL-bounds}. The bound
	$\|\wt\phi^\e\wt\Gu^\e\|_{L^2(\O_T)}
	\leq\|\wt\phi^\e\|_{L^\infty(S;L^4)}\|\wt\Gu^\e\|_{L^2(S;L^4)}\leq C$
	follows in the same way.
	
	For the double-well terms, the mean value theorem and \eqref{F01} give
	$|F(a)-F(b)|\leq C(1+|a|^3+|b|^3)|a-b|$ and
	$|F'(a)-F'(b)|\leq C(1+a^2+b^2)|a-b|$. Hence, by H\"older's inequality
	and \eqref{TSL-bounds},
	\[
	\|F(\wt\phi^\e)-F(\phi)\|_{L^1(\O_T)}
	\leq CT\bigl(1+\|\wt\phi^\e\|^3_{L^\infty(S;L^6)}
	+\|\phi\|^3_{L^\infty(S;L^6)}\bigr)
	\|\wt\phi^\e-\phi\|_{L^\infty(S;L^2)}\longrightarrow0 .
	\]
	For $2\leq s<6$ and $\frac1q=\frac13+\frac1s$, so that $q<2$,
	\[
	\|F'(\wt\phi^\e)-F'(\phi)\|_{L^\infty(S;L^q(\O))}
	\leq C\bigl(1+\|\wt\phi^\e\|^2_{L^\infty(S;L^6)}
	+\|\phi\|^2_{L^\infty(S;L^6)}\bigr)
	\|\wt\phi^\e-\phi\|_{L^\infty(S;L^s)}\longrightarrow0 .
	\]
	Since $|F'(a)|\leq C(1+|a|^3)$, the sequence $\{F'(\wt\phi^\e)\}$ is
	bounded in $L^\infty(S;L^2(\O))$ by \eqref{TSL-bounds}. Every weak
	cluster point in $L^2(\O_T)$ therefore coincides with the strong
	$L^1(\O_T)$ limit $F'(\phi)$, which gives the third line of
	\eqref{TSL02}. For $G$, \eqref{AssG01} gives
	$\|G(\wt\phi^\e)-G(\phi)\|_{L^2(\O_T)}
	\leq c_2\|\wt\phi^\e-\phi\|_{L^2(\O_T)}\to0$ and
	$\|G(\wt\phi^\e)\|_{L^2(\O_T)}\leq c_2\|\wt\phi^\e\|_{L^2(\O_T)}\leq C$.
	
	For the capillary product, let $1\leq r<3/2$ and define $s\in[2,6)$ by
	$\frac1r=\frac12+\frac1s$. Then
	$\|\wt\phi^\e\nabla\wt\mu^\e\|_{L^2(S;L^r)}
	\leq\|\wt\phi^\e\|_{L^\infty(S;L^s)}\|\nabla\wt\mu^\e\|_{L^2(\O_T)}\leq C$.
	For $\boldsymbol\psi\in L^2(S;L^{r'}(\O)^3)$, H\"older's inequality with
	$\frac12+\frac1s+\frac1{r'}=1$ gives
	\[
	\Bigl|\int_{\O_T}\bigl(\wt\phi^\e\nabla\wt\mu^\e-\phi\nabla\mu\bigr)
	\cdot\boldsymbol\psi\Bigr|
	\leq\|\wt\phi^\e-\phi\|_{L^\infty(S;L^s)}\|\nabla\wt\mu^\e\|_{L^2(\O_T)}
	\|\boldsymbol\psi\|_{L^2(S;L^{r'})}
	+\Bigl|\int_{\O_T}(\nabla\wt\mu^\e-\nabla\mu)\cdot\phi\boldsymbol\psi\Bigr| .
	\]
	The first term tends to zero by \eqref{TSL0}. The second tends to zero
	by the weak convergence of $\nabla\wt\mu^\e$, since
	$\|\phi\boldsymbol\psi\|_{L^2(\O_T)}
	\leq\|\phi\|_{L^\infty(S;L^s)}\|\boldsymbol\psi\|_{L^2(S;L^{r'})}$.
	
	\textbf{Step 3: Initial conditions.} 	Let $\zeta\in C^\infty_{c,\div}(\O)^3$. Evaluating \eqref{TSL-pairing} at
	$t=0$, where $\Gu(0)$ is defined since $\Gu\in C_w([0,T];H^\sigma)$, and
	using $\Gu^\e(0)=\Gu_0^\e$, \eqref{ICA01} and
	\eqref{SolenoidalRestrictionConvergence}$_1$, we get
	\[
	(\Gu(0),\zeta)_{L^2(\O)}
	=\lim_{\e\to0}\bigl(\overline{\Gu_0^\e},\overline{\Rc^\e\zeta}\bigr)_{L^2(\O)}
	=(\Gu_0,\zeta)_{L^2(\O)} ,
	\]
	the product of two strongly convergent sequences in $L^2(\O)^3$. Since
	$\Gu(0),\Gu_0\in H^\sigma$ and $C^\infty_{c,\div}(\O)^3$ is dense in
	$H^\sigma$, this gives \eqref{eq:limit-initial-conditions}$_1$.
	
	For the phase field, $\overline{\phi^\e}\to\phi$ in $C([0,T];L^2(\O))$ by
	\eqref{TSL0}, and
	$\|\overline{\phi^\e}(0)-\phi_0\|_{L^2(\O)}=\|\phi_0\|_{L^2(\O_s^\e)}\to0$
	since $\phi_0\in L^2(\O)$ and $|\O_s^\e|\to0$. Evaluating at $t=0$ gives
	\eqref{eq:limit-initial-conditions}$_2$.
\end{proof}

%
%
%

	
\section{The homogenized system}\label{Sec06}

This section contains the proofs of Theorem~\ref{ThHomogenized},
Corollary~\ref{CorGradPhiStrong} and Theorem~\ref{ThmUniq}.

\subsection{Derivation of the homogenized system}
\begin{proof}[Proof of Theorem \ref{ThCritical}]
	Step~1 fixes the test functions, Steps~2 and~3 pass to the limit in
	\eqref{weak-momentum1}--\eqref{weak-mu1}, Step~4 identifies
	$\partial_t\Gu$ and extends the identities to the natural test spaces,
	and Step~5 proves \eqref{weak-energy2}. The regularity of the limit
	fields is part of Lemma~\ref{lem:limit-compactness}, and the initial
	conditions are Lemma~\ref{lem:limit-compactness}(iii). Throughout, for a
	field $w^\e$ on $\O_p^\e$ we write
	$\int_{S\X\O_p^\e}w^\e\psi\,\md(x,t)=\int_{\O_T}\overline{w^\e}\psi\,\md(x,t)$
	and use $\overline{w^\e}=\chi_p^\e\wt w^\e$.
	
	\textbf{Step 1: Test functions.}
	For the scalar equations we take
	\begin{equation}\label{ScalarTestFunctions}
		\varphi_i(t,x)=\chi(t)\psi(x),
		\qquad
		\chi\in C_c^\infty(S),
		\qquad
		\psi\in C^\infty(\overline\O),
		\qquad i\in\{2,3\}.
	\end{equation}
	Since the boundary conditions for $\phi^\e$ and $\mu^\e$ on
	$\partial\O_p^\e$ are homogeneous Neumann conditions, hence natural, the
	restrictions $\varphi_i|_{\O_p^\e}$ belong to $L^2(S;H^1(\O_p^\e))$ and
	are admissible in \eqref{weak-phi1} and \eqref{weak-mu1} without any
	$\e$-dependent correction.
	
	For the momentum equation we take
	\begin{equation}\label{VelocityTestFunctions}
		\bvp_1(t,x)=\chi(t)\zeta(x),
		\qquad
		\chi\in C_c^\infty(S),
		\qquad
		\zeta\in C^\infty_{c,\div}(\O)^3,
		\qquad
		\bvp_1^\e:=\chi(t)\,\Rc^\e\zeta ,
	\end{equation}
	with $\Rc^\e\zeta\in\GH^1_{\div}(\O_p^\e)$ the solenoidal restriction of
	Lemma~\ref{lem:solenoidal-restriction}, so that
	$\bvp_1^\e\in L^\infty\bigl(S;\GH^1_{\div}(\O_p^\e)\bigr)$ is admissible
	in \eqref{weak-momentum1}. By
	\eqref{SolenoidalRestrictionConvergence}$_1$ and the Sobolev embedding
	$H^1(\O)\hookrightarrow L^6(\O)$,
	\begin{equation}\label{VelocityTestConvergence}
		\overline{\Rc^\e\zeta}\to\zeta
		\qquad\text{strongly in }H^1(\O)^3,
		\text{ hence in }L^s(\O)^3\text{ for }2\le s\le6 .
	\end{equation}
	Moreover, $\overline{\Rc^\e\zeta}$ vanishes on a neighbourhood of every
	obstacle, so that for every $\e$
	\begin{equation}\label{NoHoleErrorMomentum}
		\int_{S\X\O_s^\e}(\,\cdot\,)\cdot\overline{\Rc^\e\zeta}\,\md(x,t)=0 ,
		\qquad
		\int_{S\X\O_s^\e}(\,\cdot\,):D\bigl(\overline{\Rc^\e\zeta}\bigr)\,\md(x,t)=0 .
	\end{equation}
	
	\textbf{Step 2: Phase and chemical-potential equations.}
	We use $\varphi_2$ in \eqref{weak-phi1} and $\varphi_3$ in
	\eqref{weak-mu1}. Integrating the transport term by parts as in Step~3
	of the proof of Lemma~\ref{Est01}, the boundary terms vanishing because
	$\Gu^\e=0$ on $\partial\O$ and $\Gu^\e\cdot\Gn^\e=0$ on $\Gamma_s^\e$,
	we obtain
	\[
	\int_{S\X\O^\e_p}
	(\Gu^\e\cdot\nabla\phi^\e)\varphi_2\,\md(x,t)
	=
	-\int_{\O_T}
	\chi_p^\e\,\wt\phi^\e\wt\Gu^\e\cdot\nabla\varphi_2\,\md(x,t).
	\]
	By \eqref{TSL02}$_1$ and Lemma~\ref{lem:hole-error}(iii),
	$\chi_p^\e\wt\phi^\e\wt\Gu^\e\to\phi\Gu$ strongly in $L^2(\O_T)^3$, hence
	\begin{equation}\label{PhaseConvectionLimit}
		\int_{S\X\O^\e_p}
		(\Gu^\e\cdot\nabla\phi^\e)\varphi_2\,\md(x,t)
		\longrightarrow
		-\int_{\O_T}\phi\,\Gu\cdot\nabla\varphi_2\,\md(x,t).
	\end{equation}
	For the remaining terms, the zero extensions $\overline{\nabla\mu^\e}$,
	$\overline{\nabla\phi^\e}$ and $\overline{\mu^\e}$ converge weakly in
	$L^2(\O_T)$ by \eqref{TSL0}. Moreover,
	$\chi_p^\e G(\wt\phi^\e)\to G(\phi)$ strongly and
	$\chi_p^\e F'(\wt\phi^\e)\rightharpoonup F'(\phi)$ weakly in
	$L^2(\O_T)$, by \eqref{TSL02}$_{4,3}$ and
	Lemma~\ref{lem:hole-error}(iii),(ii). Since $\varphi_2$, $\varphi_3$,
	$\nabla\varphi_2$ and $\nabla\varphi_3$ belong to $L^2(\O_T)$, we obtain
	\begin{equation}\label{ScalarLinearLimits}
		\begin{aligned}
			\int_{S\X\O^\e_p}\nabla\mu^\e\cdot\nabla\varphi_2\,\md(x,t)
			&\to
			\int_{\O_T}\nabla\mu\cdot\nabla\varphi_2\,\md(x,t),
			&\quad
			\int_{S\X\O^\e_p}\nabla\phi^\e\cdot\nabla\varphi_3\,\md(x,t)
			&\to
			\int_{\O_T}\nabla\phi\cdot\nabla\varphi_3\,\md(x,t),
			\\
			\int_{S\X\O^\e_p}G(\phi^\e)\varphi_2\,\md(x,t)
			&\to
			\int_{\O_T}G(\phi)\varphi_2\,\md(x,t),
			&\quad
			\int_{S\X\O^\e_p}\mu^\e\varphi_3\,\md(x,t)
			&\to
			\int_{\O_T}\mu\varphi_3\,\md(x,t),
			\\
			\int_{S\X\O^\e_p}F'(\phi^\e)\varphi_3\,\md(x,t)
			&\to
			\int_{\O_T}F'(\phi)\varphi_3\,\md(x,t).
		\end{aligned}
	\end{equation}
	
	It remains to treat the time derivative. By \eqref{MainE01} and
	Lemma~\ref{lem:scalar-two-extensions},
	$\|\partial_t\overline{\phi^\e}\|_{L^2(S;H^{-1}(\O))}\le C$, so every
	subsequence has a further subsequence with
	$\partial_t\overline{\phi^\e}\rightharpoonup\Xi$ weakly in
	$L^2(S;H^{-1}(\O))$. For $\eta\in H^1(\O)$ and $\chi\in C_c^\infty(S)$,
	\[
	\int_S\chi(t)\bigl\langle\partial_t\overline{\phi^\e}(t),\eta\bigr\rangle
	\,\md t
	=-\int_S\chi'(t)\bigl(\overline{\phi^\e}(t),\eta\bigr)_{L^2(\O)}\,\md t ,
	\]
	and passing to the limit with \eqref{TSL0}$_2$ gives
	$\int_S\chi\langle\Xi,\eta\rangle\,\md t
	=-\int_S\chi'(\phi,\eta)_{L^2(\O)}\,\md t$, that is,
	$\Xi=\partial_t\phi$. Hence \eqref{LimitTimeRegularity}$_2$ holds, the
	whole sequence converges, and, since
	$\langle\partial_t\overline{\phi^\e},\eta\rangle_{H^{-1}(\O),H^1(\O)}
	=\langle\partial_t\phi^\e,\eta|_{\O_p^\e}\rangle_{H^{-1}(\O_p^\e),H^1(\O_p^\e)}$
	by Lemma~\ref{lem:scalar-two-extensions},
	\begin{equation}\label{PhaseTimeLimit}
		\int_S\langle\partial_t\phi^\e,\varphi_2\rangle\,\md t
		=
		\int_S\langle\partial_t\overline{\phi^\e},\varphi_2\rangle\,\md t
		\longrightarrow
		\int_S\langle\partial_t\phi,\varphi_2\rangle\,\md t .
	\end{equation}
	Combining \eqref{PhaseConvectionLimit}--\eqref{PhaseTimeLimit}, the
	identities \eqref{HomogenizedPhase} and \eqref{HomogenizedPotential} hold
	for all test functions of the form \eqref{ScalarTestFunctions}.
	
	\textbf{Step 3: The momentum equation.}
	We use $\bvp_1^\e$ in \eqref{weak-momentum1}. By
	\eqref{NoHoleErrorMomentum}, every integral over $S\X\O_p^\e$ in
	\eqref{weak-momentum1} equals the integral over $\O_T$ in which the
	unknowns are replaced by any of their extensions.
	
	As $\bvp_1^\e$ vanishes for $t$ near $0$ and $T$, integration by parts in
	time gives, using $\overline{\Gu^\e}\to\Gu$ strongly in $L^2(\O_T)^3$ by
	\eqref{TSL0}$_1$ and \eqref{VelocityTestConvergence},
	\begin{equation}\label{CriticalTimeLimit}
		\int_S\langle\partial_t\Gu^\e,\bvp_1^\e\rangle\,\md t
		=
		-\int_{\O_T}\overline{\Gu^\e}\cdot\chi'(t)\,\overline{\Rc^\e\zeta}
		\,\md(x,t)
		\longrightarrow
		-\int_{\O_T}\Gu\cdot\partial_t\bvp_1\,\md(x,t).
	\end{equation}
	
	For the convective term, we write the integral over $\O$ with the zero
	extension of $\Gu^\e$. Using \eqref{eq:trilinear-form} and
	$\Rey_\e\le1$, H\"older's inequality gives, for a.e.\ $t\in S$,
	\[
	\left|\Rey_\e\int_{\O}\overline{\Gu^\e}\otimes\overline{\Gu^\e}:
	\nabla\bvp_1^\e\,\md x\right|
	\le
	\Rey_\e\,\|\overline{\Gu^\e}\|^2_{L^4(\O)}
	\|\nabla\bvp_1^\e\|_{L^2(\O)} ,
	\]
	whence, integrating in time and applying H\"older's inequality with the
	exponents $\tfrac34+\tfrac14=1$,
	\begin{equation}\label{ConvectiveVanishes}
		\left|\Rey_\e\int_{S\X\O_p^\e}
		b^\e(\Gu^\e,\Gu^\e,\bvp_1^\e)\,\md(x,t)\right|
		\le
		\Rey_\e\,
		\|\overline{\Gu^\e}\|^2_{L^{8/3}(S;L^4(\O)^3)}
		\|\nabla\bvp_1^\e\|_{L^4(S;L^2(\O)^{3\times3})} .
	\end{equation}
	Both norms on the right are bounded uniformly in $\e$. For the velocity,
	the interpolation inequality
	$\|v\|_{L^4}\le\|v\|_{L^2}^{1/4}\|v\|_{L^6}^{3/4}$, valid on any domain,
	and the uniform Sobolev embedding \eqref{eq:PW-Sobolev}$_2$ give
	$\|\overline{\Gu^\e}(t)\|_{L^4(\O)}=\|\Gu^\e(t)\|_{L^4(\O_p^\e)}
	\le C\|\Gu^\e(t)\|_{L^2(\O_p^\e)}^{1/4}
	\|\Gu^\e(t)\|_{H^1(\O_p^\e)}^{3/4}$, hence
	\[
	\|\overline{\Gu^\e}\|^{8/3}_{L^{8/3}(S;L^4(\O))}
	\le C\|\Gu^\e\|^{2/3}_{L^\infty(S;L^2(\O_p^\e))}
	\|\Gu^\e\|^{2}_{L^2(S;H^1(\O_p^\e))}\le C
	\]
	by \eqref{MainE01}. For the test function,
	$\nabla\bvp_1^\e=\chi(t)\nabla\overline{\Rc^\e\zeta}$ with
	$\|\nabla\overline{\Rc^\e\zeta}\|_{L^2(\O)}\le C_\zeta$ by
	\eqref{SolenoidalRestrictionConvergence}$_2$, so that
	$\|\nabla\bvp_1^\e\|_{L^4(S;L^2(\O))}\le C_\zeta\|\chi\|_{L^4(S)}$.
	Consequently the right-hand side of \eqref{ConvectiveVanishes} is
	bounded by $C_\zeta\Rey_\e$, and
	\begin{equation}\label{ConvectiveLimit}
		\Rey_\e\int_{S\X\O_p^\e}
		b^\e(\Gu^\e,\Gu^\e,\bvp_1^\e)\,\md(x,t)
		\longrightarrow 0
		\qquad\text{as }\e\to0
	\end{equation}
	by \eqref{ReynoldsAssumption}$_2$.
	
	For the viscous term, $D\wt\Gu^\e\rightharpoonup D\Gu$ weakly in
	$L^2(\O_T)^{3\times3}$ by \eqref{TSL0}$_1$, while
	$D(\chi\,\overline{\Rc^\e\zeta})\to D\bvp_1$ strongly in
	$L^2(\O_T)^{3\times3}$ by \eqref{VelocityTestConvergence}, so that
	\begin{equation}\label{CriticalPeriodicViscosityLimit}
		2\nu\int_{S\X\O_p^\e}D\Gu^\e:D\bvp_1^\e\,\md(x,t)
		=
		2\nu\int_{\O_T}D\wt\Gu^\e:D\bigl(\chi\,\overline{\Rc^\e\zeta}\bigr)
		\,\md(x,t)
		\longrightarrow
		2\nu\int_{\O_T}D\Gu:D\bvp_1\,\md(x,t).
	\end{equation}	
	For the capillary term, fix $r\in(6/5,3/2)$ and let $s:=r'\in(3,6)$. By
	\eqref{TSL02}$_5$, $\wt\phi^\e\nabla\wt\mu^\e\rightharpoonup\phi\nabla\mu$
	weakly in $L^2(S;L^r(\O)^3)$, while
	$\chi\,\overline{\Rc^\e\zeta}\to\bvp_1$ strongly in
	$L^\infty(S;L^s(\O)^3)$ by \eqref{VelocityTestConvergence}, whence
	\begin{equation}\label{CriticalCapillaryLimit}
		\lambda\int_{S\X\O_p^\e}
		\phi^\e\nabla\mu^\e\cdot\bvp_1^\e\,\md(x,t)
		=
		\lambda\int_{\O_T}
		\wt\phi^\e\nabla\wt\mu^\e\cdot\chi\,\overline{\Rc^\e\zeta}\,\md(x,t)
		\longrightarrow
		\lambda\int_{\O_T}\phi\,\nabla\mu\cdot\bvp_1\,\md(x,t).
	\end{equation}
	For the force term, \eqref{ForceConvergence} gives
	$\overline{\th^\e}\rightharpoonup\th_0$ weakly in $L^2(\O_T)^3$, and
	pairing with \eqref{VelocityTestConvergence},
	\begin{equation}\label{CriticalForceLimit}
		\int_{S\X\O_p^\e}\th^\e\cdot\bvp_1^\e\,\md(x,t)
		=
		\int_{\O_T}\overline{\th^\e}\cdot\chi\,\overline{\Rc^\e\zeta}
		\,\md(x,t)
		\longrightarrow
		\int_{\O_T}\th_0\cdot\bvp_1\,\md(x,t).
	\end{equation}
	Collecting \eqref{CriticalTimeLimit}--\eqref{CriticalForceLimit}, we
	obtain
	\begin{equation}\label{MomentumWithoutTimeDerivative}
		-\int_{\O_T}\Gu\cdot\partial_t\bvp_1\,\md(x,t)
		+2\nu\int_{\O_T}D\Gu:D\bvp_1\,\md(x,t)
		+\lambda\int_{\O_T}\phi\,\nabla\mu\cdot\bvp_1\,\md(x,t)
		=\int_{\O_T}\th_0\cdot\bvp_1\,\md(x,t)
	\end{equation}
	for all test functions of the form \eqref{VelocityTestFunctions}.
	
	\textbf{Step 4: Time derivative of $\Gu$ and extension to the natural
		test spaces.}
	For a.e.\ $t\in S$ define $g(t):H^1_{0,\div}(\O)^3\to\R$ by
	\begin{equation}\label{eq:momentum-functional}
		\langle g(t),z\rangle
		:=
		-2\nu\int_\O D\Gu(t):Dz\,\md x
		-\lambda\int_\O\phi(t)\nabla\mu(t)\cdot z\,\md x
		+\int_\O\th_0(t)\cdot z\,\md x ,
		\qquad z\in H^1_{0,\div}(\O)^3 .
	\end{equation}
	Each of the three terms is linear in $z$. For the second we use
	H\"older's inequality with $\tfrac23+\tfrac13=1$, the embedding
	$H^1(\O)\hookrightarrow L^3(\O)$ and
	\[
	\|\phi\nabla\mu\|_{L^2(S;L^{3/2}(\O)^3)}
	\le\|\phi\|_{L^\infty(S;L^6(\O))}\|\nabla\mu\|_{L^2(\O_T)} ,
	\]
	which is finite by Lemma~\ref{lem:limit-compactness}. Hence, for a.e.\
	$t\in S$,
	\[
	\|g(t)\|_{H^{-1}_{0,\div}(\O)^3}
	\le
	C\Bigl(\|D\Gu(t)\|_{L^2(\O)}
	+\|\phi(t)\nabla\mu(t)\|_{L^{3/2}(\O)}
	+\|\th_0(t)\|_{L^2(\O)}\Bigr),
	\]
	and the right-hand side belongs to $L^2(S)$. Therefore
	\begin{equation}\label{eq:g-in-dual}
		g\in L^2\bigl(S;H^{-1}_{0,\div}(\O)^3\bigr).
	\end{equation}
	For $\zeta\in C^\infty_{c,\div}(\O)^3$ and $\chi\in C_c^\infty(S)$, the
	identity \eqref{MomentumWithoutTimeDerivative} reads
	\[
	-\int_S\chi'(t)\bigl(\Gu(t),\zeta\bigr)_{L^2(\O)}\,\md t
	=\int_S\chi(t)\bigl\langle g(t),\zeta\bigr\rangle\,\md t ,
	\]
	that is, the distributional time derivative of $\Gu$, tested against
	divergence-free fields, equals $g$. Since $C^\infty_{c,\div}(\O)^3$ is
	dense in $H^1_{0,\div}(\O)^3$ and $g(t)$ is continuous on that space, the
	identification extends to all $z\in H^1_{0,\div}(\O)^3$. Thus
	$\partial_t\Gu=g$, and \eqref{LimitTimeRegularity}$_1$ follows from
	\eqref{eq:g-in-dual}. In particular, the first term of
	\eqref{MomentumWithoutTimeDerivative} equals
	$\int_S\langle\partial_t\Gu,\bvp_1\rangle\,\md t$, which turns
	\eqref{MomentumWithoutTimeDerivative} into \eqref{HomogenizedMomentum}.
	
	It remains to enlarge the class of test functions. By
	\eqref{eq:momentum-functional} and \eqref{eq:g-in-dual}, every term of
	\eqref{HomogenizedMomentum} is continuous with respect to convergence in
	$L^2(S;H^1_{0,\div}(\O)^3)$. The linear span of the fields
	\eqref{VelocityTestFunctions} is dense there, since
	$C^\infty_{c,\div}(\O)^3$ is dense in $H^1_{0,\div}(\O)^3$ and
	$C_c^\infty(S)\otimes H^1_{0,\div}(\O)^3$ is dense in
	$L^2(S;H^1_{0,\div}(\O)^3)$. Hence \eqref{HomogenizedMomentum} holds for
	every $\bvp_1\in L^2(S;H^1_{0,\div}(\O)^3)$.
	
	Similarly, every term of \eqref{HomogenizedPhase} and
	\eqref{HomogenizedPotential} is continuous on $L^2(S;H^1(\O))$. This
	follows from $\partial_t\phi\in L^2(S;H^{-1}(\O))$,
	$\mu\in L^2(S;H^1(\O))$ and $\phi\in L^\infty(S;H^1(\O))$, together with
	$\phi\Gu\in L^2(\O_T)^3$, $F'(\phi)\in L^\infty(S;L^2(\O))$ and
	$G(\phi)\in L^\infty(S;L^2(\O))$. The last three hold by
	$H^1(\O)\hookrightarrow L^6(\O)$, \eqref{F01} and \eqref{SourceGrowth}.
	Since $C_c^\infty(S)\otimes C^\infty(\overline\O)$ is dense in
	$L^2(S;H^1(\O))$, the identities extend to every
	$\varphi_2,\varphi_3\in L^2(S;H^1(\O))$.
	
	\textbf{Step 5: The energy estimate.}
	Let $\chi\in C_c^\infty(S)$ with $\chi\ge0$, multiply the $\e$-level
	energy inequality \eqref{weak-energy1} by $\chi(t)$ and integrate over
	$S$. By Fubini's theorem,
	\[
	\int_S\chi(t)\int_0^t\!\!\int_{\O_p^\e}|D(\Gu^\e)|^2\,\md(x,s)\,\md t
	=
	\int_S w(s)\int_{\O_p^\e}|D(\Gu^\e)|^2\,\md x\,\md s ,
	\qquad
	w(s):=\int_s^T\chi(t)\,\md t\ \ge 0 ,
	\]
	and similarly for $\nabla\mu^\e$ and for the two time integrals on the
	right-hand side, so that every term is a space--time integral. We pass
	to the limit term by term.
	
	Since $\overline{\Gu^\e}\to\Gu$ strongly in $L^2(\O_T)^3$ by
	\eqref{TSL0}$_1$,
	\[
	\int_S\chi\|\Gu^\e(t)\|_{L^2(\O_p^\e)}^2\,\md t
	=
	\bigl\|\sqrt{\chi}\,\overline{\Gu^\e}\bigr\|^2_{L^2(\O_T)}
	\longrightarrow
	\int_S\chi\|\Gu(t)\|^2_{L^2(\O)}\,\md t .
	\]
	The zero extensions $\overline{D\Gu^\e}$, $\overline{\nabla\phi^\e}$ and
	$\overline{\nabla\mu^\e}$ converge weakly in $L^2(\O_T)$ to $D\Gu$,
	$\nabla\phi$ and $\nabla\mu$ by \eqref{TSL0}. Since $\sqrt\chi$ and
	$\sqrt w$ belong to $L^\infty(S)$ and do not depend on $\e$, weak lower
	semicontinuity of the norm gives
	\[
	\begin{aligned}
		\int_S\chi\|\nabla\phi(t)\|^2_{L^2(\O)}\,\md t
		&\le\liminf_{\e\to0}
		\bigl\|\sqrt\chi\,\overline{\nabla\phi^\e}\bigr\|^2_{L^2(\O_T)},\\
		\int_S w\|D\Gu(s)\|^2_{L^2(\O)}\,\md s
		&\le\liminf_{\e\to0}
		\bigl\|\sqrt w\,\overline{D\Gu^\e}\bigr\|^2_{L^2(\O_T)},
	\end{aligned}
	\]
	and the same holds for $\nabla\mu$. For the double-well term,
	\[
	\chi_p^\e F(\wt\phi^\e)-F(\phi)
	=\chi_p^\e\bigl(F(\wt\phi^\e)-F(\phi)\bigr)-\chi_s^\e F(\phi)
	\longrightarrow0\qquad\text{in }L^1(\O_T),
	\]
	the first term by \eqref{TSL02}$_2$ and the second by absolute
	continuity of the integral, since $F(\phi)\in L^1(\O_T)$ and
	$|S\X\O_s^\e|\to0$. Hence
	\[
	\int_S\chi\int_{\O_p^\e}F(\phi^\e)\,\md x\,\md t
	\longrightarrow
	\int_S\chi\int_\O F(\phi)\,\md x\,\md t .
	\]
	For the force term, \eqref{ForceConvergence} and the strong convergence
	$\overline{\Gu^\e}\to\Gu$ in $L^2(\O_T)^3$ give
	\[
	\int_S w\int_{\O_p^\e}\th^\e\cdot\Gu^\e\,\md x\,\md s
	=
	\int_{\O_T} w\,\overline{\th^\e}\cdot\overline{\Gu^\e}\,\md(x,t)
	\longrightarrow
	\int_S w\int_{\O}\th_0\cdot\Gu\,\md x\,\md s .
	\]
	For the source term, $\chi_p^\e G(\wt\phi^\e)\to G(\phi)$ strongly in
	$L^2(\O_T)$ by \eqref{TSL02}$_4$ and Lemma~\ref{lem:hole-error}(iii),
	and $\wt\mu^\e\rightharpoonup\mu$ weakly in $L^2(\O_T)$ by
	\eqref{TSL0}$_3$, so that
	\[
	\int_S w\int_{\O_p^\e}G(\phi^\e)\mu^\e\,\md x\,\md s
	=
	\int_{\O_T}w\,\chi_p^\e G(\wt\phi^\e)\,\wt\mu^\e\,\md(x,t)
	\longrightarrow
	\int_S w\int_{\O}G(\phi)\mu\,\md x\,\md s .
	\]
	Finally, $\mathcal E^\e(0)\to\mathcal E(0)$. The velocity contribution
	converges by \eqref{ICA01}. The contributions of $\phi_0$ converge by
	absolute continuity of the integral, since $|\nabla\phi_0|^2$ and
	$F(\phi_0)$ belong to $L^1(\O)$, $\phi_0\in H^1(\O)$ being independent of
	$\e$, and $|\O_s^\e|\to0$.
	
	Write the weighted inequality as
	$\int_S\chi\,A^\e\,\md t\le\int_S\chi\,B^\e\,\md t$, where $A^\e$ and
	$B^\e$ collect the left- and right-hand terms, respectively, and let $A$
	and $B$ denote the corresponding limit expressions. Since
	$A,B\in L^1(S)$, the lower semicontinuity of each term in $A^\e$ gives
	$\int_S\chi A\le\liminf_{\e\to0}\int_S\chi A^\e$, while each term in
	$B^\e$ converges, so $\int_S\chi B^\e\to\int_S\chi B$. Therefore
	\[
	\int_S\chi A
	\le
	\liminf_{\e\to0}\int_S\chi A^\e
	\le
	\limsup_{\e\to0}\int_S\chi B^\e
	=
	\int_S\chi B .
	\]
	Thus $\int_S\chi(t)\bigl(A(t)-B(t)\bigr)\,\md t\le0$ for every
	$\chi\in C_c^\infty(S)$ with $\chi\ge0$. Hence $A\le B$ a.e.\ in $S$,
	which is \eqref{weak-energy2}.
\end{proof}

\subsection{Proof of Corollary~\ref{CorGradPhiStrong}}
\begin{proof}[Proof of Corollary~\ref{CorGradPhiStrong}]
	Since $\phi^\e\in L^\infty(S;H^1(\O^\e_p))\subset L^2(S;H^1(\O^\e_p))$, the choice
	$\varphi_3=\phi^\e$ is admissible in \eqref{weak-mu1} and gives
	\begin{equation}\label{GradPhiIdentity}
		\int_{\Qc^\e_p}|\nabla\phi^\e|^2\,\md(x,t)
		=\int_{\Qc^\e_p}\mu^\e\phi^\e\,\md(x,t)
		-\int_{\Qc^\e_p}F'(\phi^\e)\phi^\e\,\md(x,t).
	\end{equation}
	Observe that, using the convergences \eqref{TSL0}$_{2,3}$ and \eqref{TSL02}$_3$, we obtain
	\begin{multline*}
		\int_{\Qc^\e_p}\mu^\e\phi^\e\,\md(x,t)
		-\int_{\Qc^\e_p}F'(\phi^\e)\phi^\e\,\md(x,t)\\
		=\int_{\O_T}\wt\mu^\e\overline{\phi^\e}\,\md(x,t)
		-\int_{\O_T}F'(\wt\phi^\e)\overline{\phi^\e}\,\md(x,t)\to \int_{\O_T}\mu\phi\,\md(x,t)
		-\int_{\O_T}F'(\phi)\phi\,\md(x,t).
	\end{multline*}
%
%
%
	
	The limit fields satisfy \eqref{HomogenizedPotential}, and
	$\phi\in\fC\subset L^2(S;H^1(\O))$ is an admissible test function there. Taking
	$\varphi_3=\phi$ gives
	\[
	\int_{\O_T}\mu\phi\,\md(x,t)-\int_{\O_T}F'(\phi)\phi\,\md(x,t)
	=\int_{\O_T}|\nabla\phi|^2\,\md(x,t),
	\]
	so that \eqref{GradPhiIdentity} yields
	$\|\overline{\nabla\phi^\e}\|_{L^2(\O_T)}\to\|\nabla\phi\|_{L^2(\O_T)}$.
	
	Finally $\overline{\nabla\phi^\e}=\chi_p^\e\nabla\wt\phi^\e$, and
	$\nabla\wt\phi^\e\rightharpoonup\nabla\phi$ weakly in $L^2(\O_T)^3$ by
	\eqref{TSL0}$_4$, so Lemma~\ref{lem:hole-error}(ii) gives
	$\overline{\nabla\phi^\e}\rightharpoonup\nabla\phi$ weakly in $L^2(\O_T)^3$. In a
	Hilbert space weak convergence together with convergence of the norms implies
	norm convergence, which is \eqref{GradPhiStrong}.
\end{proof}

\subsection{Existence of unique weak solution}
\begin{proof}[Proof of Theorem \ref{ThmUniq}]
	Throughout, $C>0$ denotes a generic constant, which may depend on the norms of
	the two solutions in \eqref{LimitBounds01} below, and $\delta>0$ is an
	arbitrary parameter, with $C_\delta$ depending on $\delta$. Since $\delta$ is
	arbitrary, we write each estimate with a single factor $\delta$ in front of
	the dissipative terms; the finitely many applications of Young's inequality
	affect only $C_\delta$. All norms without a domain are taken over $\O$. The
	proof is divided into five steps.
	
	{\bf Step 1. The bootstrap \eqref{ExtraReg01}.}
	By definition, a weak solution has the regularity
	\begin{equation}\label{LimitBounds01}
		\begin{aligned}
			\phi\in L^\infty(S;H^1(\O))\cap H^1(S;H^{-1}(\O)),\qquad
			\mu\in L^2(S;H^1(\O)),\\
			\mathbf u\in L^\infty(S;L^2(\O)^3)\cap L^2(S;H^1_{0,\div}(\O)^3)
			\cap H^1(S;H_{0,\div}^{-1}(\O)^3),
		\end{aligned}
	\end{equation}
	the time regularity of $\mathbf u$ being automatic for any triple
	satisfying \eqref{HomogenizedMomentum}, by the argument of Step~4 of the
	proof of Theorem~\ref{ThHomogenized}. In particular,
	$\phi\in L^\infty(S;L^6(\O))$, and, by the Lions--Magenes lemma for the
	triple $H^1_{0,\div}(\O)^3\subset H^\sigma\subset H^{-1}_{0,\div}(\O)^3$
	\cite[Ch.~III, Lemma~1.2]{temam1977}, $\mathbf u\in C([0,T];H^\sigma)$ and
	$\langle\partial_t\mathbf u,\mathbf u\rangle=\frac12\frac{\md}{\md t}\|\mathbf u\|^2_{L^2}$
	a.e.\ in $S$.
	
	Taking $\varphi_3\in H^1(\O)$ in \eqref{HomogenizedPotential} shows that
	$-\Delta\phi+F'(\phi)=\mu$ holds in $\O$, together with
	$\Gn\cdot\nabla\phi=0$ on $\partial\O$ in the sense of the Neumann trace. By
	the Sobolev embedding $H^1(\O)\hookrightarrow L^6(\O)$ and the growth
	condition \eqref{F01},
	\[
	\|F'(\phi)\|_{L^2}\le C\bigl(1+\|\phi\|^3_{L^6}\bigr)
	\le C\bigl(1+\|\phi\|^3_{H^1}\bigr),
	\]
	so that $F'(\phi)\in L^\infty(S;L^2(\O))$ and the right-hand side
	$\mu-F'(\phi)$ belongs to $L^2(S;L^2(\O))$. Elliptic regularity for the
	Neumann Laplacian on a domain of class $C^{1,1}$ \cite[Theorem~31.27(ii) and Remark~31.28]{ErnGuermond2021} yields
	\begin{equation}\label{H2reg01}
		\phi\in L^2(S;H^2(\O)),
		\qquad\text{hence}\qquad
		\nabla\phi\in L^2(S;L^6(\O)^3)
	\end{equation}
	by $H^2(\O)\hookrightarrow W^{1,6}(\O)$ in dimension three. Consequently,
	using $F''(s)=3s^2-1$,
	\[
	\|\nabla F'(\phi)\|_{L^2}
	=\|F''(\phi)\nabla\phi\|_{L^2}
	\le C\bigl(1+\|\phi\|^2_{L^6}\bigr)\|\nabla\phi\|_{L^6}
	\in L^2(S),
	\]
	so that $\Delta\phi=F'(\phi)-\mu\in L^2(S;H^1(\O))$. Elliptic regularity for
	the Neumann Laplacian on a domain of class $C^{2,1}$ \cite[Theorem~31.29 and Remark~31.28]{ErnGuermond2021} now
	gives \eqref{ExtraReg01}.
	
	{\bf Step 2. The difference and its spatial mean.}
	Let $(\mathbf u_i,\phi_i,\mu_i)$, $i=1,2$, be two weak solutions of the
	homogenized system with the same data, and set
	\[
	\mathbf u:=\mathbf u_1-\mathbf u_2,\qquad
	\phi:=\phi_1-\phi_2,\qquad
	\mu:=\mu_1-\mu_2 .
	\]
	Subtracting \eqref{HomogenizedMomentum}--\eqref{HomogenizedPotential} gives,
	for all admissible test functions,
	\begin{align}
		\langle\partial_t\mathbf u,\bvp_1\rangle
		+2\nu\bigl(D(\mathbf u),D(\bvp_1)\bigr)
		+\lambda\bigl(\phi_1\nabla\mu_1-\phi_2\nabla\mu_2,\bvp_1\bigr)&=0,
		\label{DiffU01}\\
		\langle\partial_t\phi,\varphi_2\rangle
		-\bigl(\phi_1\mathbf u_1-\phi_2\mathbf u_2,\nabla\varphi_2\bigr)
		+\bigl(\nabla\mu,\nabla\varphi_2\bigr)
		+\bigl(G(\phi_1)-G(\phi_2),\varphi_2\bigr)&=0,
		\label{DiffPhi01}\\
		\mu=-\Delta\phi+F'(\phi_1)-F'(\phi_2)&\quad\text{in }\O ,
		\label{DiffMu01}
	\end{align}
	the forcing term dropping out of \eqref{DiffU01} because both solutions are
	driven by the same $\th_0$. Taking $\varphi_2=1$ in \eqref{DiffPhi01}, the
	transport and diffusion terms vanish since $\nabla 1=0$, and we obtain
	\begin{equation}\label{MeanDiff01}
		\frac{\md}{\md t}\langle\phi\rangle
		+\frac{1}{|\O|}\int_{\O}\bigl(G(\phi_1)-G(\phi_2)\bigr)\md x=0,
		\qquad
		\langle\phi\rangle(t):=\frac{1}{|\O|}\int_{\O}\phi(t,x)\,\md x .
	\end{equation}
	Since $G$ is Lipschitz with constant $c_2$ by \eqref{AssG01}, H\"older's
	inequality gives
	\begin{equation}\label{MeanDiff02}
		\Bigl|\frac{\md}{\md t}\langle\phi\rangle\Bigr|
		\le\frac{c_2}{|\O|^{1/2}}\|\phi\|_{L^2},
		\qquad\text{hence}\qquad
		\frac{\md}{\md t}|\langle\phi\rangle|^2
		\le C\|\phi\|_{L^2}\,|\langle\phi\rangle|
		\le C\bigl(\|\phi\|^2_{L^2}+|\langle\phi\rangle|^2\bigr).
	\end{equation}
	Unlike in the case of a conserved phase field, $\langle\phi\rangle$ need not
	vanish, so it must be carried along in the estimate. By the
	Poincar\'e--Wirtinger inequality,
	\begin{equation}\label{PWdiff01}
		\|\phi\|^2_{H^1}\le C\bigl(\|\nabla\phi\|^2_{L^2}+|\langle\phi\rangle|^2\bigr).
	\end{equation}
	Accordingly, we set
	\begin{equation}\label{EnergyDiff01}
		E(t):=\|\mathbf u(t)\|^2_{L^2}
		+\lambda\|\nabla\phi(t)\|^2_{L^2}
		+|\langle\phi\rangle(t)|^2 ,
	\end{equation}
	so that, by \eqref{PWdiff01}, $\|\phi\|^2_{H^1}\le CE$ and
	$\|\mathbf u\|^2_{L^2}\le E$, and the right-hand side of
	\eqref{MeanDiff02} is bounded by $CE$.
	
	{\bf Step 3. The differential inequality.}
	We choose $\bvp_1=\mathbf u$ in \eqref{DiffU01} and
	$\varphi_2=-\lambda\Delta\phi$ in \eqref{DiffPhi01}. The former is
	admissible because $\mathbf u\in L^2(S;H^1_{0,\div}(\O)^3)$ with
	$\partial_t\mathbf u\in L^2(S;H^{-1}_{0,\div}(\O)^3)$, so that
	$\langle\partial_t\mathbf u,\mathbf u\rangle
	=\frac12\frac{\md}{\md t}\|\mathbf u\|^2_{L^2}$ by Step~1; the latter is
	admissible by \eqref{ExtraReg01}, which gives
	$\Delta\phi\in L^2(S;H^1(\O))$.
	
	The identity
	$\langle\partial_t\phi,-\lambda\Delta\phi\rangle
	=\tfrac{\lambda}{2}\tfrac{\md}{\md t}\|\nabla\phi\|^2_{L^2}$ requires a word.
	Let $V:=\{v\in H^3(\O):\Gn\cdot\nabla v=0\text{ on }\partial\O\}$, which is
	dense in $H^1(\O)$, and take $H^1(\O)$ with its inner product as pivot space,
	so that $V\subset H^1(\O)\subset V'$. For $v\in V$, integration by parts gives
	$(v,w)_{H^1}=(v-\Delta v,w)_{L^2}$ for all $w\in H^1(\O)$, so the functional
	$v\mapsto\langle\partial_t\phi,v-\Delta v\rangle_{H^{-1},H^1}$ represents
	$\partial_t\phi$ in $V'$, with
	$\|\partial_t\phi\|_{V'}\le C\|\partial_t\phi\|_{H^{-1}(\O)}$. Since
	$\phi\in L^2(S;V)$ by \eqref{ExtraReg01} and $\partial_t\phi\in L^2(S;V')$,
	the Lions--Magenes lemma for this triple gives
	$\phi\in C([0,T];H^1(\O))$ and
	$\langle\partial_t\phi,\phi-\Delta\phi\rangle
	=\tfrac12\tfrac{\md}{\md t}\|\phi\|^2_{H^1}$, while the triple
	$H^1(\O)\subset L^2(\O)\subset H^{-1}(\O)$ gives
	$\langle\partial_t\phi,\phi\rangle=\tfrac12\tfrac{\md}{\md t}\|\phi\|^2_{L^2}$.
	The claim follows by subtraction.
	
	We now add the two identities. Using
	$\phi_1\nabla\mu_1-\phi_2\nabla\mu_2=\phi\nabla\mu_1+\phi_2\nabla\mu$,
	$\phi_1\mathbf u_1-\phi_2\mathbf u_2=\phi\,\mathbf u_1+\phi_2\,\mathbf u$ and,
	by \eqref{DiffMu01},
	$\nabla\mu=-\nabla\Delta\phi+\nabla\bigl(F'(\phi_1)-F'(\phi_2)\bigr)$, the
	capillary term contributes $\lambda(\phi_2\nabla\Delta\phi,\mathbf u)$ and the
	transport term contributes $-\lambda(\phi_2\mathbf u,\nabla\Delta\phi)$; these
	two terms cancel. We obtain
	\begin{multline}\label{DiffIneq01}
		\frac12\frac{\md}{\md t}
		\Bigl(\|\mathbf u\|^2_{L^2}+\lambda\|\nabla\phi\|^2_{L^2}\Bigr)
		+2\nu\|D(\mathbf u)\|^2_{L^2}
		+\lambda\|\nabla\Delta\phi\|^2_{L^2}\\
		=\ \underbrace{-\lambda\bigl(\phi\nabla\mu_1,\mathbf u\bigr)}_{=:J_1}
		\underbrace{-\lambda\bigl(\phi_2\nabla(F'(\phi_1)-F'(\phi_2)),
			\mathbf u\bigr)}_{=:J_2}
		+\underbrace{\lambda\bigl(\nabla(F'(\phi_1)-F'(\phi_2)),
			\nabla\Delta\phi\bigr)}_{=:J_3}\\
		\underbrace{-\lambda\bigl(\phi\,\mathbf u_1,\nabla\Delta\phi\bigr)}_{=:J_4}
		+\underbrace{\lambda\bigl(G(\phi_1)-G(\phi_2),\Delta\phi\bigr)}_{=:J_5}.
	\end{multline}
	Only the moduli $|J_i|$ enter the estimates below, so the signs play no
	further role.
	
	{\bf Step 4. Estimate of the right-hand side.}
	Throughout we use $H^1(\O)\hookrightarrow L^6(\O)$, the interpolation
	inequality $\|v\|_{L^3}\le C\|v\|^{1/2}_{L^2}\|v\|^{1/2}_{H^1}$, and the
	elliptic estimates
	\begin{equation}\label{InterpDelta01}
		\|\Delta\phi\|_{L^2}
		\le \|\nabla\phi\|^{1/2}_{L^2}\|\nabla\Delta\phi\|^{1/2}_{L^2},
		\qquad
		\|\nabla\phi\|_{L^6}\le C\|\phi\|_{H^2}
		\le C\bigl(\|\Delta\phi\|_{L^2}+\|\phi\|_{L^2}\bigr).
	\end{equation}
	The first estimate follows from
	$\|\Delta\phi\|^2_{L^2}=-(\nabla\phi,\nabla\Delta\phi)$, the boundary term
	vanishing because $\Gn\cdot\nabla\phi_i=0$ on $\partial\O$ for $i=1,2$ and
	hence $\Gn\cdot\nabla\phi=0$. Since $\|\phi_i\|_{L^6}\in L^\infty(S)$ by
	\eqref{LimitBounds01}, powers of $\|\phi_i\|_{L^6}$ are absorbed into $C$.
	
	By H\"older's and Young's inequalities,
	\begin{equation}\label{J1bound}
		|J_1|\le\lambda\|\phi\|_{L^6}\|\nabla\mu_1\|_{L^2}\|\mathbf u\|_{L^3}
		\le\delta\|\mathbf u\|^2_{H^1}
		+C_\delta\|\nabla\mu_1\|^2_{L^2}\,\|\phi\|^2_{H^1}.
	\end{equation}
	
	Writing
	$\nabla\bigl(F'(\phi_1)-F'(\phi_2)\bigr)
	=F''(\phi_1)\nabla\phi+\bigl(F''(\phi_1)-F''(\phi_2)\bigr)\nabla\phi_2$
	and using $F''(\phi_1)-F''(\phi_2)=3(\phi_1+\phi_2)\phi$, we obtain
	\begin{equation}\label{FprimeDiff01}
		\bigl\|\nabla(F'(\phi_1)-F'(\phi_2))\bigr\|_{L^2}
		\le C\|\nabla\phi\|_{L^6}
		+C\|\phi_2\|_{H^2}\|\phi\|_{H^1},
	\end{equation}
	where we used $\|\nabla\phi_2\|_{L^6}\le C\|\phi_2\|_{H^2}\in L^2(S)$ by
	\eqref{H2reg01}.
	
	The term $J_2$ is the most delicate one: through \eqref{FprimeDiff01} it
	carries both a factor that must be absorbed into
	$\|\nabla\Delta\phi\|^2_{L^2}$ and one that must be absorbed into
	$\|D(\mathbf u)\|^2_{L^2}$, so a single application of Young's inequality
	does not suffice. By \eqref{FprimeDiff01} and \eqref{InterpDelta01},
	\[
	|J_2|\le C\bigl(\|\Delta\phi\|_{L^2}+\|\phi\|_{L^2}
	+\|\phi_2\|_{H^2}\|\phi\|_{H^1}\bigr)\|\mathbf u\|_{L^3}.
	\]
	In the leading contribution we use both interpolation inequalities, which
	give
	$C\|\nabla\phi\|^{1/2}_{L^2}\|\nabla\Delta\phi\|^{1/2}_{L^2}
	\|\mathbf u\|^{1/2}_{L^2}\|\mathbf u\|^{1/2}_{H^1}$. Young's inequality,
	applied first to the product of the two factors carrying derivatives and the
	two $L^2$ factors, and then to each pair, bounds this by
	\[
	\delta\|\nabla\Delta\phi\|^{2}_{L^2}
	+\delta\|\mathbf u\|^{2}_{H^1}
	+C_\delta\bigl(\|\nabla\phi\|^{2}_{L^2}+\|\mathbf u\|^{2}_{L^2}\bigr).
	\]
	The two remaining contributions contain no factor
	$\|\nabla\Delta\phi\|_{L^2}$ and are bounded, using
	$\|\mathbf u\|_{L^3}\le C\|\mathbf u\|_{H^1}$, by
	$\delta\|\mathbf u\|^2_{H^1}
	+C_\delta\bigl(1+\|\phi_2\|^2_{H^2}\bigr)\|\phi\|^2_{H^1}$. Altogether,
	\begin{equation}\label{J2bound}
		|J_2|\le
		\delta\|\nabla\Delta\phi\|^2_{L^2}
		+\delta\|\mathbf u\|^2_{H^1}
		+C_\delta\bigl(1+\|\phi_2\|^2_{H^2}\bigr)E .
	\end{equation}
	
	For $J_3$, we pair \eqref{FprimeDiff01} with $\|\nabla\Delta\phi\|_{L^2}$.
	Through \eqref{InterpDelta01}, the first summand produces a term
	$C\|\nabla\phi\|^{1/2}_{L^2}\|\nabla\Delta\phi\|^{3/2}_{L^2}$, to which
	Young's inequality with the conjugate exponents $4$ and $4/3$ applies, and a
	term $C\|\phi\|_{L^2}\|\nabla\Delta\phi\|_{L^2}$. Hence
	\begin{equation}\label{J3bound}
		|J_3|\le\delta\|\nabla\Delta\phi\|^2_{L^2}
		+C_\delta\bigl(1+\|\phi_2\|^2_{H^2}\bigr)\|\phi\|^2_{H^1}.
	\end{equation}
	
	By H\"older's and Young's inequalities,
	\begin{equation}\label{J4bound}
		|J_4|\le\lambda\|\phi\|_{L^6}\|\mathbf u_1\|_{L^3}
		\|\nabla\Delta\phi\|_{L^2}
		\le\delta\|\nabla\Delta\phi\|^2_{L^2}
		+C_\delta\|\mathbf u_1\|^2_{L^3}\,\|\phi\|^2_{H^1},
	\end{equation}
	with $\|\mathbf u_1\|^2_{L^3}\le C\|\mathbf u_1\|_{L^2}\|\mathbf u_1\|_{H^1}$.
	
	Since $G$ is Lipschitz with constant $c_2$, \eqref{InterpDelta01} and
	Young's inequality give
	\begin{equation}\label{J5bound}
		|J_5|\le c_2\lambda\|\phi\|_{L^2}\|\Delta\phi\|_{L^2}
		\le\delta\|\nabla\Delta\phi\|^2_{L^2}+C_\delta\|\phi\|^2_{H^1}.
	\end{equation}
	
	Since $\mathbf u(t)\in H^1_{0,\div}(\O)^3$, Korn's and Poincar\'e's
	inequalities on $\O$ give
	$\|\mathbf u\|^2_{H^1}\le C_K\|D(\mathbf u)\|^2_{L^2}$. Hence all
	contributions $\delta\|\mathbf u\|^2_{H^1}$ in \eqref{J1bound} and
	\eqref{J2bound} are absorbed into $2\nu\|D(\mathbf u)\|^2_{L^2}$, and all
	contributions $\delta\|\nabla\Delta\phi\|^2_{L^2}$ into
	$\lambda\|\nabla\Delta\phi\|^2_{L^2}$, provided $\delta>0$ is chosen small
	enough. No smallness assumption relating $\nu$ and $\lambda$ is required:
	the absorption uses the free parameter $\delta$ alone.
	
	{\bf Step 5. Conclusion by Gr\"onwall's inequality.}
	Collecting \eqref{J1bound}--\eqref{J5bound}, adding \eqref{MeanDiff02},
	and using \eqref{PWdiff01} to replace $\|\phi\|^2_{H^1}$ by
	$C\bigl(\|\nabla\phi\|^2_{L^2}+|\langle\phi\rangle|^2\bigr)$, we arrive at
	\begin{equation}\label{Gronwall01}
		\frac{\md}{\md t}E(t)
		+2\nu\|D(\mathbf u)\|^2_{L^2}
		+\lambda\|\nabla\Delta\phi\|^2_{L^2}
		\ \le\ \kappa(t)\,E(t)
		\qquad\text{for a.e. }t\in S,
	\end{equation}
	where
	\[
	\kappa(t):=C\Bigl(
	1
	+\|\nabla\mu_1(t)\|^2_{L^2}
	+\|\phi_2(t)\|^2_{H^2}
	+\|\mathbf u_1(t)\|^2_{L^3}
	\Bigr)\in L^1(S)
	\]
	by \eqref{LimitBounds01} and \eqref{H2reg01}. By Steps~1 and~3,
	$\mathbf u\in C([0,T];H^\sigma)$ and $\phi\in C([0,T];H^1(\O))$, so $E$ is
	continuous on $[0,T]$; since the two solutions share the same initial data,
	$E(0)=0$. Discarding the two nonnegative dissipation terms on the left of
	\eqref{Gronwall01} and applying Gr\"onwall's inequality to
	$\frac{\md}{\md t}E\le\kappa E$ yields $E\equiv0$ on $S$. Hence
	$\mathbf u\equiv0$ and, by \eqref{PWdiff01}, $\phi\equiv0$; substituting
	into \eqref{DiffMu01} gives $\mu\equiv0$. Therefore
	\[
	(\mathbf u_1,\phi_1,\mu_1)=(\mathbf u_2,\phi_2,\mu_2),
	\]
	which proves uniqueness.
\end{proof}

\section{Proofs of the results on the control problems}\label{sec:control}

This section contains the proofs of Lemma~\ref{LemBoundControl},
Theorem~\ref{ThmExistOPe}, Theorem~\ref{ThmGamma} and
Theorem~\ref{ThmOCPHomo}.

\subsection{Existence of microscopic optimal pairs}

\begin{proof}[Proof of Lemma \ref{LemBoundControl}]
	Take the admissible competitor $\th^\e=0$ and let
	$(\Gu^\e_\ast,\phi^\e_\ast,\mu^\e_\ast)\in\Sc^\e(0)$, which is nonempty by
	Remark~\ref{rem:existence}.
	Since $\|0\|_{L^2(\Qc^\e_p)}=0$, Lemma~\ref{Est01} applies with vanishing force
	and gives
	\[
	\|\Gu^\e_\ast\|_{L^2(\Qc^\e_p)}
	+\|\phi^\e_\ast\|_{L^2(\Qc^\e_p)}
	+\|\nabla\phi^\e_\ast\|_{L^2(\Qc^\e_p)}\le C
	\]
	with $C$ independent of $\e$, the first summand by \eqref{MainE01}$_2$ and the
	second and third by \eqref{MainE01}$_4$, in both cases using the boundedness of
	$S$. As $\O^\e_p\subset\O$, we also have
	$\|\mathbf u_d\|_{L^2(\Qc^\e_p)}\le\|\mathbf u_d\|_{L^2(\O_T)}$ and
	$\|\phi_d\|_{L^2(S;H^1(\O^\e_p))}\le\|\phi_d\|_{L^2(S;H^1(\O))}$. Hence, using
	$\|a-b\|^2\le2\|a\|^2+2\|b\|^2$ on each tracking term and recalling
	\eqref{H1PoreNorm},
\begin{multline*}
	\m_\e
	\le \GJ^\e(0,\Gu^\e_\ast,\phi^\e_\ast)
	\le \|\Gu^\e_\ast\|^2_{L^2(\Qc^\e_p)}+\|\mathbf u_d\|^2_{L^2(\O_T)}
	+\|\phi^\e_\ast\|^2_{L^2(\Qc^\e_p)}
	+\|\nabla\phi^\e_\ast\|^2_{L^2(\Qc^\e_p)}
	+\|\phi_d\|^2_{L^2(S;H^1(\O))}\\
	\le C^2+\|\mathbf u_d\|^2_{L^2(\O_T)}+\|\phi_d\|^2_{L^2(S;H^1(\O))}
	=:C_0 ,
\end{multline*}
which is \eqref{InfBound01}, with $C_0$ independent of $\e$. If now
	$\GJ^\e(\th^\e,\Gu^\e,\phi^\e)\le C_0+1$, then, all three terms of
	\eqref{Cost01} being nonnegative,
	$\tfrac\beta2\|\th^\e\|^2_{L^2(\Qc^\e_p)}\le C_0+1$, and
	\eqref{ControlBound01} follows with $K:=\bigl(2(C_0+1)/\beta\bigr)^{1/2}$, which
	is independent of $\e$ because $\beta>0$ is fixed.
\end{proof}

\begin{proof}[Proof of Theorem \ref{ThmExistOPe}]
	Fix $\e\in(0,\e_0)$; throughout this proof $\e$ is frozen and all constants may
	depend on it. Let $\bigl(\th_k,(\Gu_k,\phi_k,\mu_k)\bigr)_{k\in\N}\subset\Ac^\e$
	be a minimizing sequence, so that $\GJ^\e(\th_k,\Gu_k,\phi_k)\to\m_\e$; discarding finitely many terms,
	$\GJ^\e(\th_k,\Gu_k,\phi_k)\le C_0+1$ for every $k$, and $\|\th_k\|_{L^2(\Qc^\e_p)}\le K$ by Lemma~\ref{LemBoundControl}.
	Extracting a subsequence, not relabeled,
	\begin{equation}\label{ControlWeakLimit}
		\th_k\rightharpoonup\th
		\qquad\text{weakly in }L^2(S\X\O^\e_p)^3 .
	\end{equation}
	
	\textbf{Step 1: Compactness of the states.}
	Since $\|\th_k\|_{L^2(\Qc^\e_p)}\le K$ uniformly in $k$, Lemma~\ref{Est01} gives
	\begin{equation}\label{StatesBounded}
		\begin{aligned}
			\|\Gu_k\|_{L^\infty(S;L^2(\O^\e_p))}
			+\|\Gu_k\|_{L^2(S;H^1(\O^\e_p))}
			+\|\partial_t\Gu_k\|_{L^{4/3}(S;\GH^{-1}_{\div}(\O^\e_p))}&\le C,\\
			\|\phi_k\|_{L^\infty(S;H^1(\O^\e_p))}
			+\|\partial_t\phi_k\|_{L^2(S;H^{-1}(\O^\e_p))}
			+\|\mu_k\|_{L^2(S;H^1(\O^\e_p))}&\le C .
		\end{aligned}
	\end{equation}
	The domain $\O^\e_p$ being fixed and Lipschitz, the embedding
	$H^1(\O^\e_p)\Subset L^2(\O^\e_p)$ is compact. 
Applied to the velocity with the triple
$\GH^1_\div(\O^\e_p)\Subset H^\sigma_\e\hookrightarrow\GH^{-1}_\div(\O^\e_p)$,
where the second embedding is injective because $\GH^1_\div(\O^\e_p)$ is dense
in $H^\sigma_\e$, and with the exponent $4/3>1$, the Aubin--Lions--Simon theorem
\cite{simon1987} yields strong convergence in $L^2(\Qc^\e_p)^3$;
	applied to the phase field, the bounds
	\eqref{StatesBounded}$_2$ give equicontinuity in $H^{-1}(\O^\e_p)$ and hence, by
	interpolation with the uniform $H^1$ bound, strong convergence in
	$C([0,T];L^2(\O^\e_p))$. After a further extraction,
	\begin{equation}\label{StatesCompact}
		\begin{aligned}
			\Gu_k&\to\Gu \quad\text{strongly in }L^2(S;L^s(\O^\e_p)^3),\ 2\le s<6,\quad
			&&\Gu_k\rightharpoonup\Gu \quad\text{weakly in }L^2(S;\GH^1_{\div}(\O^\e_p))\\
			\phi_k&\to\phi \quad\text{strongly in }C([0,T];L^s(\O^\e_p)),\ 2\le s<6,\quad
			&&\nabla\phi_k\rightharpoonup\nabla\phi\quad\text{weakly in }L^2(\Qc^\e_p)^3,\\
			\mu_k&\rightharpoonup\mu\quad\text{weakly in }L^2(S;H^1(\O^\e_p)),
		\end{aligned}
	\end{equation}
the two strong convergences being upgraded from $s=2$ by interpolation with the
uniform bounds in $L^2(S;L^6(\O^\e_p))$ for the velocity and in
$L^\infty(S;L^6(\O^\e_p))$ for the phase field, both following from
\eqref{StatesBounded} and the Sobolev embedding
$H^1(\O^\e_p)\hookrightarrow L^6(\O^\e_p)$, and \eqref{StatesCompact}$_4$
	following from \eqref{StatesBounded}$_2$.
	
	\textbf{Step 2: The limit triple is a weak solution for $\th$.}
	We pass to the limit $k\to\infty$ in \eqref{weak-momentum1}--\eqref{weak-mu1}
	for fixed test functions $\bvp_1\in L^4(S;\GH^1_{\div}(\O^\e_p))$ and
	$\varphi_2,\varphi_3\in L^2(S;H^1(\O^\e_p))$. 
The linear terms converge by \eqref{StatesCompact}, the time-derivative terms
by the weak convergences $\partial_t\Gu_k\rightharpoonup\partial_t\Gu$ in
$L^{4/3}(S;\GH^{-1}_\div(\O^\e_p))$ and $\partial_t\phi_k\rightharpoonup\partial_t\phi$
in $L^2(S;H^{-1}(\O^\e_p))$, which hold after a further extraction by
\eqref{StatesBounded}.
	The capillary term $\phi_k\nabla\mu_k$, the transport
	term $\phi_k\Gu_k$, the source term $G(\phi_k)$ and the double-well term
	$F'(\phi_k)$ are treated exactly as in Step~2 of Theorem~\ref{ThHomogenized},
	with \eqref{StatesCompact} in place of Lemma~\ref{lem:limit-compactness} and
	with no hole error to control, the domain being fixed here. Only the convective
	term differs, $\Rey_\e$ being a fixed number at this level: the sequence
	$\Gu_k\otimes\Gu_k$ is bounded in $L^{4/3}(S;L^2(\O^\e_p)^{3\X3})$, by
	\eqref{StatesBounded}$_1$ and the interpolation
	$L^\infty(S;L^2)\cap L^2(S;H^1)\hookrightarrow L^{8/3}(S;L^4)$, and converges
	strongly in $L^1(S;L^2(\O^\e_p)^{3\X3})$ by \eqref{StatesCompact}$_1$ with
	$s=4$; the two together give
	$\Gu_k\otimes\Gu_k\rightharpoonup\Gu\otimes\Gu$ weakly in
	$L^{4/3}(S;L^2(\O^\e_p)^{3\X3})$, and
	$\nabla\bvp_1\in L^4(S;L^2(\O^\e_p)^{3\X3})$ makes the pairing converge.
	
	The force term converges by \eqref{ControlWeakLimit} paired with the fixed
	$\bvp_1$. 
The initial conditions are attained: the embedding
$W^{1,4/3}(S;\GH^{-1}_\div(\O^\e_p))\hookrightarrow C([0,T];\GH^{-1}_\div(\O^\e_p))$
is continuous and $\Gu_k\rightharpoonup\Gu$ weakly in
$W^{1,4/3}(S;\GH^{-1}_\div(\O^\e_p))$, so
$\Gu^\e_0=\Gu_k(0)\rightharpoonup\Gu(0)$ in $\GH^{-1}_\div(\O^\e_p)$; since
$H^\sigma_\e\hookrightarrow\GH^{-1}_\div(\O^\e_p)$ is injective,
$\Gu(0)=\Gu^\e_0$. For the phase field, $\phi_k(0)=\phi^\e_0$ and
$\phi_k\to\phi$ in $C([0,T];L^2(\O^\e_p))$ by \eqref{StatesCompact}$_2$.
	Thus \eqref{weak-momentum1}--\eqref{weak-mu1} hold
	for the limit triple with body force $\th$.
	
	It remains to verify the energy inequality \eqref{weak-energy1}, which is part
	of Definition~\ref{Def01}. Multiplying \eqref{weak-energy1} for
	$(\Gu_k,\phi_k,\mu_k)$ by a weight $\chi\in C^\infty_c(S)$, $\chi\ge0$, and
	proceeding as in Step~5 of Theorem~\ref{ThHomogenized}, the left-hand side is
	handled by weak lower semicontinuity of the $L^2$ norms of $\Gu_k$,
	$\nabla\phi_k$, $D(\Gu_k)$ and $\nabla\mu_k$, and by Fatou's lemma for
	$F(\phi_k)$, using $F\ge-\tfrac34$ from \eqref{F01} and the a.e.\ convergence of
	$\phi_k$. On the right-hand side, $\mathcal E^\e(0)$ is the same for every $k$;
	the force term converges by \eqref{ControlWeakLimit} paired with the strong
	convergence $\Gu_k\to\Gu$ in $L^2(\Qc^\e_p)^3$; and the source term converges by
	the strong convergence of $G(\phi_k)$ paired with the weak convergence of
	$\mu_k$. Taking the $\liminf$ and letting $\chi$ vary gives \eqref{weak-energy1}
	for the limit triple. Hence $(\Gu,\phi,\mu)\in\Sc^\e(\th)$, that is,
	$\bigl(\th,(\Gu,\phi,\mu)\bigr)\in\Ac^\e$.
	
	\textbf{Step 3: Lower semicontinuity of the cost.}
	We split the phase-field term according to \eqref{H1PoreNorm}. The velocity
	term and the $L^2$ part of the phase-field term converge, by the strong
	convergences \eqref{StatesCompact}$_{1,2}$. 
The gradient part converges only weakly, by \eqref{StatesCompact}$_4$, and is
therefore weakly lower semicontinuous.
	The control term is weakly lower
	semicontinuous with respect to \eqref{ControlWeakLimit}. Adding two convergent
	terms to two lower semicontinuous ones gives
	\[
	\GJ^\e(\th,\Gu,\phi)
	\le\liminf_{k\to\infty}\GJ^\e(\th_k,\Gu_k,\phi_k)
	=\m_\e ,
	\]
	and since $\bigl(\th,(\Gu,\phi,\mu)\bigr)$ is admissible the inequality is an
	equality. Thus the limit pair is optimal and $\m_\e$ is attained.
\end{proof}

\subsection{Proofs of Theorems~\ref{ThmGamma} and~\ref{ThmOCPHomo}}
\label{ssec:gamma}

First observe that, $\|\chi^\e_pw_d-w_d\|_{L^2(\O_T)}=\|w_d\|_{L^2(S\X\O^\e_s)}\to0$
	by absolute continuity of the integral, since $|\O^\e_s|\to0$. Hence, by
	\eqref{eq:ideC},
	\begin{equation}\label{eq:tracking-conv}
		\overline{w^\e}\to w\ \text{ strongly in }L^2(\O_T)
		\quad\Longrightarrow\quad
		\|w^\e-w_d\|_{L^2(\Qc^\e_p)}\longrightarrow\|w-w_d\|_{L^2(\O_T)} ,
	\end{equation}
	and the same holds componentwise for vector-valued fields.

\begin{proof}[Proof of Theorem \ref{ThmGamma}]
	\textbf{Step 1: Equicoercivity.}
	Let $\Jc^\e(\Lambda^\e)\le M$. By \eqref{GammaFunctional},
	$\Lambda^\e=(\th^\e,\Gu^\e,\phi^\e,\mu^\e)\in\Ac^\e$, and the last term
	of \eqref{GammaCost} gives
	$\|\overline{\th^\e}\|_{L^2(\O_T)}=\|\th^\e\|_{L^2(\Qc^\e_p)}\le K
	:=(2M/\beta)^{1/2}$. We show that every $\{\Lambda^\e\}\subset\Ac^\e$
	with controls bounded by $K$ has the property stated in (i); this is
	the only compactness statement used below.
	
	Fix a subsequence. A further subsequence satisfies
	$\overline{\th^\e}\rightharpoonup\th$ weakly in $L^2(\O_T)^3$, which is
	\eqref{OCConv1} and, together with the bound, hypothesis
	\eqref{ForceConvergence} with $\th_0=\th$. By
	Lemma~\ref{lem:limit-compactness}, a further subsequence and some
	$(\Gu,\phi,\mu)\in\fU\X\fC\X\fW$ satisfy \eqref{TSL0}, which contains
	\eqref{OCConv2}, \eqref{OCConv3}$_1$ and \eqref{OCConv4}. By
	Theorem~\ref{ThHomogenized} and Lemma~\ref{lem:limit-compactness}(iii),
	$(\Gu,\phi,\mu)$ is a weak solution of
	\eqref{HomogenizedMomentum}--\eqref{HomogenizedPotential} with body
	force $\th$ and initial data \eqref{eq:limit-initial-conditions}, that
	is, $\Xi:=(\th,\Gu,\phi,\mu)\in\Ac$. Finally,
	Corollary~\ref{CorGradPhiStrong} gives \eqref{OCConv3}$_2$. Hence
	$\Lambda^\e\ocarrow\Xi$, which proves (i).
	
	\textbf{Step 2: The $\liminf$ inequality.}
	Extract a subsequence, not relabeled, along which
	$\Jc^\e(\Lambda^\e)$ converges to the $\liminf$ and
	$\Jc^\e(\Lambda^\e)\le M<\infty$. By Step~1, a further subsequence
	satisfies $\Lambda^\e\ocarrow\Xi$ for some $\Xi\in\Ac$. The limits in
	Definition~\ref{DefOCConv} are unique, so $\Lambda=\Xi\in\Ac$.
	
	Write $\Lambda=(\th,\Gu,\phi,\mu)$. By \eqref{eq:tracking-conv}, the
	three tracking terms of \eqref{GammaCost} converge to those of
	\eqref{Cost02}: the velocity term by \eqref{OCConv2}$_1$, the phase term
	by \eqref{OCConv3}$_1$, and the gradient term, componentwise, by
	\eqref{OCConv3}$_2$. For the control term, weak lower semicontinuity of
	the norm and \eqref{OCConv1} give
	$\|\th\|_{L^2(\O_T)}\le\liminf_{\e\to0}\|\overline{\th^\e}\|_{L^2(\O_T)}$.
	Adding the three convergent terms to the lower semicontinuous one gives
	(ii).
	
	\textbf{Step 3: The recovery sequence.}
	Let $\Xi=(\th,\Gu,\phi,\mu)\in\Ac$ and set $\th^\e:=\th|_{\O^\e_p}$, so
	that $\overline{\th^\e}=\chi^\e_p\th$ and
	\begin{equation}\label{RestrictedForceConvergence}
		\|\th^\e\|_{L^2(\Qc^\e_p)}\le\|\th\|_{L^2(\O_T)},
		\qquad
		\overline{\th^\e}\to\th\quad\text{strongly in }L^2(\O_T)^3 ,
	\end{equation}
	the convergence by absolute continuity of the integral, since
	$|\O^\e_s|\to0$. For every $\e\in(0,\e_0)$ choose a weak solution
	$(\Gu^\e,\phi^\e,\mu^\e)$ of \eqref{MainP01} with body force $\th^\e$,
	which exists by Remark~\ref{rem:existence}, and set
	$\Xi^\e:=(\th^\e,\Gu^\e,\phi^\e,\mu^\e)\in\Ac^\e$.
	
	By Step~1 with $K=\|\th\|_{L^2(\O_T)}$, every subsequence has a further
	subsequence with $\Xi^\e\ocarrow\Xi'=(\th',\Gu',\phi',\mu')\in\Ac$,
	and $\th'=\th$ by \eqref{RestrictedForceConvergence}. By Theorem \ref{ThCritical} and \ref{ThmUniq}, $\Ac$ is the graph of $\Sc$, so
	$(\Gu',\phi',\mu')=\Sc(\th)=(\Gu,\phi,\mu)$, that is, $\Xi'=\Xi$.  Since
	every subsequence has a further subsequence converging to the same
	limit $\Xi$, and each convergence in Definition~\ref{DefOCConv} is weak
	or strong convergence in a normed space, the whole family satisfies
	$\Xi^\e\ocarrow\Xi$.
	
	The three tracking terms of $\Jc^\e(\Xi^\e)$ converge by
	\eqref{eq:tracking-conv}, as in Step~2, and the control term converges
	by \eqref{RestrictedForceConvergence}, since
	$\|\overline{\th^\e}\|_{L^2(\O_T)}=\|\chi^\e_p\th\|_{L^2(\O_T)}
	\to\|\th\|_{L^2(\O_T)}$. Hence $\Jc^\e(\Xi^\e)\to\Jc(\Xi)$, which is
	(iii).
\end{proof}

\begin{proof}[Proof of Theorem \ref{ThmOCPHomo}]
	By Theorem~\ref{ThmExistOPe} and Lemma~\ref{LemBoundControl},
	$\Jc^\e(\hat\Xi_\e)=\m_\e\le C_0$ for every $\e\in(0,\e_0)$, so (i)
	follows from Theorem~\ref{ThmGamma}(i) with $M=C_0$.
	
	For $\Xi'\in\Ac$, let $\{\Xi'_\e\}$ be the recovery sequence of
	Theorem~\ref{ThmGamma}(iii). Minimality of $\m_\e$ gives
	$\limsup_{\e\to0}\m_\e\le\lim_{\e\to0}\Jc^\e(\Xi'_\e)=\Jc(\Xi')$, and
	taking the infimum over $\Xi'\in\Ac$,
	\[
	\limsup_{\e\to0}\m_\e\le\m .
	\]
	Conversely, let $\hat\Xi_\e\ocarrow\Xi$ along a subsequence, not
	relabeled. By Theorem~\ref{ThmGamma}(ii), $\Xi\in\Ac$ and
	\begin{equation}\label{ClusterBound}
		\m\le\Jc(\Xi)\le\liminf_{\e\to0}\m_\e
	\end{equation}
	along that subsequence. Applying this to a subsequence, provided by
	(i), of a subsequence realizing $\liminf_{\e\to0}\m_\e$, we obtain
	$\m\le\liminf_{\e\to0}\m_\e$ for the whole family. Hence $\m_\e\to\m$,
	which is (ii). For an arbitrary cluster point, \eqref{ClusterBound} and
	(ii) give $\Jc(\Xi)=\m$, which is (iii).
	
	Along the subsequence, the three tracking terms of \eqref{GammaCost}
	converge to those of \eqref{Cost02} by \eqref{eq:tracking-conv}, as in
	Step~2 of the proof of Theorem~\ref{ThmGamma}, and
	$\Jc^\e(\hat\Xi_\e)=\m_\e\to\m=\Jc(\Xi)$ by (ii) and (iii). Subtracting,
	$\|\overline{\hat\th^\e}\|_{L^2(\O_T)}\to\|\th\|_{L^2(\O_T)}$. Together
	with $\overline{\hat\th^\e}\rightharpoonup\th$ weakly in
	$L^2(\O_T)^3$ by \eqref{OCConv1}, this gives
	\eqref{StrongControlConvergence}, since in a Hilbert space weak
	convergence and convergence of the norms imply strong convergence.
\end{proof}

		\section*{Acknowledgements}
Haradhan Dutta gratefully acknowledges the PMRF (ID-2402788) scheme and IIT Kharagpur for funding his PhD position.

		\addcontentsline{toc}{section}{References}
		{\small
			\bibliographystyle{plain}
			\bibliography{ref_new}
		}

	\end{document}